\documentclass[11pt]{article}
\usepackage{amsmath,amsthm,amssymb,dsfont,mathrsfs}
\usepackage[dvipsnames]{xcolor}
\usepackage{enumerate}
\usepackage[english]{babel}
\usepackage{graphicx}
\usepackage{cite}
\usepackage{comment}
\usepackage{oands}
\usepackage{tikz}
\usepackage{changepage}
\usepackage{bbm}
\usepackage{mathtools}
\usepackage[margin=1in]{geometry}
\usepackage[pagewise,mathlines]{lineno}
\usepackage{appendix}
\usepackage{stmaryrd} 
\usepackage{multicol}
\usepackage{microtype}
\usepackage{subfigure}
\usepackage{upgreek}

\definecolor{royalblue(web)}{rgb}{0.25, 0.41, 0.88}
\usepackage[pdftitle={},
pdfauthor={},
colorlinks=true,linkcolor=MidnightBlue,urlcolor=RoyalBlue,citecolor=royalblue(web),bookmarks=true,bookmarksopen=true,bookmarksopenlevel=2,unicode=true,linktocpage]{hyperref}
\usepackage[capitalise,noabbrev]{cleveref}

\numberwithin{equation}{section}

\newcommand{\ssb}{\begin{adjustwidth}{2.5em}{0pt}}
\newcommand{\sse}{\end{adjustwidth}}

\usepackage[utf8]{inputenc} 
\usepackage[T1]{fontenc} 

\usepackage{enumitem}
\usepackage{float}
\usepackage[labelfont={bf,sf}, textfont=sf]{caption}
\usepackage{chngcntr}
\usepackage{multicol}

\setlist[enumerate,1]{label={(\roman*)}, ref={(\roman*)}}

\newcommand\N{\mathbb{N}}
\newcommand\Z{\mathbb{Z}}
\newcommand\D{\mathbb{D}}

\newcommand\R{\mathbb{R}}
\def\C{\mathbb{C}}

\newcommand\Hb{\mathbb{H}}

\newcommand\Pb{\mathbb{P}}
\newcommand\Eb{\mathbb{E}}
\newcommand{\Var}{\mathrm{Var}}
\newcommand{\Cov}{\mathrm{Cov}}

\newcommand{\cle}{\mathrm{CLE}}
\newcommand{\sle}{\mathrm{SLE}}

\newcommand{\mfrak}{\mathfrak{m}}
\newcommand{\tfrak}{\mathfrak{t}}

\newcommand{\nlg}{\mathfrak{n}}
\newcommand{\Ccal}{\mathcal{C}}
\newcommand{\lfrak}{\mathfrak{l}}
\newcommand{\sfrak}{\mathfrak{s}}
\newcommand{\ndms}{\mathbbm{n}}
\newcommand{\Fcal}{\mathcal{F}}
\newcommand{\eb}{\mathbbm{e}}

\newcommand{\efrak}{\mathfrak{e}}

\newcommand\Xbf{\textbf{X}}
\newcommand\Zbf{\textbf{Z}}
\newcommand{\Mcal}{\mathcal{M}}

\usepackage[normalem]{ulem}

\theoremstyle{plain}
\newtheorem{Thm}{Theorem}[section]
\newtheorem{Thm-fr}{Théorème}[section]
\newtheorem*{Thm*}{Theorem}
\newtheorem{Prop}[Thm]{Proposition}
\newtheorem*{Prop*}{Proposition}
\newtheorem{Def}[Thm]{Definition}
\newtheorem*{Def*}{Definition}
\newtheorem{Lem}[Thm]{Lemma}
\newtheorem*{Cor*}{Corollary}
\newtheorem{Cor}[Thm]{Corollary}

\theoremstyle{definition}

\newtheorem*{Ex*}{Example}
\newtheorem{Rk}[Thm]{Remark}
\newtheorem*{Rmq*}{Remarques}
\newtheorem{Rks}[Thm]{Remarks}

\title{\textsc{Growth-fragmentations, Brownian cone excursions and SLE$_6$ explorations of a quantum disc} }
\author{William Da Silva\thanks{University of Vienna, Austria, \texttt{william.da.silva@univie.ac.at}} \quad \quad  Ellen Powell\thanks{Durham University, UK, \texttt{ellen.g.powell@durham.ac.uk}} \quad \quad  Alex Watson\thanks{University College London, UK, \texttt{alexander.watson@ucl.ac.uk}}}
\date{} 

\begin{document}

\maketitle

\vspace{-2em}
\begin{abstract}
  The aim of this article is to present a growth-fragmentation process naturally embedded in a Brownian excursion from boundary to apex in a cone of angle $2\pi/3$. This growth-fragmentation process corresponds, via the so-called mating-of-trees encoding \cite{DMS}, to the quantum boundary length process associated with a branching SLE$_6$ exploration of a $\gamma=\sqrt{8/3}$ quantum disc. However, our proof uses only Brownian motion techniques, and along the way we discover various properties of Brownian cone excursions and their connections with stable L\'{e}vy processes. Assuming the mating of trees encoding, our results imply several fundamental properties of the $\gamma=\sqrt{8/3}$--quantum disc SLE$_6$--exploration.
\end{abstract}

\tableofcontents


\section{Introduction}

\subsection{Main result for Brownian cone excursions} \label{sec: intro main result cones}

Let $\theta\in (\pi/2, \pi)$. Our starting point is a correlated Brownian motion pair, $W=((W^1_t,W^2_t))_{t\ge 0}$, satisfying
\begin{equation} \label{eq: cov MoT}
  \quad \Var(W_t^1)=\Var(W^2_t)=\mathbbm{a}^2t, \quad \Cov(W^{{1}}_t,W^{{2}}_t)=-\cos(\theta)\mathbbm{a}^2t,
  \quad \mathbbm{a}=\sqrt{{2}/{\sin(\theta)}},
\end{equation} for $t\ge 0$. We distinguish two types of {special} times for $W$: the first ones were considered in \cite{Bur-cones,Shi,DMS} and the second ones in \cite{LG-cones}. For compactness, in what follows let us denote $\R_+=[0,\infty)$ and $\R_+^*=(0,\infty)$. We say that $s\ge 0$ is {contained in a \textbf{forward cone excursion} if there exists $t\in [0,s)$ such that $W_r\in W_t+(\R_+^*)^2$ for $r\in (t,s]$; otherwise, we say that $s$ is a \textbf{forward cone-free time}.} We say that $t\ge 0$ is a \textbf{backward cone time} if $W_{{r}}\in W_t+(\R_+^*)^2$ for all ${r}\in (0,t)${: this is equivalent to asking that the two co-ordinates of $W$ reach a simultaneous running infimum at time $t$}. We shall see that one can construct local times $\ell_\theta$ and $\lfrak_\theta$, and their inverses $\uptau_\theta$ and $\tfrak_{\theta}$, on the set of forward cone-free times and backward cone times respectively.\footnote{Although our main results deal with the case when $\theta=2\pi/3$, we still provide a few results in the general case, which is why we keep $\theta$ as a subscript in general.}

For $z\in \partial \R_+^2 \setminus\{0\}$, we write {$P_\theta^z$} for the law of $W$ started from the point $z$ on the boundary and conditioned to \emph{remain in the positive quadrant} $(\R_+^*)^2$ until exiting \emph{at the origin} at time $\zeta$. The Brownian conditioning above is singular; we will make its meaning more precise later on in \cref{sec: forward cones}.
An excursion with law $P_\theta^z$ corresponds to 
 {a \emph{forward cone excursion}}.  The reason for the  ``cone'' terminology is that under the shear transformation
\begin{equation} \label{eq: correlation transformation}
  \mathbf{\Lambda}:= \frac{1}{\mathbbm{a}}
  \begin{pmatrix}
    \frac{1}{\sin \theta} & \frac{1}{\tan \theta} \\
    0 & 1
  \end{pmatrix},
\end{equation}
$W$ is mapped to a standard planar Brownian motion, and the quadrant $\R_+^2$ onto the closure of the cone $\Ccal_{\theta}:=\mathbf{\Lambda}(\R_+^2) = \{z\in\C, \; \arg(z)\in(0,\theta)\}$ with apex angle $\theta$. See \cref{rk: cone transformation}. 

\begin{figure}[ht]
	\bigskip
	\begin{center}
		\includegraphics[scale=0.9]{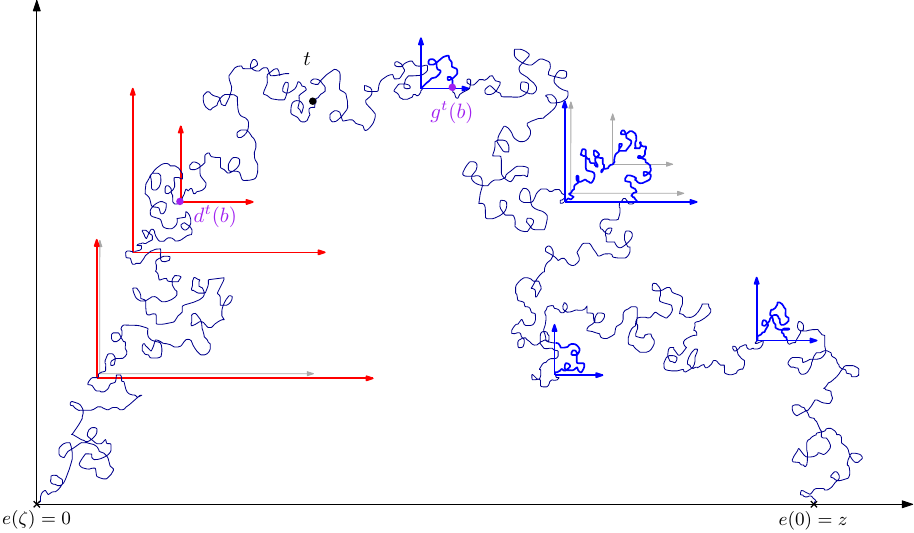}
	\end{center}
	\caption{The growth-fragmentation $\mathbf{Z}$ embedded in cone excursions. If $t$ is a time in the excursion, we record the forward cone excursions of $e^{t,-}$ (blue). We also depict some nested cone excursions in grey, to suggest that there is an accumulation of them inside each maximal excursion. For $0\le {b} \le {\varsigma^t}$, we construct (purple) an interval {$(g^t(b),d^t(b))$} containing $t$ such that {$g^t(b)=t-\uptau^t(b)$} and {$d^t(b)-t$} is the first simultaneous running infimum (red) of the path $e^{t,+}$ run from time $t$ to time $\zeta$, that also falls below the whole trajectory from $g^t(b)$ to $t$.  
	} 
	\label{fig: GF-cones}
\end{figure}

Let $e$ be an excursion under $P^z_\theta$, $z\in \partial \R_+^2\setminus \{0\}$, and for $t\in (0,\zeta(e))$ let
\[
  e^{t,-}:=(e(t-s)-e(t), 0\le s\le t) \quad \text{and} \quad e^{t,+}:=(e(t+s)-e(t), 0\le s\le \zeta-t).
\]
{That is, $e^{t,-}$ is the time reversal of the path from $e(0)=z$ to $e(t)$ (so, roughly speaking, corresponds to following the path starting from $e(t)$ and ``going right'' in Figure \ref{fig: GF-cones}), while $e^{t,+}$ is the path from $e(t)$ to $e(\zeta)=0$ (starting from $e(t)$ and ``going left'' in Figure \ref{fig: GF-cones}).} 
By local absolute continuity with respect to $W$, one may define the forward cone-free times of $e^{t,-}$ and the backward cone times of $e^{t,+}$ as above. This yields in particular an inverse local time $\uptau^t_\theta$ related to the forward cone-free times of $e^{t,-}$ which is an increasing process with some lifetime $\varsigma^t_{\theta}$.

Our main result concerns the case $\theta=2\pi/3$, and in this case we simply write $P^z$ for the measure discussed above and drop the subscript $\theta$ for all the quantities. Let {$e$} be a process with law $P^z$ for some $z\in \partial \R_+^2\setminus \{0\}$. For each $t\in (0,\zeta)$ we define a sequence of non-decreasing intervals $(g^t(b), d^{t}(b))$ containing $t$, indexed by $b\in [0,\varsigma^t]$. We first simply let $\tilde{g}^t(b) := \uptau^t({b})$, and {then set $g^t(b)=t-\tilde{g}^t(b)$, so $\tilde{g}^t(b)$ is the first time that $e^{t,-}$ has accumulated local time $b$ on the complement of its forward cone excursions, and $g^t(b)$ is the corresponding time in the excursion $e$}. We then define

{
\begin{equation} \label{eq: def d_t and g_t}
  d^t(b)
  = \inf\{ u>t :
    e(r) \in e(u) + \R^2_+
    \text{ for all }
    r \in [g^t(b), u]
  \},
\end{equation}
In words, $d^t(b)$ is the first simultaneous running infimum of $e$ after time $t$ that also falls below $e([g^t(b),t])$.} See Figure \ref{fig: GF-cones}.

Finally, define for all $a\in [0,\varsigma^t]$,
\begin{equation} \label{eq: cal Z and Z^t}
{\mathcal{Z}^t(a) 
  :=   
  e(g^t((\varsigma^t-a)^-)) - e(d^t((\varsigma^t-a)^-))} 
  \quad \text{and} \quad
  Z^t(a) 
  := 
  \lVert \mathcal{Z}^t(a) \rVert_1.
\end{equation}
{Notice the reversal of time here: as $a$ increases we are considering the difference of $e$ at the end points of the intervals $(g^t(b),d^t(b))$ with $b=\varsigma^t-a$, which are \emph{decreasing} from $(0,\zeta)$ to $\emptyset$}, {so that $\mathcal{Z}^t(0)= z$ and $\mathcal{Z}^t(\varsigma^t)=0$}. We may view $\mathcal{Z}^t$ as a process defined for all {(local) times $a$} by sending it to the cemetery state $0$ after time $\varsigma^t$.
Note that by construction of $d^t$, $\mathcal{Z}^t(a)$ lies in the positive quadrant $\R_+^2$, so that $Z^t(a)$ is nothing but the sum of the co-ordinates of $\mathcal{Z}^t(a)$. We observe that this {construction} extends naturally to define $Z^t$ simultaneously for all $t\in(0,\zeta)$. Indeed, the collection $\varsigma^r, g^r((\varsigma^r-a)-), d^r((\varsigma^r-a)-)$ is already defined simultaneously (on an event of probability one) for all $r\in \mathbb{Q}\cap (0,\zeta)$ and $a<\varsigma^r$. Moreover, for any $a>0$, the $(g^r((\varsigma^r-a)-), d^r((\varsigma^r-a)-))$ for $r\in \mathbb{Q}\cap (0,\zeta)$ are a countable collection of {either equal or} disjoint intervals, which can only decrease in size or split as $a$ increases. Thus for any $t\in (0,\zeta)$, 
	\[
	\varsigma^t=\sup\{a: t\in (g^r((\varsigma^r-a)-), d^r((\varsigma^r-a)-)) \text{ for some } r\in \mathbb{Q}\cap (0,\zeta)\}, 
	\]
is well-defined and for $a<\varsigma^t$, all intervals of the form $(g^r((\varsigma^r-a)-), d^r((\varsigma^r-a)-)) \text{ with } r \in \mathbb{Q}\cap (0,\zeta)$ {containing $t$} must coincide, so we can set $\mathcal{Z}^t(a)=e(g^r((\varsigma^r-a)^-)) - e(d^r((\varsigma^r-a)^-))$ for any such $r$, and $Z^t(a)=\lVert \mathcal{Z}^t(a) \rVert_1$.
We then set
\begin{align} \label{eq: GF cones}
  \Zbf(a) & = \big\{ Z^t(a), \; t\in(0,\zeta) \text{ such that } \varsigma^t>a \big\} \nonumber \\
  & = \big\{ Z^t(a), \; t\in\mathbb{Q}\cap (0,\zeta) \text{ such that } \varsigma^t>a \big\},
\end{align}
for $a\ge 0$, where the second equality above is clear from the preceding discussion. We emphasise that many times $t$ correspond to the same value of $Z^t(a)$, but we only record this value once in the above set. 

{To unpack this definition a little, first notice that when $a=0$, the interval $(g^t({\varsigma^t}-a),d^t({\varsigma^t}-a))$ will be equal to $(0,\zeta)$ for all $t$, and so $\mathbf{Z}(0)$ will consist of a single element ${|z|}={\lVert e(0)-e(\zeta)\rVert_1}$. As $a$ increases, this will no longer be the case, and $\mathbf{Z}$ will contain more elements. Indeed, for $s\ne t$, as long as $s\in
(g^t(\varsigma-a),d^t(\varsigma-a))$, the values of $Z^t(a)$ and $Z^s(a)$ will coincide. However, {as soon as this breaks down}, the intervals will ``split'' and the corresponding element of $\mathbf{Z}$ will become two.}

Our main theorem shows that $\Zbf$ is a growth-fragmentation process and moreover, that this process is explicitly described \textit{via} a positive self-similar Markov process with index $\frac32$. To be more specific, under $\Pb_x$, $x>0$, let $(X^{3/2}(a), 0\le a<T_0)$ be the positive self-similar Markov process with index $\frac32$ starting from $x$, {killed at the first time $T_0$ when it reaches $0$}, given by 
\[
  X^{3/2}(a)
  :=
  x \exp(\xi(\tau(x^{-3/2}a))), \quad a<T_0,
\]
where $\xi$ is a Lévy process with Laplace exponent 
\[
  \Phi_{3/2}(q) 
  :=
  -\frac{16}{3}q + 2\int_{-\log(2)}^0 (\mathrm{e}^{qy}-1-q(\mathrm{e}^y-1)) \mathrm{e}^{-3y/2}(1-\mathrm{e}^y)^{-5/2}\mathrm{d}y,
  \quad
  q\in \R,
\]
and $\tau$ is the Lamperti time change
\[
  \tau(t) 
  :=\inf\{s\ge 0, \int_0^s \mathrm{e}^{\tfrac32\xi(u)}\mathrm{d}u>t\}, \quad t\ge 0.
\]
The \textbf{growth-fragmentation process} driven by $X^{3/2}$ can be roughly constructed as follows. At time $t=0$, the system starts from one particle $\mathcal{X}_{\varnothing}$ with initial size $x>0$, which then evolves as $X^{3/2}$ under $\Pb_x$. Conditionally on $\mathcal{X}_{\varnothing}$, one starts a new particle at any time $t$ when $\mathcal{X}_{\varnothing}$ has a negative jump, starting from $y=-\Delta \mathcal{X}_{\varnothing}(t):=\mathcal{X}_{\varnothing}(t^-)-\mathcal{X}_{\varnothing}(t)$ and whose behaviour is governed by independent copies of $\Pb_y$. This constructs the children of $X$, for which we repeat the same procedure, thus creating the second generation, and so on. For $a\ge 0$, we let $\Xbf^{3/2}(a)$ denote the collection of sizes of the cells alive at time $a$. More details are provided in \cref{sec: GF process}. 
Our first main result in the following.

\begin{Thm}[Growth-fragmentation process: cone excursions] \label{thm: main}
  Under $P^z$, the process $\mathbf{Z}$ has the same law as $\mathbf{X}^{3/2}$ under $\Pb_{|z|}$.
\end{Thm}

The process $\Xbf^{3/2}$ was first introduced by Bertoin, Curien and Kortchemski in \cite{BCK} (see also \cite{BBCK}). We comment on this connection {and related work} in \cref{sec: related work} and provide additional details in \cref{sec: GF process}.


\subsection{Interpretation in terms of Liouville quantum gravity and Schramm--Loewner evolutions} 
\label{sec: intro LQG}
Liouville 	quantum gravity (LQG) surfaces are a family of ``canonical'' random fractal surfaces, that conjecturally describe the large-scale behaviour of discrete random surfaces called random planar maps. Such surfaces were first considered in the physics literature \cite{HK,polyakov1981quantum,KPZ}: see \cite{Dup-She} for a comprehensive list of references. Informally speaking, a $\gamma$--Liouville quantum gravity surface parametrised by $D\subset \mathbb{C}$ should be a random Riemannian surface with metric tensor  
\begin{equation} \label{eq: LQG metric}
  \mathrm{e}^{\gamma h(z)} (\mathrm{d}x^2+\mathrm{d}y^2), \quad z=x+iy\in D,
\end{equation}
where $\mathrm{d}x^2+\mathrm{d}y^2$ is the Euclidean metric tensor and $h$ is a variant of the planar Gaussian free field. The issue with this definition is that $h$ is not a random function but a random distribution, so that making sense of its exponential requires some highly non-trivial work. Nevertheless, one can give a meaning to \eqref{eq: LQG metric} for $\gamma\in(0,2)$ in a number of different ways. The first progress in this direction was to construct the associated \textbf{quantum area measure} $\mu^\gamma_h$ on $D$, using the so-called \textbf{Gaussian multiplicative chaos} approximation procedure, \cite{Kah, Dup-She, Ber-GMC}. Similarly, one can construct a \textbf{quantum boundary length measure} $\nu_h^{\gamma}$ on $\partial D$, and more generally on some curves in $D$, including $\sle_{\kappa}$ or $\sle_{\kappa'}$--type curves when $\kappa:=\gamma^2$ and $\kappa':=16/\gamma^2$ \cite{She-zip}. 
More recently, a metric $D_h^\gamma$ corresponding to \eqref{eq: LQG metric} was constructed via approximation, \cite{Dingetal-tightness,GM-existenceuniqueness} (see also \cite{MS-BM1} for the case $\gamma=\sqrt{8/3}$ and \cite{DG-supercritical} for the critical and supercritical cases).

Quantum surfaces conjecturally correspond to the scaling limits of random planar maps coupled with critical statistical mechanics models. 
In these discrete couplings the randomness of the map and the statistical mechanics decoration are finely tuned so that the partition function of the latter matches the distribution of the planar map. One example is the FK--decorated map model, {that can be seen as} a model on loop-decorated planar maps, where the probability of observing a given map and collection of loops is proportional to $\sqrt{q}^L$, where $q$ is a parameter and $L$ is the total number of loops. 
Despite the fact that at the discrete level the map and loops 
are not independent, it is believed that in an appropriate scaling limit (for example, as the number of faces in the map 
goes to $\infty$ and the whole picture is embedded in $\mathbb{C}$ in a suitable way) 
the loops and the geometry of the map decouple. Moreover, the limiting geometry of the random map should be described by a $\gamma$--Liouville quantum gravity surface and the limiting loop collection should be an independent, nested \textbf{conformal loop ensemble} $\cle_{\kappa'}$, which is a random collection of nested, non-crossing loops in the disc \cite{She-exploration,SW}, where 
\begin{equation} \label{eq: relation q kappa}
  q = 2 + 2\cos(8\pi/\kappa'), \quad \gamma = 4/\sqrt{\kappa'}.
\end{equation}

Although such a convergence statement is not proven, it is known that one can encode the FK--decorated map model in terms of non-Markovian random walk on $\Z^2$ \cite{Mul,She}, and \cite{She} further proved that this random walk converges to a correlated Brownian motion as in \eqref{eq: cov MoT} with \begin{equation}\label{eq: intro relation gamma theta}
	\gamma=\sqrt{4\theta/\pi}, \text{ equivalently } \kappa'=4\pi/\theta.
\end{equation}
Moreover, there is an analogue of this encoding in the continuum, which gives 
a way to encode an independent {CLE$_{\kappa'}$}, $\gamma$--LQG surface pair in terms of
such a Brownian motion,
\cite{DMS}. In fact, the main theorem of \cite{DMS} describes a correspondence between the correlated Brownian motion and a $\gamma$--LQG surface together with an independent space-filling curve called \textbf{space-filling} $\sle_{\kappa'}$, but this space-filling curve can be used to define {an} entire nested CLE$_{\kappa'}$ (and vice-versa).

The version of this theorem most directly related to our work will be when the $\gamma$--Liouville quantum gravity surface is something called a \textbf{unit boundary length quantum disc.} This is a natural quantum surface with boundary (that is, with the topology of a disc, or parameterised by $D\subset \C$ with $\partial D\ne \emptyset$) that has finite (but random) quantum area and quantum boundary length equal to $1$. It should arise as the scaling limit of random planar maps with an outer perimeter of given length. In this setting, the analogous result to \cite{DMS} is \cite[Theorem 1.1]{AG}. It says that if one draws a \emph{counterclockwise} space-filling $\sle_{\kappa'}$ $\eta$ on top of a unit boundary length $\gamma$--quantum disc\footnote{These will be defined precisely in Section \ref{sec: mot}.} parametrised by the unit disc $\mathbb{D}\subset \mathbb{C}$, and parametrises $\eta$ by quantum area, then the change $L_t$ and $R_t$ in quantum boundary lengths of the left and right sides of $\eta([0,t])$ relative to time $0$ as in \cref{fig: space-filling sle on disc}, normalised so that $(L_0,R_0)=(0,1)$, has the law $P^{(0,1)}_\theta$ from below \eqref{eq: cov MoT}, where $\theta$ is as in \eqref{eq: intro relation gamma theta}. See Theorem \ref{thm: MoT complete} for an exact statement.

\bigskip
\begin{figure}[ht]
  \begin{center}
    \includegraphics[scale=0.7]{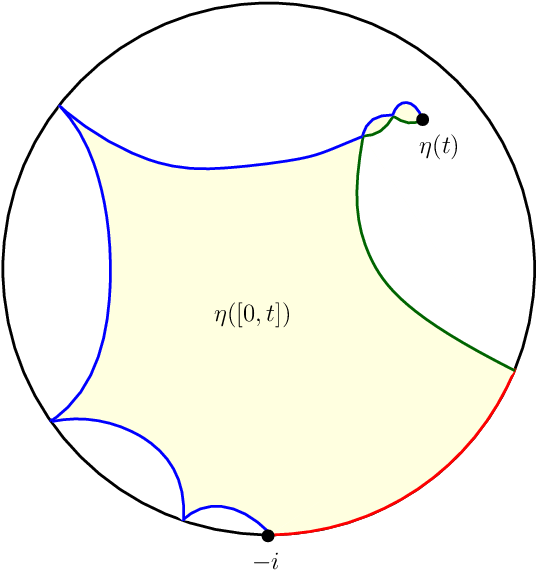}
  \end{center}
  \caption{A unit boundary length quantum disc decorated with an independent space-filling $\sle_{\kappa'}$ $\eta$ from $-i$ to $-i$, parametrised by quantum area.  $L_t$ corresponds to the quantum length of the blue curve. $R_t$ corresponds to {one plus} the quantum length of the green curve \emph{minus} the quantum length of the red one. 
  }
  \label{fig: space-filling sle on disc}
\end{figure}

Our main result \cref{thm: main} concerning $P_\theta^{(0,1)}=P^{(0,1)}$ with
$\theta=2\pi/3$ thus corresponds to the case $\gamma=\sqrt{8/3}$, $\kappa'=6$,
and can be rephrased as follows.

Let $(\D,h,-i)$ a singly marked
unit-boundary $\sqrt{8/3}$--quantum disc and $\eta$ a space-filling $\sle_6$ (see Section \ref{sec: mot} for precise definitions).  Let $e=(L,R)$ be the associated correlated Brownian excursion given by \cite[Theorem 1.1]{AG} (see also Theorem \ref{thm: MoT complete}), as described above.

For any deterministic point $z\in \mathbb{D}$, one can consider the \textbf{branch} $\eta^z$
of $\eta$ towards $z$, in the sense that it does not explore the components of $\D$
that it disconnects from $z$ along its way. In other words, we erase intervals of time on which $\eta$ is visiting such a component. {In the Brownian motion picture,} such intervals correspond precisely to forward cone excursions for $e^{t_z,-}$ where $t_z$ is the almost surely unique time that $\eta(t_z)=z$ (but they are visited by $\eta$ in the opposite order to their appearance in $e^{t_z,-}$). Thus we can parametrise $\eta^{z}$ run backwards, from $z$ to $-i$, by the inverse local time {$\uptau^{t_z}$} for $e^{t_z,-}$ defined in Section \ref{sec: intro main result cones}. After then reversing time, this branch has the law of a radial SLE$_6$ from $i$ to $z$ \cite{MS-space-filling,DMS}, and this time parametrisation is called its {\bf quantum natural time parametrisation} in \cite{DMS}.

\begin{figure}
  \bigskip
  \centering
  \subfigure[]{\includegraphics[scale=0.9,page=1]{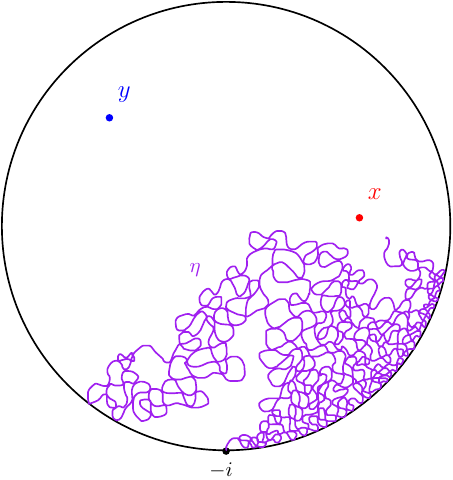}} 
  \hspace{1cm}
  \subfigure[]{\includegraphics[scale=0.9,page=2]{Figures/branch-sle-lqg.pdf}} 
  \caption{Branches of the space-filling $\sle_6$ $\eta$ on the $\sqrt{8/3}$--quantum disc towards $x$ and $y$. (a) The branch of $\eta$ towards $x$ and $y$ (purple) is the same. (b) The two branches get disconnected: a loop has been cut out, surrounding $x$. The branch $\eta^x$ targeted at $x$ is shown in (purple and then) red, and the branch $\eta^y$ targeted at $y$ is in (purple and then) blue.}
  \label{fig: branches SLE x and y}
\end{figure}

\begin{figure}
  \bigskip
  \centering
  \includegraphics{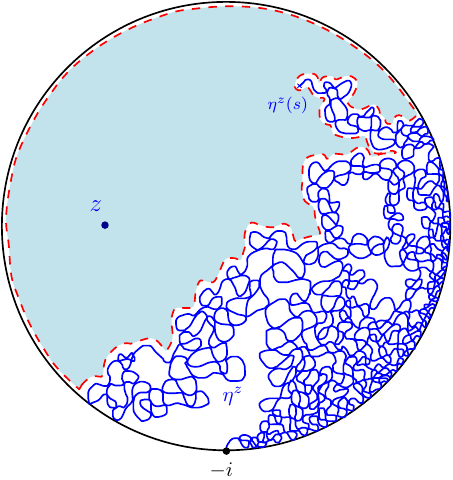} 
  \caption{The total boundary process towards $z\in\D$. At time $s$ along the branch $\eta^z$, we record the total boundary length (dashed red) of the component containing $z$ (blue).}
  \label{fig: total boundary}
\end{figure}

Simultaneously constructing the branches
towards all points in the disc (with rational co-ordinates, say) yields a {\bf branching SLE$_6$}. It has the property that for any two points, the branches coincide until $\eta$
disconnects them (see \cref{fig: branches SLE x and y}).
{With a slight abuse of notation, we let $\varsigma^z$ denote the lifetime of the branch $\eta^z$, \textit{i.e.} $
\varsigma^z := \inf\{s>0 : \eta^z(s)=z\},
$
so that since $\eta^z$ is parametrised using {$\uptau^{t_z}$}, we have $\varsigma^z=\varsigma^{t_z}$, where the right-hand side is as defined in Section \ref{sec: intro main result cones}. 
At fixed (quantum natural time) $a\in (0,\varsigma^z)$, one can define $D^z_a $ to be the connected component of $\D\setminus \eta^{z}([0,a])$ containing $z$. 
}Again via the correspondence with $e=(L,R)$ it is straightforward to see that 
\begin{equation} \label{eq: D^z_a into cones}
D^z_a
=
\eta ((g^{t_z}(a), d^{t_z}(a))).
\end{equation}

Thus for any $a>0$, the collection $\{D^z_a, \, z\in \mathbb{Q}^2\cap \D\}$ is a countable collection of open sets containing all $y\in \mathbb{Q}^2\cap \D$ such that $\varsigma^{t_y}>a$.
If we record the quantum boundary lengths (using $h$) of these sets, this yields
a countable collection of positive real numbers that we call $\widetilde{\Zbf}(a)$ {(see \cref{fig: total boundary})}. By \eqref{eq: D^z_a into cones} and the definition of $(L,R)$ from $(\eta,h)$ we have that
\begin{align*}\widetilde{\Zbf}(a) & = \big\{ Z^t(a), \; t\in(0,\zeta) \text{ such that } t=t_z \text{ for some } z\in \mathbb{Q}^2\cap \D \text{ and } \varsigma^t>a \big\} \\
	& =\Zbf(a), 
\end{align*}
as processes in $a>0$, where $Z^t, \Zbf$ are constructed from $e=(L,R)$ as in
\eqref{eq: GF cones}. The second equality holds since the $t_z$ for $z\in\mathbb{Q}^2\cap \D$ are dense in $(0,\zeta)$ and by the same reasoning as in \eqref{eq: GF cones}. 
From \cref{thm: main} we therefore obtain the following
explicit description of the law of $\widetilde{\Zbf}$: 

\begin{Thm}[Growth-fragmentation: $\sqrt{8/3}$--LQG] \label{thm: mainLQG}
  The branching total boundary length process $\widetilde{\mathbf{Z}}$ described above, defined from a space filling $\sle_6$ exploration of a $\gamma=\sqrt{8/3}$ unit boundary length quantum disc, has the law of the growth-fragmentation process $\mathbf{X}^{3/2}$ from \cref{sec: intro main result cones} {under $\Pb_1$}.
\end{Thm}

We stress once again that our proof of \cref{thm: mainLQG} relies only on Brownian motion arguments (assuming the mating of trees), since it comes as a corollary of \cref{thm: main}. We comment on related work in \cref{sec: related work} and provide more LQG background in \cref{sec: mot}.
{It is {possible} that \cref{thm: mainLQG} could be proved using variants of the arguments in \cite{MSW}. The aim of the present paper is to provide an elementary, purely Brownian, proof of \cref{thm: mainLQG}, establishing on the way new elements of excursion theory for $2\pi/3$ cone times. In particular we do not use the target-invariance property of $\sle_6$, but rather derive it from excursion-theoretic techniques. More generally we believe that these arguments provide a new toolbox for the $\sqrt{8/3}$--LQG, $\sle_6$ coupling.}

\subsection{Further results for cone excursions and their counterparts for SLE and LQG}
\label{sec: further}

We describe a few additional results that we obtain along the way and explain how they translate in the setup of $\sle_6$ on $\sqrt{8/3}$--LQG. Most results appear in some form or another in the LQG literature, 
but we emphasise again that our proofs are elementary, using only Brownian motion arguments. It is likely that the excursion theory we develop for cone points in the present paper can be used to extract more information on both sides. {We also stress that, apart from their connection to $\sqrt{8/3}$--LQG, it is not clear and somehow surprising from the Brownian perspective that $\frac{2\pi}{3}$ cone times display so many special features.}

The first result we obtain is the explicit density of the duration $\zeta$ under $P^z$, which was derived in \cite[Theorem 1.2]{AG}. 
\begin{Prop}[Law of duration conditioned on displacement] 
  \label{prop: intro law duration}
  Under ${P^{(0,1)}}$, the law of $\zeta$ is given by 
  \[
    P^{{(0,1)}}(\zeta\in\mathrm{d}t)
    =
    \frac{3^{-3/4}}{\sqrt{2\pi}} \mathrm{e}^{-\frac{1}{2\sqrt{3} t}} t^{-5/2}  \mathrm{d}t.  
  \]
  Moreover, the law of $\zeta$ under $P^z$ is that of $|z|^2 \zeta$ under $P^{{(0,1)}}$.
\end{Prop}
\noindent {Actually, the above result will be proved by taking a limit from a much stronger result giving the joint law of the displacement and duration of \emph{backward cone excursions} (starting from the interior of the quadrant $\R_+^2$). This stronger version solves a question that was raised by Le Gall \cite{LG-cones} (see Remark (ii) on page 613) about giving an explicit formula for the Lévy measure of planar Brownian motion subordinated on the set of backward cone times. We provide an explicit expression for this Lévy measure in the case when $\gamma=\sqrt{8/3}$, thereby solving the question in the case when $\alpha=\pi/3$ and $\nu=3/2$ in Le Gall's notation.}
\cref{prop: intro law duration} allows us to describe the law of the area of a unit-boundary quantum disc, as in \cite{AG}.

\begin{Cor}[Law of area of unit-boundary quantum disc] \label{cor: intro law area}
  For $\gamma=\sqrt{8/3}$, the law of the area of a unit-boundary quantum disc is
  \[
    \frac{3^{-3/4}}{\sqrt{2\pi}} \mathrm{e}^{-\frac{1}{2\sqrt{3} t}} t^{-5/2}  \mathrm{d}t.  
  \]
  Moreover, the law of the area of a quantum disc with boundary length $x>0$ is $x^2$ times that of a unit-boundary quantum disc. 
\end{Cor}

The second result we obtain is a new (purely Brownian) proof of the so-called \textbf{target-invariance} of $\sle_6$ on the $\sqrt{8/3}$--quantum disc, that was obtained by Miller and Sheffield \cite{MS-mot}.
We first state the result as a property of Brownian excursions. Recall from \cref{sec: intro main result cones} that $P^z$ denotes the law of a correlated Brownian excursion in the positive quadrant starting from $z\in\partial \R_+^2\setminus \{0\}$ and ending at the origin, with correlation \eqref{eq: cov MoT}. Introduce the probability measure
\begin{equation} \label{eq: intro def P bar}
  \overline{P}^z (\mathrm{d}T,\mathrm{d}e)
  :=
  \sqrt{3} \lVert z\rVert_1^{-2} \mathds{1}_{\{ 0\le T\le \zeta(e)\}} \mathrm{d}T P^z(\mathrm{d}e)
\end{equation}
on $\R_+ \times E$, {which consists in sampling a uniform time in the excursion weighted by its duration.}
It can be checked from \cref{prop: intro law duration} that $\Eb^{P^z}[\zeta] = {\lVert z \rVert_1^2}/{\sqrt{3}}$, which ensures that $\overline{P}^z$ is a probability measure . 
Under $P^z$ we define the processes $Z^t$ and $\mathcal{Z}^t$ for $t\in(0,\zeta)$ as in \eqref{eq: cal Z and Z^t}, and the total local time $\varsigma^t$ as in \cref{sec: intro main result cones}.

\begin{Prop}[Target-invariance: cone excursions] \label{prop: intro target inv cone}
  Under $\overline{P}^z$ and for all $a\ge 0$, on the event that $\varsigma^T>a$, $\frac{\mathcal{Z}^T(a)}{Z^T(a)}$ is independent of $(Z^T(b), b\ge 0)$ and distributed as $(U,1-U)$ with $U$ uniform in $(0,1)$.
\end{Prop}

\noindent We rephrase the above result in terms of $\sle_6$ explorations of the $\sqrt{8/3}$--quantum disc. In the setting of \cref{sec: intro LQG}, we consider a unit boundary length quantum disc $(\D,h,-i)$ with law reweighted by its total quantum area $\mu_h^\gamma(\D)$, and given $h$, we sample  $z^{\bullet}$ in $\D$ according to the quantum area measure $\mu_h^\gamma$.
We then look at the branch $\eta^{z^\bullet}$ targeted at $z^\bullet$ and define $(L^\bullet, R^\bullet)$ as the left and right quantum boundary length process of the component containing this point, when $\eta^{z^\bullet}$ is parametrised by quantum natural time. Write $Z^\bullet := L^\bullet + R^\bullet$ for the total boundary length process, and $\varsigma^\bullet$ for the duration of the branch $\eta^{z^\bullet}$. 

\begin{Cor}[Target-invariance: LQG]
  \label{cor: intro target inv lqg}
  For all $a\ge 0$, on the event that $\varsigma^{\bullet}>a$,  $\left( \frac{L^\bullet(a)}{Z^\bullet(a)},  \frac{R^\bullet(a)}{Z^\bullet(a)}\right)$ is independent of $(Z^\bullet(b), b\ge 0)$ and distributed as $(U,1-U)$ with $U$ uniform in $(0,1)$.
\end{Cor}
\noindent \cref{cor: intro target inv lqg} states that {given the total boundary length process $Z^\bullet$ as we explore towards $z^{\bullet}$, the position of the tip of $\eta^{z^\bullet}$ on the boundary of $D^{z^\bullet}_a$ at any time $a$ is distributed uniformly according to quantum boundary length. In particular, we can resample the location of the tip according to quantum boundary length at any time without changing the law of the boundary length process.} 
The above two results will be proved at the end of Section \ref{sec: target invariance}.

In addition, we can describe explicitly the law of the total boundary length process $Z^\bullet$. This is a continuum analogue of the third item of \cite[Proposition 6.6]{BBCK}; see also  \cite{GM-sle6} for closely related results in a slightly different setting. 
More details on Lévy processes and their conditionings will be given in \cref{sec: pssMp and Levy}.
\begin{Prop}[Law of the uniform exploration: cone excursions]
  In the setting of \cref{prop: intro target inv cone}, the process $(Z^T(a), 0\le a<\varsigma^T)$ evolves as a spectrally negative $\frac32$--stable Lévy process conditioned to be absorbed at $0$, started at $z$. More precisely, it has normalising constant $c_\Lambda=2$ {as in \cref{sec: pssMp and Levy}}. \label{prop: law unif intro}
\end{Prop}

\begin{Cor}[Law of the uniform exploration: LQG]
  In the setting of \cref{cor: intro target inv lqg}, the process $(Z^{\bullet}(a), 0\le a<\varsigma^{\bullet})$ evolves as a spectrally negative $\frac32$--stable Lévy process conditioned to be absorbed at $0$, started at $z$. More precisely, it has normalising constant $c_\Lambda=2$ {as in \cref{sec: pssMp and Levy}}. \label{cor: law unif intro}
\end{Cor}

Furthermore, we obtain the following pathwise construction of the spectrally positive $\frac32$--stable Lévy process conditioned to stay positive, which is of independent interest. The construction can be seen as a planar version of the one-dimensional construction of Bertoin \cite{bertoin1993splitting}, albeit in the special case of the $\frac32$--stable process. Our proof however does not rely on \cite{bertoin1993splitting} and it is not clear to us whether one implies the other. We only state an informal version since the claim requires to introduce quite a bit of notation.

\begin{Prop}[Pathwise construction of the spectrally positive $\frac{3}{2}$--stable process conditioned to stay positive -- informal version]
  \label{prop: intro pathwise 3/2 stable cond}
  Let $W$ and $W'$ two independent correlated Brownian motions, with correlation as in \eqref{sec: BM construction conditioning}, started from $0$. Denote by $W\circ \tfrak$ (\textit{resp.} $W'\circ\uptau$) the process time-changed by the inverse local time of its backward cone times (\textit{resp.} forward {cone-free} times). Introduce the first passage time 
  \[
    \sfrak(a) := \inf\{s\ge 0, \; W'(\uptau(t)) \in W(\tfrak(s))+\R_+^2 \; \text{ for all } \; t\le a \}.
  \]
  Then, the process $S$ defined as 
  \[
    S(a):= \lVert W'\circ\uptau(a)-W\circ \tfrak(\sfrak(a)) \rVert_1, \quad a\ge 0,
  \]
  is the spectrally positive $\frac32$--stable Lévy process conditioned to stay positive (with $c_\Lambda=2$).
\end{Prop}

\noindent We emphasise that it is not even clear \textit{a priori} that the above construction of $S$ yields a Markov process. We prove this in \cref{sec: target invariance}, and the full claim in \cref{sec: BM construction conditioning}.

The above claim has a natural LQG interpretation. Since it concerns (correlated) Brownian motion in the whole plane, the setup here corresponds to Liouville quantum gravity surfaces called {\bf quantum cones}, which have infinite area, {and whole-plane space-filling SLE$_6$, see \cite{DMS} or \cite[Chapter 9]{BP}}. In this case one can also define a branch of the $\sle_6$ corresponding to any point in the domain as in \cref{sec: intro LQG} and re-parametrise it by quantum natural time. Through the mating of trees for quantum cones (see \cite[Theorem 1.9]{DMS}), \cref{prop: intro pathwise 3/2 stable cond} then describes the law of the total quantum boundary length process along the branch corresponding to a quantum-typical point, {but time-reversed}. The law is that of a spectrally positive $\frac32$--stable Lévy process conditioned to stay positive. The two-sided Brownian motion in \cite[Theorem 1.9]{DMS} is given by gluing the two paths $W$ and $W'$ from \cref{prop: intro pathwise 3/2 stable cond}, and the Liouville--typical point then corresponds to time $0$.

Finally, we recover a natural martingale for the growth-fragmentation process $\bf Z$. {Recall from \cref{sec: intro main result cones} that  growth-fragmentation processes are constructed generation by generation, starting from an initial common ancestor $Z$ say, grafting copies of $Z$ at each jump time of $Z$, and so on.} Denote by $\mathcal{Z}_u$ the particle indexed by $u\in \cal U$ in the growth-fragmentation process, and {write $|u|$ for the generation of $u$}. See \cref{sec: GF process} for a rigorous definition. 
{Note that $\mathcal{Z}_u$ depends on the choice of initial ancestor $Z$. In the following claim, we fix an arbitrary choice of $Z$ (the claim holds regardless of that choice).}

\begin{Thm}\label{thm: martingale}
  Let $z\in \partial \R_+^2 \setminus \{0\}$. Under $P^z$, the process 
  \[
    \Mcal(n) :=
    \frac{1}{\sqrt{3}}\sum_{|u|=n} \mathcal{Z}_u(0)^2, 
    \quad n\ge 1,  
  \]
  is a martingale. Furthermore, it is uniformly integrable and converges $P^z$--almost surely and in $L^1$ to the duration $\zeta$ of the excursion.
\end{Thm}
\noindent The above martingale already appears in \cite{BCK} (or \cite{BBCK}) for the process $\mathbf{X}^{3/2}$. We will prove this result in \cref{sec: few prop of Zbf} using purely Brownian arguments once more, see \cref{thm: Mcal convergence}.


\subsection{Related work and open questions}
\label{sec: related work}

\paragraph{Related work.}
{The process $\Xbf^{3/2}$ of \cref{thm: main} belongs to a} family $\Xbf^{\alpha}$, $\alpha\in(\frac12, \frac32]$, of growth-fragmentation processes that was first introduced by Bertoin, Budd, Curien and Kortchemski \cite{BBCK}, who proved that they arise in the scaling limit of a branching peeling exloration of Boltzmann planar maps (the case $\alpha=3/2$ was actually considered earlier \cite{BCK}). This family is described more precisely in \cref{sec: GF process}. Since then, it has been retrieved in the continuum in a variety of contexts. Le Gall and Riera \cite{LGR} first proved that a time-change of $\Xbf^{3/2}$ shows up when slicing the Brownian disc at heights. This is particularly relevant to our work since it concerns $\alpha=\frac32$. In fact in this case there is a beautiful correspondence, due to a breakthrough of Miller and Sheffield  \cite{MS-BM1,MS-BM2,MS-BM3}, between the $\sqrt{8/3}$--quantum disc (endowed with the QLE metric) and the Brownian disc.  Nonetheless, we stress that our exploration is different to that of \cite{LGR} since it does not involve the metric. A similar distinction actually already appears in \cite{BBCK} (see Section 6.5 there), {where our $\sle_6$ exploration rather relates to the peeling exploration.} 

On the other hand, Aïdékon and Da Silva \cite{AD} proved that the growth-fragmentation process $\Xbf^{1}$ appears when cutting a half-plane Brownian excursion at heights. Through the ``mating-of-trees encoding'' of critical Liouville quantum gravity and CLE$_4$ by Aru, Holden, Powell and Sun \cite{AHPS}, this translates into a very similar picture to the present work, for $\gamma=2$ and $\kappa'=4$. Observe from \eqref{eq: intro relation gamma theta} that $\theta = \pi$ in this case, which is consistent with half-plane excursions. The first-named author also recovered the processes $\Xbf^{\alpha}$ for $\alpha\in(\frac12,1)$ by studying variants of the previous half-plane Brownian excursions, where the imaginary part is replaced with a stable process \cite{DS}.

Finally, Miller, Sheffield and Werner \cite{MSW} contructed the processes $\Xbf^{\alpha}$ for $\alpha\in(1,\frac32)$ directly in the quantum gravity setting. Let us now  describe some of their results, since their viewpoint is very relevant to the present work. Let $\gamma\in(\sqrt{8/3},2)$ and, as usual, set $\kappa:=\gamma^2$ and $\kappa':=16/\gamma^2$. Consider a singly-marked unit boundary length $\gamma$--quantum disc $(\D,h,i)$ together with an independent \emph{conformal loop ensemble} $\cle_{\kappa}$ on $\D$ and an associated  \emph{conformal percolation interface} (CPI) in the $\cle_{\kappa}$ carpet. Roughly speaking, the CPI is an $\sle_{\kappa'}$--type curve that stays in the $\cle_{\kappa}$ carpet. 
Consider a CPI branch (parametrised by quantum length) towards some point $z$ in the $\cle_{\kappa}$ carpet, and record the quantum boundary length of the connected component containing $z$ as the CPI evolves. This exploration has positive and negative jumps (see \cref{fig: CLE on LQG MSW}). Moreover, if $x$ and $y$ are any two points in the $\cle_{\kappa}$ carpet, the branches targeting $x$ and $y$ respectively will coincide up to some time when they will get disconnected by the CPI. \cite[Theorem 1.1]{MSW} shows that this branching structure is described by $\Xbf^{\alpha}$ with $\alpha=\frac{\kappa'}{4}$. This is a quantum analogue of the set-up in \cite{BBCK}.

We stress that our result in \cref{thm: mainLQG} corresponds informally to the boundary case $\gamma\to\sqrt{8/3}$ in the latter work \cite{MSW}. Indeed, the conformal loop ensemble {degenerates} 
when $\kappa=\frac83$, leading to our model.

\begin{figure}[ht]
  \medskip
  \centering
  \subfigure[]{\includegraphics[scale=0.5,page=2]{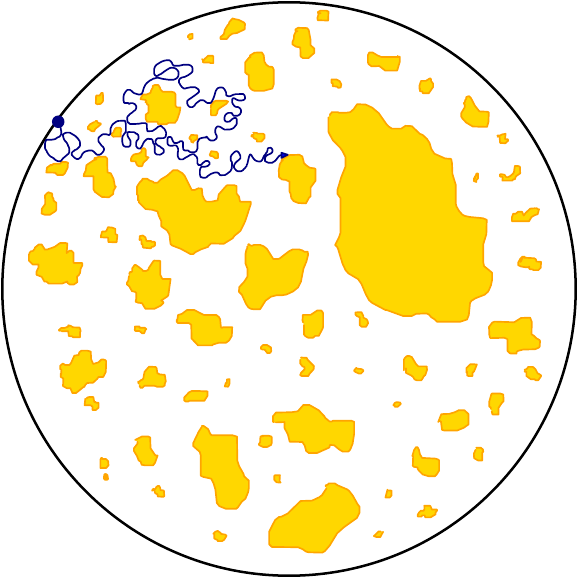}} 
  \hspace{0.5cm}
  \subfigure[]{\includegraphics[scale=0.5,page=3]{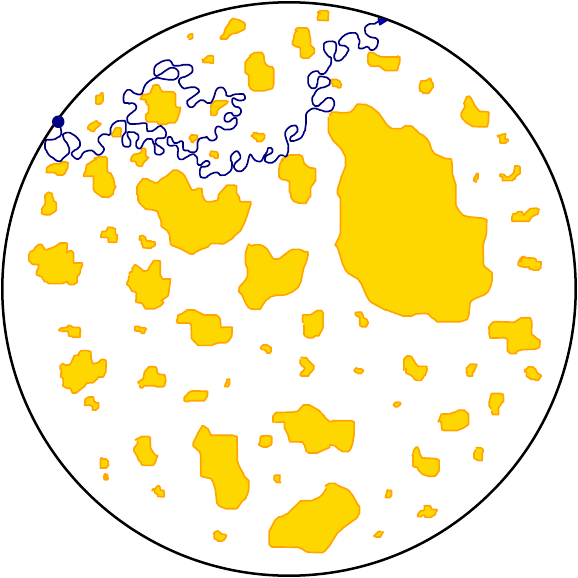}} 
  \hspace{0.5cm}
  \subfigure[]{\includegraphics[scale=0.5,page=5]{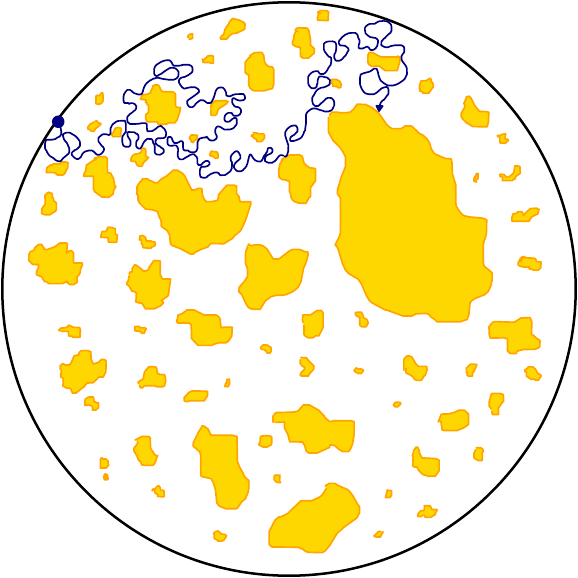}} 
  \caption{Different situations when the branch towards $z$ has a jump in \cite{MSW}: (a) the CPI discovers a new $\cle_{\kappa}$ loop; (b) the CPI hits the boundary {of $\D$} (or itself); (c) the CPI hits a previously visited $\cle_{\kappa}$ loop. The first case corresponds to positive jumps, (b) and (c) to negative jumps.}
  \label{fig: CLE on LQG MSW}
\end{figure}


\paragraph{Open questions.} 
It is interesting to note that our setup also makes sense for general $\gamma$, as presented in \cref{sec: intro LQG}.  
It would be interesting to describe the branching structures obtained in that case. We emphasise that in order to hope for a Markovian exploration for general $\gamma$, one needs to record not only the total boundary length, but the pair of left/right boundary lengths. We conjecture that the process obtained by recording the left and right boundary lengths of the branches toward all the points of the $\gamma$--quantum disc is some two-dimensional version of a growth-fragmentation process {or self-similar Markov tree \cite{bertoin2024self}}.
{We leave this to future work.}

The present paper also raises many questions of independent interest for self-similar Markov processes. For example, it is not clear whether one can construct other (spectrally positive) $\alpha$--stable processes conditioned to stay positive using a variant of the two-dimensional construction presented in \cref{prop: intro pathwise 3/2 stable cond}.

\subsection{Outline of the article}

In \cref{sec: prelims} we provide some necessary background on positive self-similar Markov processes and L\'{e}vy processes, Poisson point processes, growth fragmentations, and the connection between Brownian motion and SLE-decorated Liouville quantum gravity (the so-called mating of trees encoding). In \cref{sec: general cones} we define and study forward and backward Brownian cone exursions with general parameter, which involves formulating and proving a \emph{Bismut} description for a weighted version of the backward excursion law. We then move on in \cref{sec: special cones} to focus on the case $\gamma=\sqrt{8/3}$ and prove several special properties including an explicit joint law for the displacement and duration of an excursion, as well as a target invariance property. In particular, 
\cref{prop: intro law duration} and \cref{cor: intro law area} are proved in \cref{sec: joint law}, while \cref{prop: intro target inv cone} and \cref{cor: intro target inv lqg} are proved in \cref{sec: target invariance}. \cref{prop: intro pathwise 3/2 stable cond} is also proved in \cref{sec: BM construction conditioning}. In the final section, we move on to questions concerning the growth fragmentation process and the proof of our main theorem. We start by describing a special branch of the growth fragmentation $\mathbf{Z}$ defined in \eqref{eq: GF cones}, which is targeted towards a uniformly chosen point in the excursion. Namely, \cref{prop: law unif intro} and \cref{cor: law unif intro} are proved in \cref{sec: law unif}. Transforming to the law of the locally largest fragment in \cref{sec: law loc largest}, \cref{thm: main} is then proved in {\cref{sec:proof of main thm}}. \cref{thm: mainLQG} follows directly from \cref{thm: main} as explained above. Finally \cref{thm: martingale} is proved in \cref{sec: few prop of Zbf}.

\paragraph{Acknowledgements.} 

{We thank Élie Aïdékon, Andreas Kyprianou and Juan Carlos Pardo for helpful discussions, and Jean-François Le Gall and Armand Riera for their interest.} 
The research of E.P.\,is supported by UKRI Future Leader's Fellowship MR/W008513. W.D.S acknowledges the support of the Austrian Science Fund (FWF) grant on “Emergent branching structures in random geometry” (DOI: 10.55776/ESP534).

\section*{Index of notation}
\begin{flushleft}
  \begin{tabular}{|r|l|}
    \hline 
    $\theta$& parameter in $(\pi/2,\pi)$ \\
    $\gamma=\sqrt{4\theta/\pi}$ & parameter in $(\sqrt{2},2)$\\
    $\kappa=\gamma^2$ & parameter in $(2,4)$ \\
    $\kappa'=16/\kappa$ & parameter in $(4,8)$ \\
    ${\mathbbm{a}^2=2/\sin(\theta)}$ & mating of trees variance \\
    $\alpha=\pi/\theta=\kappa'/4$ & parameter in $(1,2)$ \\
    $\R_+$ & $[0,\infty)$ \\
    $\R_+^*$ & $(0,\infty)$ \\
    $\mathcal{C}_\theta=\{z\in \mathbb{C}: \arg(z)\in (0,\theta)\}$ &  cone of angle $\theta$ \\
    $W$ & correlated Brownian motion in $\C$,  as in \eqref{eq: cov MoT} \\
    (forward) cone-free time {$s$} for $W$ & $\nexists t\in [0,s)$ s.t.\, $W_r\in W_t+(\R_+^*)^2$ for $r\in (t,s]$ \\
    backward cone time {$t$} for $W$ &  $t>0$ s.t.\,$W_{{r}}\in W_t+(\R_+^*)^2$ for all ${r}\in [0,t)$ \\
    $\ell_\theta$ & local time on the set of {(forward)} cone-free times\\
    $\uptau_\theta$ & inverse of $\ell_\theta$ \\
    $\mathfrak{l}_\theta$ & local time on the set of backward cone times\\
    $\mathfrak{t}_\theta$ & inverse of $\mathfrak{l}_\theta$ \\
      $\varsigma_\theta^t$ & total local time towards $t\in(0,\zeta)$, see \cref{sec: intro main result cones} \\
      $\lozenge$ & cemetery state \\
      $E=\{e:[0,\zeta(e)]\to \C ; e(\zeta(e))=0\}\cup \{\lozenge\}$ & functions $e$ with finite duration $\zeta(e)$ vanishing at $\zeta(e)$ \\
      $(\mathbbm{e}_\theta(s); s>0)$ & PPP of forward cone excursions for $W$ \\
      $\mathbbm{n}_\theta$ & intensity measure of $(\mathbbm{e}_\theta(s); s>0)$ \\
      $(\mathfrak{e}_\theta(s); s>0)$ & PPP of backward cone excursions for $W$ \\
      $\mathfrak{n}_\theta$ & intensity measure of $(\mathfrak{e}_\theta(s); s>0)$ \\
      $\overline{\mathfrak{n}}_\theta$ & measure on $\R_+ \times E$, see \eqref{thm: Bismut} \\
    {$P^z_\theta$}, $z\in(\R_+^{*})^2$ & law $\nlg_{\theta}$ conditioned on ${e(0)}=z$  \\
    {$P^z_\theta$}, $z\in \partial \R_+^{2} \setminus \{0\}$ & law $\ndms_{\theta}$ conditioned on {$e(0)=z$}  \\
    $\mathbf{Z}$ & growth-fragmentation  in a cone excursion, see \eqref{eq: GF cones} \\
   $Z^t$ & size of fragment towards $t$, see \eqref{eq: cal Z and Z^t} \\
    $(L,R)$ & left/right quantum boundary length process in LQG \\
    $\xi^{\uparrow}$ & Lévy process $\xi$ conditioned to stay positive \\
    $\xi^{\searrow}$ &$\xi$ conditioned to be absorbed continuously at $0$ \\
    $\mathcal{U}$ & Ulam tree \\
    $\mathbf{X}^{\alpha}$ & growth-fragmentation process defined in \cref{sec: GF process} \\
    
    $\mathbb{P}_{x,y}$ & law 
     $(W,W')$ independent, $W_0=(0,0)$ $W'_0=(x,y)$  \\
    $Q_L$ & averaged version of $\mathbb{P}_{x,y}$ when $x+y=L$, see \eqref{eq: Q_L def} \\
    \hline
  \end{tabular}
\end{flushleft}
\newpage


\section{Preliminaries}\label{sec: prelims}

Throughout the article, $\theta\in (\pi/2,\pi)$ is a parameter, and we define several other quantities which will always implicitly be associated with $\theta$, namely:
\begin{equation}
  \label{eq: notation}
  \alpha := \frac{\pi}{\theta}\in (1,2) \quad ; \quad \gamma:=\sqrt{\frac{4\theta}{\pi}}\in(\sqrt{2},2) \quad ; \quad \kappa=\gamma^2\in (2,4)\quad \text{and} \quad \kappa'=\frac{16}{\kappa}=\frac{4\pi}{\theta}\in (4,8).
\end{equation}



\subsection{Positive self-similar Markov processes and Lévy processes} 
\label{sec: pssMp and Levy}
We set some notation and conventions for the self-similar Markov processes and Lévy processes showing up in this work.

\paragraph{Lévy processes.} A (killed) Lévy process is a càdlàg process $\xi = (\xi(t),t\ge 0)$ with independent and stationary increments, starting from $\xi(0)=0$ and possibly killed after an independent exponential time. Its (potentially infinite) Laplace exponent is defined as
\[
  \Psi(q) :=  \log \Eb [\mathrm{e}^{q\xi(1)}], \quad q\in \R,
\]
and satisfies $\Eb [\mathrm{e}^{q\xi(t)}] = \mathrm{e}^{t\Psi(q)}$ for all $t\ge 0$ and $q\in \R$. We henceforth assume that $\Psi(q)<\infty$ on an open interval of $q\in \R$. By a version of the Lévy-Khintchine theorem, there exists a \emph{unique} triplet $(a,\sigma,\Lambda)$ with $a\in \R$, $\sigma\ge 0$ and $\Lambda$ a measure satisfying $\int_{\R} (1\wedge y^2) \Lambda(\mathrm{d}y)$, such that $\Psi$ can be written 
\begin{equation} \label{eq: Levy-Khintchine def}
  \Psi(q) 
  =
  -\mathrm{k} + aq + \frac12 \sigma q^2 + \int_{-\infty}^\infty (\mathrm{e}^{qy}-1-q(\mathrm{e}^y-1)\mathds{1}_{y\le 1}) \Lambda(\mathrm{d}y),
\end{equation}
whenever this makes sense. The measure $\Lambda$ is called the \emph{Lévy measure} of $\xi$. In the sequel, we shall sometimes omit the cutoff in the indicator appearing in \eqref{eq: Levy-Khintchine def} when the Lévy measure satisfies $\int_{-\infty}^{\infty} \mathrm{e}^{y} \Lambda(\mathrm{d}y)<\infty$ (this has the effect of changing $a$). Some of the Lévy processes considered in this paper are \textbf{spectrally positive} (\textit{resp.} \textbf{negative}), meaning that they almost surely have no negative (\textit{resp.} positive) jumps -- equivalently $\Lambda(-\infty,0)=0$ (\textit{resp.} $\Lambda(0,\infty)=0$). A \textbf{subordinator} is a non-decreasing Lévy process.

A particularly important class of Lévy processes is that of \textbf{stable} Lévy processes. For $\alpha\in (0,2)$, a Lévy process is called $\alpha$--stable if for all $c>0$, 
\[
  (c\xi(c^{-\alpha}t), t\ge 0)  
  \overset{\text{d}}{=}
  \xi.
\]
The corresponding $\alpha$ is called the \emph{index} of $\xi$. The Lévy measure of a spectrally positive $\alpha$--stable process is of the form $\Lambda(\mathrm{d}x) = c_{\Lambda} x^{-(1+\alpha)} \mathds{1}_{x>0} \mathrm{d}x$ for some constant $c_\Lambda>0$. Its Laplace exponent is then given by the formula (see \cite[Example 46.7]{Sato}):
\begin{equation} \label{eq: Psi stable >0}
  \Psi_\alpha(-q) =   c_{\Lambda} \Gamma(-\alpha) q^\alpha, \quad q\ge 0.
\end{equation}
We refer to \cite{Ber,kyprianou2014fluctuations,KP} for more on Lévy processes.

\paragraph{Positive self-similar Markov processes.} Let $X$ be a regular Feller process with values in $\R_+$, and denote by $\Pb_x$ its law starting from $x$. Then $X$ is said to be a positive self-similar Markov process with index $\alpha>0$ if for all $c,x>0$, under $\Pb_x$, $(cX(c^{-\alpha}t), t\ge 0)$ has law $\Pb_{cx}$. 

An important property of this class of processes is that they are in bijection with Lévy processes up to killing. More precisely, let $T_0$ the hitting time of $0$ by $X$. Then any such process $X$ can be represented under $\Pb_x$ through the Lamperti representation \cite{Lam} as
\begin{equation} \label{eq: lamperti}
  X(t) 
  =
  x\mathrm{e}^{\xi(\tau(x^{-\alpha}t))},  \quad t<T_0,
\end{equation}
where $\xi$ is a Lévy process and 
\begin{equation} \label{eq: Lamperti time change}
  \tau(t) 
  :=
  \inf \{s\ge 0, \; \int_0^s \mathrm{e}^{\alpha \xi(u)} \mathrm{d}u > t\}.
\end{equation}

\paragraph{Conditioned stable processes.} It will be useful to define a couple
of conditionings of the stable process $\xi$. We start with the process $\xi$
\textbf{conditioned to stay positive}, denoted $\xi^{\uparrow}$. We focus on
the case when $\xi$ is a spectrally positive $\alpha$--stable process with
$\alpha\in(1,2)$ (although only $\alpha=3/2$ is relevant to our work), and
write $\xi^\dagger$ for the process $\xi$ killed upon entering $(-\infty,0)$.
We refer to \cite{chaumont1996conditionings,CC} for a more complete exposition.
One way to make sense of $\xi^{\uparrow}$ is as the Doob
$h^{\uparrow}$--transform of $\xi^\dagger$ with $h^{\uparrow}(x):=x$. Another
way is to write down the generator of $\xi^{\uparrow}$, which is {according to \cite[Equation (3.8)]{CC}, after simplification,}
\begin{equation} \label{eq: generator conditioned stable}
  \mathscr{G}_\alpha f(y)
  :=
  \int_0^\infty (f(y+z)-f(y)-zf'(y)) \Lambda(\mathrm{d}z) 
  + \frac{1}{y} \int_0^\infty (f(y+z)-f(y)) z\Lambda(\mathrm{d}z),
  \quad f\in \text{Dom}(\mathscr{G}_\alpha),
\end{equation}
where $\text{Dom}(\mathscr{G}_\alpha)$ {is the domain of the generator $\mathscr{G}_\alpha$ and contains}  
\[\{f:[0,\infty] \to \R, \; f, xf' \text{ and } x^2f'' \text{ are continuous on } [0,\infty]\}.\]
Finally, one can describe $\xi^{\uparrow}$ as a positive self-similar Markov process with index $\alpha$. Its Lamperti exponent is then given by 
\begin{equation} \label{eq: Lamperti xi uparrow}
  \xi^\uparrow(t) 
  =
  x\mathrm{e}^{\widetilde{\xi}(\widetilde{\tau}(x^{-\alpha}t))},  
\end{equation}
where the Lévy measure of $\widetilde{\xi}$ is
\[
  \widetilde{\Lambda}(\mathrm{d}x)
  :=
  c_{\Lambda} \frac{\mathrm{e}^{2y}}{(\mathrm{e}^y-1)^{\alpha+1}} \mathrm{d}y,
\]
and its Laplace exponent can be expressed \cite[Theorem 1]{kuznetsov2013fluctuations} as
\[
  \widetilde{\Psi}(z)  
  :=
  \frac{\pi c_{\Lambda}}{\Gamma(1+\alpha)\sin(\pi\alpha)} \cdot \frac{(1+z) \Gamma(\alpha-1-z)}{\Gamma(-z)}.
\]

Another conditioning will appear in this work, in the context of the spectrally negative $\alpha$--stable process $\xi$ with $\alpha\in(1,2)$, namely the process \textbf{conditioned to be absorbed continuously} at $0$. This  process $\xi^{\searrow}$ can be constructed as a Doob $h^{\searrow}$--transform of the killed process $\xi^{\dagger}$, with $h^{\searrow}(x) := x^{\alpha-2}$. As for the above process $\xi^{\uparrow}$, there are alternative descriptions of $\xi^{\searrow}$. For future reference, we mention that the Lamperti representation of $\xi^{\searrow}$ is given as in \eqref{eq: Lamperti xi uparrow}, where $\widetilde{\xi}$ has Laplace exponent 
\begin{equation} \label{eq: Laplace xi searrow}
  \widetilde{\Psi}(q) 
  :=
  a^{\searrow}q + c_{\Lambda} \int_{-\infty}^0 (\mathrm{e}^{qy}-1-q(\mathrm{e}^y-1)) \frac{\mathrm{e}^{(\alpha-1)y}\mathrm{d}y}{(1-\mathrm{e}^{y})^{5/2}},
\end{equation}
with\footnote{We emphasise that \cite[Equation (23)]{CC} contains a typo: the constant $c_-$ should read $(-c_-)$.} $a^{\searrow} := \frac{c_{\Lambda}}{\alpha-1} - c_{\Lambda} \int_0^1 \frac{(1-x)^{\alpha-2}-1}{x^{\alpha}} \mathrm{d}x$. The above expression can be given a closed form using \cite[Theorem 1]{kuznetsov2013fluctuations}, but we will not need this.


\subsection{Poisson point processes and the compensation formula}
We provide a minimal toolbox on Poisson point processes for completeness. A few measurability issues are swept under the carpet, we refer to \cite{RY} for details.

\paragraph{Poisson point processes.}
We begin by recalling some basic definitions. Let $(\Omega,\mathscr{F}, (\mathscr{F}_t),\Pb)$ be a probability space and $(E, \mathscr{E})$ be a measurable space. By convention, we add  a cemetery point $\lozenge$ to $E$ without changing notation.
\begin{Def}[Point processes]
  A \textbf{point process} is a process $e = (e_t,t>0)$ with values in $(E_{},\mathscr{E}_{})$ such that:
  \begin{itemize}
    \item[\emph{(i)}] The map $(t,\omega) \mapsto e_t(\omega)$ is $\mathscr{B}((0,\infty)) \otimes \mathscr{F}$--measurable.
    \item[\emph{(ii)}] Almost surely, the set $D := \{t>0, \; e_t(\omega) \neq \lozenge\}$ is countable.
  \end{itemize} 
\end{Def}
Property (ii) above enables us in particular to consider sums over $s\in D$ of positive elements $f(s)$; it will be convenient to use the slightly abusive notation $\sum_{s>0} f(s)$ for such sums. An important quantity for point processes is the counting function
\[
  N^{X}_{s,t} 
  :=
  \sum_{s<r\le t} \mathds{1}_{\{e_s\in X\}},
  \quad 0\le s<t \text{ and } X\in \mathscr{E}.  
\]
The point process $e$ is called \textbf{$\sigma$--discrete} if there exists an exhaustion $(E_n)$ of $E$ such that almost surely, for all $n$, $N^{E_n}_{0,t}$ is finite for all $t$.

\begin{Def}[Poisson point processes]
  \label{def: ppp}
  An $(\mathscr{F}_t$)--\textbf{Poisson point process} is a $\sigma$--discrete point process $e$ such that:
  \begin{itemize}
    \item[\emph{(i)}] $e$ is $\mathscr{F}_t$--adapted.
    \item[\emph{(ii)}] For $s,t>0$ and $X\in \mathscr{E}$, the conditional law of $N^X_{s,s+t}$ given $\mathscr{F}_s$ is that of $N^X_{0,t}$.
  \end{itemize}  
\end{Def}
The examples of Poisson point processes we have in mind for application are the (cone) excursion processes that will be defined in \cref{sec: general cones}. We conclude with another definition.

\begin{Def}[Intensity measure]
  The quantity 
  \[
    n(X)
    :=
    \frac{1}{t} \Eb[N_{0,t}^X], 
    \quad X\in \mathscr{E},   
  \]
  is independent of $t$ and defines a $\sigma$--finite measure on $\mathscr{E}$, called the \textbf{intensity measure} of the Poisson point process $e$.
\end{Def}

\paragraph{The compensation formula.} We devote a paragraph to the following claim, which expresses a key formula for Poisson point processes allowing to compute sums over the Poisson point process. This formula goes by different names in the literature, such as the Master formula, the compensation formula or Campbell's formula. We stick to the former terminology in the sequel.
\begin{Prop}[Compensation formula]
  Let $H: \R_+ \times \Omega \times E_{} \to \R_+$ an $(\mathscr{F}_t)$--predictable process, with $H(t,\omega,\lozenge) = 0$ for all $t,\omega$. Then
  \[
    \Eb\bigg[ \sum_{s>0} H(s,\omega,e_s(\omega))\bigg]  
    =
    \Eb\bigg[ \int_0^\infty \mathrm{d}s \int_E H(s,\omega, e) n(\mathrm{d}e)\bigg].
  \]
\end{Prop}

\subsection{Growth-fragmentation processes}
\label{sec: GF process}

\paragraph{Construction.}
We recall from \cite{Ber-GF} the definition of growth-fragmentation processes. The building block is a positive self-similar process $X$, in the sense of \cref{sec: pssMp and Levy}. We assume that $X$ is either absorbed at a cemetery point $\partial$ after a finite time $\zeta$ or converges to $0$ as $t\rightarrow\infty$ under $\Pb_x$ for all $x$. We further write $\Delta X(r) = X(r)-X(r^-)$ for the jump of $X$ at time $r$.

One can define the \emph{cell system} driven by $Z$ as follows. We use the Ulam tree $\mathcal{U} = \bigcup_{i=0}^{\infty} \N^i$, where $\N=\{1,2,\ldots\}$, to encode the genealogy of the cells (we write $\N^0 = \{\varnothing\}$, and $\varnothing$ is called the common ancestor). A node $u\in \mathcal{U}$ is a list $(u_1,\ldots, u_i)$ of positive integers where $|u|=i$ is the \emph{generation} of $u$. The children of $u$ are the lists in $\N^{i+1}$ of the form $(u_1,\ldots,u_i,k)$, with $k\in \N$. 
To define the cell system $\mathcal{X} = (\mathcal{X}_u, u\in \mathcal{U})$ driven by $X$, we first define a copy $\mathcal{X}_{\varnothing}$ of $X$, started from some initial mass $x>0$, and set $b_{\varnothing}=0$. Now record all the \emph{negative} jumps of $\mathcal{X}_{\varnothing}$. By our assumptions on the asymptotic behaviour of $X$, we may rank the sequence of these jump sizes and times $(x_1,\beta_1), (x_2,\beta_2),\ldots$ of $-\mathcal{X}_{\varnothing}$ by descending order of the $x_i$'s. Conditionally on this sequence, we define independent copies $\mathcal{X}_i, i\in\N,$ of $X$, where each $\mathcal{X}_i$ starts from $x_i$. We also set $b_i = b_{\varnothing}+\beta_i$ for the birth time of particle $i\in\N$. This constructs the first generation of the cell system. By recursion, one defines the $n$-th generation given generations $1,\ldots,n-1$ in the same way. In short, the cell labelled by $u=(u_1,\ldots, u_n)\in\N^n$ is born from $u'=(u_1,\ldots, u_{n-1})\in\N^{n-1}$ at time $b_u=b_{u'}+\beta_{u_n}$, where $\beta_{u_n}$ is the time of the $u_n$-th largest jump of $\mathcal{X}_{u'}$, and conditionally on $\Delta \mathcal{X}_{u'}(\beta_{u_n})=-y$, $\mathcal{X}_u$ is a copy of $X$ under $\Pb_y$ and is independent of the other daughter cells at generation $n$. We write $\zeta_u$ for the lifetime of the particle $u$. 

This uniquely defines the law $\mathcal{P}_x$ of the cell system $(\mathcal{X}_u(t),u\in \mathcal{U})$ driven by $X$ and started from $x>0$. The cell system is meant to describe the evolution of a population of cells $u$ with trait $\mathcal{X}_u(t)$ evolving in time $t$ and dividing in a binary way. 

The \textbf{growth-fragmentation process} $\Xbf$ is then defined as
\[
  \Xbf(t) := \left\{\left\{ \mathcal{X}_u(t-b_u), \; u\in\mathcal{U} \; \text{and} \; b_u\le t<b_u+\zeta_u \right\}\right\},
  \quad t\ge 0,
\]
where the double brackets denote multisets. In other words, at time $t\ge 0$, $\Xbf(t)$ is the collection of the sizes of all the cells alive in the system at time $t$.
Growth-fragmentation processes have been studied in general in \cite{BBCK}. They have been proved to arise in a large variety of contexts, from random planar maps \cite{BCK,BBCK} to Brownian geometry \cite{LGR} and Liouville quantum gravity \cite{MSW}, as well as excursion theory \cite{AD,DS,da2023spatial}. The present paper takes the latter viewpoint and reveals a growth-fragmentation process embedded in the $\frac{2\pi}{3}$--cone excursions of planar Brownian motion (or, alternatively, in $\sle_6$--explorations of the pure gravity quantum disc).

\paragraph{The family of growth-fragmentation processes $\Xbf^{\alpha}$.} In \cite{BBCK}, Bertoin, Budd, Curien and Kortchemski introduced an important family of growth-fragmentation processes, which is related to stable processes and shows up in the scaling limit of the peeling exploration of Boltzmann planar maps. This family is indexed by a self-similarity parameter $\alpha \in (\frac12,\frac32]$. Introduce the Lévy measure
\begin{equation} \label{eq: Pi alpha measure}
  \Pi_\alpha(\mathrm{d}y)
  :=
  \frac{\Gamma(\alpha+1)}{\pi}  \left(\frac{\mathrm{e}^{-\alpha y}}{(1-\mathrm{e}^y)^{\alpha+1}} \mathds{1}_{\{-\log(2)<y<0\}} + \sin(\pi(\alpha-1/2)) \cdot \frac{\mathrm{e}^{-\alpha y}}{(\mathrm{e}^y-1)^{\alpha+1}} \mathds{1}_{\{y>0\}}  \right) \mathrm{d}y,
\end{equation}
and the drift coefficient\footnote{The expression for $d_{\alpha}$ still makes sense at $\alpha=1$ after compensating the pole of $\Gamma$ at $0$ with the zero of the $\sin$ function.}
\[
  d_\alpha :=
  -\frac{\Gamma(2-\alpha)}{4\Gamma(2-2\alpha)\sin(\pi\alpha)} - \int_{-\log(2)}^0 (1-\mathrm{e}^y)^2 \Pi_\alpha(\mathrm{d}y) - \int_{0}^\infty (1-\mathrm{e}^y)^2 \Pi_\alpha(\mathrm{d}y).
\]
Let $X^{\alpha}$ the positive $\alpha$--self-similar Markov process given under $\Pb_x$ as the Lamperti transform \eqref{eq: lamperti}, where $\xi$ is the Lévy process with Laplace transform
\[
  \Phi_{\alpha}(q)  := 
  d_\alpha q + \int_{-\infty}^{\infty} (\mathrm{e}^{qy}-1- q(\mathrm{e}^y-1)) \Pi_\alpha(\mathrm{d}y).
\]
Observe that the Lévy measure \eqref{eq: Pi alpha measure} is carried on $(-\log(2),\infty)$, which in turn ensures that the process $X^{\alpha}$ never more than halves itself during a jump. This property corresponds to a ``canonical'' choice of driving cell process for the growth-fragmentation, called the \emph{locally largest evolution} in \cite{BBCK}. 
The process $\Xbf^{\alpha}$ is then defined to be the growth-fragmentation process driven by $X^{\alpha}$. Our paper is concerned with the process $\Xbf^{3/2}$.

\subsection{SLE, LQG and the mating-of-trees encoding} \label{sec: mot}
We recall for completeness in this section the definitions of the main objects of interest in this work: the Gaussian free field and quantum discs. 

\paragraph{Neumann Gaussian free field.} Let $\Hb=\{z\in \C: \Im(z)>0\}$ be the upper half-plane of $\C$. 
\begin{Def}(Neumann GFF on $\Hb$ with zero average on the upper unit semicircle).
\label{def: ngff}

\noindent Let $
\{ (h,f)\}_{f\in \mathcal{C}^{\infty}(\Hb)}
$
be the centered Gaussian process indexed by smooth compactly supported test functions $f\in \mathcal{C}_c^{\infty}(\Hb)$ which has covariance 
\[
\Cov((h,f)(h,g))=\int\int G^{\Hb}(x,y) f(x)g(y) \mathrm{d}x \mathrm{d}y, 
\]
for $f,g\in \mathcal{C}^\infty(\Hb)$, where
\[
G^{\Hb}(x,y)=-\log(|x-y|)-\log(|x-\bar{y}|)+2(\log(|x|)\vee 0)+2(\log(|y|) \vee 0).\
\]
  
\end{Def}

It is well-known (see for example the lectures notes \cite{BP}) that there exists a version of the Neumann GFF on $\Hb$ which almost surely defines a distribution on $\Hb$, and in fact, can be extended to define a distribution on  the boundary $\partial \Hb = \R$ as well. 

If one views the Neumann GFF  as a distribution \emph{modulo additive constants}, (i.e.\, as a continuous linear functional on the space of smooth compactly supported test functions $f$ on $D$ such that $\int_D f  = 0$) then the law of the Neumann GFF (modulo constants) is invariant under conformal automorphisms of $\Hb$. The Neumann GFF  can thus be defined (modulo constants) unambiguously in any simply connected domain $D\subset \C$ by taking the image of a Neumann GFF in $\Hb$ (modulo constants) under a conformal map from $\Hb$ to $D$. One can then define the Neumann GFF in $D$ as a random distribution by fixing the additive constant in some way. 

Now suppose that $\tilde{h}$ is a Neumann GFF in some domain $D$, with the additive constant  fixed in some way, and consider $h=\tilde{h}+g$, where $g$ is a random continuous function on $D$. Then one can define, for $\gamma\in(0,2)$, the so-called $\gamma$--LQG area measure by the limit in probability
\begin{equation} \label{eq: prelim LQG measure}
  \mu^{\gamma}_h(\mathrm{d}z) := \underset{\varepsilon\rightarrow 0}{\lim} \; \varepsilon^{\gamma^2/2} \mathrm{e}^{\gamma h_{\varepsilon}(z)} \mathrm{d}z,
\end{equation}
where $h_{\varepsilon}(z)$ denotes the average of the field $h$ on the circle with radius $\varepsilon$ centred at $z$ \cite{Dup-She}. Likewise, one can define the $\gamma$--LQG boundary length measure $\nu^{\gamma}_h$ of a segment of $\partial D$ where $g$ extends continuously. These constructions can be seen as instances of so-called \emph{Gaussian multiplicative chaos} \cite{Kah}, where one tries to construct measures defined as the exponential of a $\log$--correlated Gaussian field. 
It has also been proved in an important paper of Sheffield \cite{She-zip} that one can construct the \textbf{quantum boundary length measure} $\nu_h^{\gamma}$ of more general curves in $D$, including $\sle_{\kappa}$ or $\sle_{\kappa'}$ type curves that are independent of the field, with $\kappa=\gamma^2$ and $\kappa'=16/\gamma^2$.

A \textbf{$\gamma$--quantum surface} is an equivalence class of pairs $(D,h)$ with $D$ a simply connected domain and $h$ a distribution on $D$, where $(D,h)$ and $(D',h')$ are equivalent if $h$ and $h'$ satisfy the change of co-ordinates formula 
\begin{equation} \label{eq: change co-ordinates}
	h' := h \circ f^{-1} + Q \log |(f^{-1})'|, \quad Q:= \frac{\gamma}{2}+\frac{2}{\gamma},
\end{equation}
for some conformal map $f:D\rightarrow D'$. 
This definition of equivalence is chosen so that if $h$ is of the form $\tilde{h}+g$ (as above {\eqref{eq: prelim LQG measure}}) and $h'=h\circ f^{-1}+Q\log|(f^{-1})'|$, then $\mu^{\gamma}_h \circ f^{-1} = \mu^{\gamma}_{h'}$ and $\nu^{\gamma}_h \circ f^{-1} = \nu^{\gamma}_{h'}$ almost surely; the latter equality holding wherever the measures $\nu^\gamma$ are defined \cite{Dup-She,SW}. 

In reality, we will often want to consider quantum surfaces with some distinguished points on $D\cup\partial D$ or some extra decoration. In this case we introduce equivalent classes as in \eqref{eq: change co-ordinates}, except that we also require that $f$ maps the decorations of $D$ (\textit{e.g.} the marked points) onto those of $D'$.

Quantum surfaces conjecturally correspond to the scaling limits of {critical} random planar maps. In this setting, the measures $\mu^{\gamma}_h$ and $\nu_h^{\gamma}$ are expected to be the scaling limits of the counting measures on vertices and on boundary vertices respectively. This is already known for a few models of planar maps conformally embedded in the plane via the Tutte embedding \cite{GMS-Tutte} or the Cardy embedding \cite{HS}. For several models of planar maps chosen uniformly at random, this has also been proved in the so-called Gromov-Hausdorff-Prokhorov topology: see \cite{LG-BrownianMap, Mie,BM,GM-quadrangulations,BMR}. The present work focuses on the case $\gamma=\sqrt{8/3}$, sometimes called \emph{pure gravity}, and associated with uniform random planar maps. 
We denote $\mu_h=\mu_h^{\gamma}$ and $\nu_h=\nu_h^{\gamma}$ for $\gamma=\sqrt{8/3}$.

\paragraph{Quantum discs.} 
The \textbf{unit boundary length quantum disc} \cite{DMS} is a specific instance of quantum surface which has 
fixed quantum (i.e., measured using $\nu^\gamma$) boundary length equal to $1$. We will define the \emph{doubly marked} unit boundary length quantum disc and the \emph{singly marked} unit boundary length quantum disc, which are $\gamma$--quantum surfaces with two and one marked points respectively (see the discussion below \eqref{eq: change co-ordinates}). 

The strip $\mathcal{S}=\R\times i(0,\pi)$ turns out to be convenient as a parametrising domain. 
We start with the Neumann GFF on $\Hb$ from Definition \ref{def: ngff} and consider its image $\tilde{h}$ on $\mathcal{S}$ under the map $z\mapsto \log z$, which is a Neumann GFF on $\mathcal{S}$ with average $0$ on $(0,i\pi)$. A direct computation verifies that as a process in $s\in \R$, the average $X_s$ of $\tilde{h}$ on the vertical segment $s+(0,i\pi)$ is an almost surely continuous function. Moreover the difference $h^\dagger=\tilde{h}-X_{\Re(\cdot)}$ (which is a function with average $0$ on each vertical segment) is independent of $X$.
We now define a new field $h$ where we will keep the zero vertical average part $h^\dagger$ the same, but replace $X$ with a different continuous function; note that $h$ therefore has the law of a Neumann GFF on $\mathcal{S}$ plus a continuous function. More precisely, we define the random continuous function $Y$ on $\R$ by 
\begin{equation}\label{eq:Y}
	Y_t =\begin{cases} B_{2t}+(Q-\gamma) t & t\ge 0 \\ \widehat{B}_{-2t}+(Q-\gamma)(-t) & t<0 \end{cases}
\end{equation}
where $B, \widehat{B}$ are independent standard linear Brownian motions defined for $t\ge 0$, started from 0 and conditioned that $B_{2t} + (Q - \gamma)t$ (resp. $\widehat{B}_{2t} + (Q - \gamma)t$) is negative for all $t>0$. Then we set 
\[ h=h^\dagger + Y_{\Re(\cdot)}\]
where $h^\dagger$ has the law of $\tilde{h}-X_{\Re(\cdot)}$ as above and is sampled independently of $Y$.

Then it is possible to show \cite{DMS,HRV} (see also \cite{BP} for a proof), that $\nu_{h}^\gamma(\partial \mathcal{S})$ is almost surely finite and in fact has a finite moment of order $-(2/\gamma)(Q-\gamma)$. Therefore,
 it makes sense to weight {the law of $h$} by 
\[(\nu_{h}^\gamma(\partial S))^{-(2/\gamma)(Q-\gamma)},\]
and setting
\[
\widetilde{h}:={h}-(2/{\gamma}) \log \nu_{h}^\gamma(\partial \mathcal{S})
\]
defines a field with quantum boundary length almost surely equal to 1. 
The law of the field $\widetilde{h}$ under this reweighting, is what we define to be the law of the (doubly marked) \emph{quantum disc with boundary length $1$}. To define the law of the doubly marked quantum disc with boundary length $\ell$ we simply add the constant $(2/\gamma)\log(\ell)$ to the field. These random fields should be considered as the representatives of a random $\gamma$--quantum surface with two marked points in the sense discussed above \eqref{eq: change co-ordinates}, when parametrised by the strip $\mathcal{S}$ with the two marked points at $\pm \infty$. In other words, if $\widetilde{h}$ has the law described above, we want to view the doubly marked  quantum disc with boundary length $\ell$ as the random doubly marked quantum surface given by the equivalence class of $(\mathcal{S}, \widetilde{h}, -\infty, +\infty)$. 

One can define the law of the \emph{doubly marked quantum disc with left and right boundary lengths} $(\ell_L,\ell_R)$ as the regular conditional distribution of the doubly marked quantum disc with boundary length $\ell=\ell_L+\ell_R$ given $\nu_h^{\gamma}(\R\times \{\pi\}) = \ell_L$ and $\nu_h^{\gamma}(\R\times \{0\}) = \ell_R$.
Finally, one defines a \emph{singly marked} quantum disc, which is a $\gamma$--quantum surface with one marked point, by forgetting the second marked point: see \cite[Section 4.5]{DMS}. 


\paragraph{The mating-of-trees encoding.}
We now describe more precisely the connection between the SLE/LQG coupling and Brownian cone excursions. We will be interested in space-filling variants of $\sle$, for which we review the so-called \emph{peanosphere construction} of \cite{DMS}. We stress that although this paper is mostly concerned with the case $\gamma=\sqrt{8/3}$ and $\kappa'=6$, the results of the present section hold for general $\gamma$ and $\kappa'$ as in \eqref{eq: notation}.

When $\kappa'> 4$, the paper \cite{MS-space-filling} introduces a variant of $\sle_{\kappa'}$ \cite{Sch} which is space-filling. 
When $\kappa'\in(4,8)$, which is the regime where ordinary $\sle_{\kappa'}$ is \emph{not} space-filling the \emph{space-filling} variant can roughly be obtained by iteratively filling in bubbles that ordinary $\sle_{\kappa'}$ creates with space-filling loops, see \cite{MS-space-filling,GHS}. It can be defined on any simply connected domain $D$ from $x\in \partial D$ to $y\ne x$ on $\partial D$.  Following the mating-of-trees theorem of \cite{MS-mot,AG}, we will also consider a variant of space-filling $\sle_{\kappa'}$ called a \emph{counterclockwise space-filling $\sle_{\kappa'}$ loop} from $x$ to $x$ in $D$, see \cite{BG}. 
This is defined as the limit of the above space-filling curve when $y\rightarrow x$ in the counterclockwise direction (see \cite{BG}). 
A typical point $z$ on the boundary $\partial D$ is almost surely visited once by this curve, although some exceptional points are visited twice. An important fact is that counterclockwise space-filling $\sle_{\kappa'}$ from $x$ to $x$ in $D$ will visit all points of the first (non-exceptional) type in counterclockwise order starting from $x$. 

The \emph{mating-of-trees} theorem for the singly marked $\gamma$--quantum disc gives the law of the left/right boundary lengths in a counterclockwise space-filling $\sle_{\kappa'}$ exploration of the quantum disc, as follows. Recall the notion of cone excursion $P^z_\theta$ defined in \cref{sec: intro main result cones}.

\begin{Thm}[\!\!{\cite[Theorem 2.1]{MS-mot}},\cite{AG}] \label{thm: MoT complete}
  Let $\gamma\in (0,2)$ and $(\D,h,-i)$ be (an equivalence class representative of) a unit boundary length singly marked {$\gamma$}--quantum disc, with random quantum
  area $\mu_{h}^{\gamma}(\D)$. Consider a counterclockwise space-filling
  $\sle_{\kappa'}$ $\eta:[0,\mu_{h}^{\gamma}(\D)]\rightarrow \overline{\D}$ from
  $-i$ to $-i$, independent of $h$, but {re-parametrised} so that $\mu_h^\gamma(\eta([0,t]))=t$ for all $t$. Denote by
  $L_t$ and $R_t$ the change in quantum boundary lengths of the left and right
  sides of $\eta([0,t])$ relative to time $0$ as in \cref{fig: space-filling
  sle on disc}, normalised so that $(L_0,R_0)=(0,1)$. Then
  \[(L_t,R_t)_{t\in[0,\mu_{h}^{\gamma}(\D)]} \overset{(d)}{=}
P_\theta^{(0,1)}.\]
  Furthermore, the pair $(L_t,R_t)_{t\in[0,\mu_{h}^{\gamma}(\D)]}$ almost surely determines $(\D,h,\eta,-i)$ modulo the conformal change of co-ordinates \eqref{eq: change co-ordinates}. 
\end{Thm}

If $\gamma=\sqrt{8/3}$ and one instead considers a doubly marked quantum disc $(\D,h,-i,i)$ with boundary length $(\ell_L,\ell_R)$, and explores with an independent space-filling SLE$_{\kappa'}$ from $-i$ to $i$, then the  left/right boundary length process will have law $P_\theta^{(\ell_L,\ell_R)}$ (this is a variant of $P_\theta^z$ for $z\in (\R_+^*)^2$ that we introduce in \cref{prop: nlg disintegration}). In fact, if we take general $\gamma$ the same will hold if $(\D,h,-i,i)$ is a variant of the doubly marked quantum disc we have defined, but with $\gamma$ replaced by another parameter $\alpha(\gamma)$ in \eqref{eq:Y} (we will not use this fact).

We refer to \cite{DMS} and \cite{AG} for variants of this result for other types of quantum surfaces. 
\section{General properties of Brownian cone excursions}
\label{sec: general cones}

\subsection{Forward cone excursions of Brownian motion} \label{sec: forward cones}

Two different types of cone excursions will naturally come into play in our setting: \emph{forward} ones and \emph{backward} ones. We will define and review both of them here; in particular, we will be interested in the density of the \emph{end and start points} for these cone excursions. We stress that, although we are ultimately interested in the case when $\theta=\frac{2\pi}{3}$, we describe the results in full generality.   
We will use the following notation.
\begin{itemize}
  \item[$\bullet$] $W$ denotes correlated planar Brownian motion with correlations as in \eqref{eq: cov MoT}, started at $0$, and defined on a filtered probability space $(\Omega,\Fcal, (\Fcal_t,t\ge 0), \Pb)$. We extend the definition of  $\Fcal_T$ to stopping times $T$ in the usual way;
  \item $E$ is the set of functions $e$ defined on a finite interval $[0,\zeta(e)]$, with values in $\C$ and vanishing at $\zeta(e)$ (the $e$'s we will be interested in will actually remain in the quadrant $(\R_+^*)^2$, and start somewhere inside it or on the boundary). By convention, we add a cemetery function $\lozenge$ to $E$. We endow $E$ with the $\sigma$--field $\mathscr{E}$ generated by the co-ordinate mappings.
  \item  $\Ccal_{\theta}=\{z\in\C, \; \arg(z)\in(0,\theta)\}$ is the cone with apex angle $\theta$, so that $\Ccal_{\pi/2}=(\R_+^*)^2$.
\end{itemize}
\begin{Rk}\label{rk: cone transformation}
  Recall that the matrix $\mathbf{\Lambda}$ in \eqref{eq: correlation transformation} sends a pair $W$ of correlated Brownian motions with covariance structure \eqref{eq: cov MoT} onto a standard planar Brownian motion $\mathbf{\Lambda} \cdot W$. Moreover, $\mathbf{\Lambda}$ maps the quadrant $(\R_+^*)^2$ onto  $\Ccal_{\theta}:=\mathbf{\Lambda}((\R_+^*)^2)$ with apex angle $\theta$. This is the reason for the terminology \emph{cone excursions}. Everything that follows in this section could equivalently be phrased in terms of uncorrelated Brownian motions making excursions in cones of angle $\theta$, but the correlated Brownian framework is more convenient for us and is more directly linked to applications in Liouville quantum gravity.
\end{Rk}

\paragraph{Forward cone-free times of Brownian motion.} The first special type of point for Brownian motion that we will be interested in is already of particular importance in \cite{DMS}. 
For $u>0$, if there exists {$t< u$ such that $W_s\in W_t +(\R^*_+)^2$} for all $s\in(t,u]$, then (following \cite[Section 10.2]{DMS}) we say that $u$ is a \emph{pinched} time. The set of pinched times almost surely forms an open subset of $[0,\infty)$, and we can therefore express it as a countable disjoint union of open intervals. {Each of these intervals will correspond} to  \textbf{forward cone excursions}, {as we now define them}. 

If $u$ is \emph{not} a pinched time, we say that $u$ is (forward) \emph{cone-free} (ancestor-free in \cite{DMS}).
One can see that cone-free times form a regenerative set in the sense of \cite{Mai}, so that one can define a local time $(\ell_{\theta}(t), t\ge 0)$ supported on the set of cone-free times.

{For a fixed choice of local time}, its inverse $\uptau_{\theta}$,
\[
  \uptau_{\theta}(t) := \inf\{s\ge 0, \; \ell_{\theta}(s)>t\}, \quad t>0,
\]
naturally gives a way of labelling the forward cone excursions by $(\eb_{\theta}(s),s>0)$. 
{
  More precisely, we introduce
  \begin{itemize}
    \item[(i)] if $\uptau_{\theta}(s)>\uptau_{\theta}(s^-)$, then 
      \[
        \eb_{\theta}(s) : r \mapsto W(\uptau_{\theta}(s)-r)-W(\uptau_{\theta}(s^-)), \quad 0\le r\le \uptau_{\theta}(s)-\uptau_{\theta}(s^-),
      \]
    \item[(ii)] if $\uptau_{\theta}(s)=\uptau_{\theta}(s^-)$ then $\eb_{\theta}(s):=\lozenge$.
  \end{itemize}
  See \cref{fig: forward cone points}. We stress that the definition of $\eb_{\theta}(s)$ involves a time-reversal compared to the original time direction of $W$. This is because, for later purposes, we prefer to have excursions end at the apex. Note indeed that with these definitions, $\eb_\theta(s)\in E$ for all $s>0$, and any non-degenerate excursion $\eb_\theta(s)$ starts somewhere on the boundary $\partial \R_+^2 \setminus \{0\}$ of the quadrant.
}
The following proposition describes the structure of these forward cone excursions as a Poisson point process. We will not give the details, as this is done in \cite[Section 10.2]{DMS} (using a classical Brownian motion argument) and since we will further discuss the analogue for \emph{backward} cone excursions. 
\begin{Prop}
  The forward cone excursions $(\eb_{\theta}(s),s>0)$ form an $(\Fcal_{\uptau_{\theta}(s)}, s>0)$--Poisson point process in $(E,\mathscr{E})$. We denote its intensity measure by $\ndms_{\theta}$, {which is defined up to a multiplicative constant that depends on the choice of local time $\ell_\theta$.} 
\end{Prop}

\begin{figure}[ht]
  \bigskip
  \begin{center}
    \includegraphics[scale=1.1]{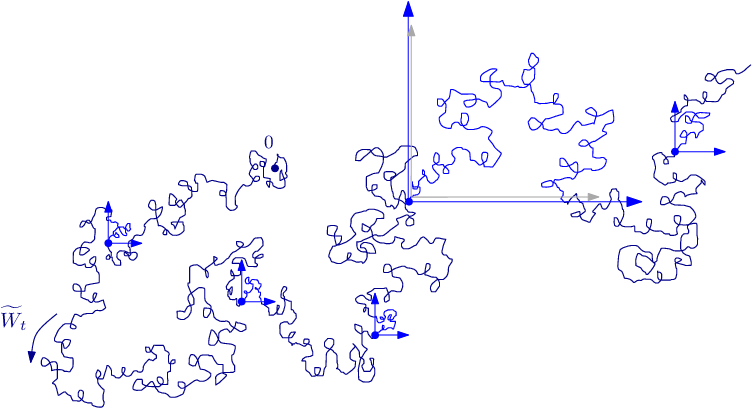}
  \end{center}
  \caption{Forward {cone-free} times of planar {(correlated)} Brownian motion. The reader should imagine accumulation of {cone-free} times, as suggested by the grey excursion in the middle. The forward cone excursions are shown in blue.}
  \label{fig: forward cone points}
\end{figure}

\noindent Notice that excursions under the measure $\ndms_{\theta}$ remain in $(\R_+^*)^2$, except for {when they start, \emph{on the boundary} of $\R_+^2$, and end, at the apex}.  It will be important for our purposes to establish the density of the endpoint under $\ndms_{\theta}$. This density was already described in the proof of \cite[Proposition 10.3]{DMS}, {again using purely Brownian techniques}. Let $\alpha=\frac{\pi}{\theta} \in (1,2).$

\begin{Prop} \label{prop: DMS disintegration}
  We have the following disintegration formula for $\ndms_{\theta}$:
  \begin{equation} \label{eq: disint ndms}
    \ndms_{\theta} = c_{\theta} \int_{\partial \R_+^2 \setminus \{0\}} \frac{\mathrm{d}z}{ |z|^{1+\alpha}} {P_{\theta}^z},
  \end{equation}
  where $\mathrm{d}z$ is the Lebesgue measure on $\partial \R_+^2 \setminus \{0\}$, the {$P_{\theta}^z$} are probability measures supported on  excursions ending at $z\in\partial\R_+^2$ 
  and $c_{\theta}$ is a constant depending on the choice of local time $\ell_\theta$. 
\end{Prop}

\begin{Rk} \label{rk: choice ell}
  In what follows, we fix a choice of local time $\ell_\theta$ so that $c_{\theta}=1$. We note that $\ell_\theta$ is a measurable function of the Brownian path $W$ (for example, it can be constructed as an almost sure limit of approximate local times, \cite[Section 10.2]{DMS}). 
\end{Rk}

\noindent The interpretation of the {law $P_{\theta}^z$ is that it corresponds} (\textit{via} \eqref{eq: disint ndms}) to the measure {$\ndms_{\theta}$} \emph{conditional} on the {start point} being $z$. As in the introduction, when $\theta=2\pi/3$ we simply write $P^z$.


Finally, it is natural to wonder what the law of $W$ is when it is time-changed by the inverse local time. 
The following result is from \cite[Proposition 1.13]{DMS}. Recall that $\alpha = \frac{\pi}{\theta} \in (1,2)$.

\begin{Thm} \label{thm: DMS stable}
  The time-changed process $\big(W_{\uptau_{\theta}(t)}, t\ge 0 \big) =
  \big(W^1_{\uptau_{\theta}(t)},W^2_{\uptau_{\theta}(t)}, t\ge 0 \big)$ evolves
  as a pair of independent spectrally positive $\alpha$--stable Lévy processes. 
  More precisely, the Laplace exponent of each co-ordinate is given by
  $\Psi(q) = \frac{4\sqrt{\pi}}{3} q^{\alpha}$ and their Lévy measure is
  $x^{-(1+\alpha)}\mathds{1}_{x>0}\mathrm{d}x$.
\end{Thm}
\begin{Rk} \label{rk: tau subordinator}
  Using the same arguments as in \cite[Proposition 1.13]{DMS}, one can prove
  that the process $\uptau_\theta$ is an $\frac{\alpha}{2}$--stable
  subordinator, hence has Lévy measure $\frac{\mathrm{d}t}{t^{1+\alpha/2}}$
  up to a multiplicative constant. This entails that $\ndms_{\theta}(\zeta>t)
  = c't^{-\alpha/2}$ for some $c'$.
  The value of $c'$ can be worked out from \cite{AG}, and in fact, the latter
  work gives
  a more precise result, since
  it describes the law of $\zeta$ under the conditioning $P_\theta^z$, for
  all $z\in \partial\R_+^2$. We will not need this (and we will actually
  provide an alternative derivation of an even stronger statement in \cref{sec: joint law}).
\end{Rk}

\cref{thm: DMS stable} is also related to \cref{prop: DMS disintegration}, since the density {$|z|^{-\alpha-1} \mathrm{d}z\mathds{1}_{\R_+^2}(z)$}  of the endpoint corresponds to the Lévy measure of the time-changed process.
{We emphasise that the results stated here from \cite{DMS} are proved using only classical arguments concerning Brownian motion and L\'{e}vy processes.}


\subsection{Backward cone excursions of Brownian motion} \label{sec: backward cones}
\paragraph{Backward cone times of Brownian motion.} The other type of cone times we want to discuss were introduced by Le Gall \cite{LG-cones} (actually, the cone times we describe here are obtained from those in \cite{LG-cones} after applying $\mathbf{\Lambda}^{-1}$ and a rotation of the plane). Call $t\in\R_+$ a \textbf{backward cone time} if $W_s\in 
W_t+(\R_+^*)^2$ for all $s\in [0,t)$ (\cref{fig: backward cone points}). In other words, $t$ is a backwards cone time if both co-ordinates reach a simultaneous running infimum at time $t$. 
We focus on the case $\theta\in(\frac{\pi}{2},\pi)$,
since the results of \cite{Bur-cones, Shi} prove that such times exist if, and only if, $\theta>\frac{\pi}{2}$. The set $H_{\theta}$ of backward cone times is regenerative, so that one can again define a local time $(\lfrak_{\theta}(s),\; s\ge 0)$ supported on $H_{\theta}$ \cite[Proposition~5.1]{LG-cones}.

Le Gall constructs {a choice of such local time}, see \cite[Sections 3 \& 5]{LG-cones}, by defining a measure 
\begin{equation} \label{eq: def LG local time}
  M_\theta := \frac12\lim_{\varepsilon\to 0}  \varepsilon^{-\alpha}\mathrm{Leb}(\cdot \cap \{w\in \mathbb{C}:  W_s\in w+(\R_+^*)^2 \; \forall s\le \inf\{r: W_r\in w+\mathbf{\Lambda}^{-1}(B(0,\varepsilon))\}\}), 
\end{equation}
on $\C$, and then setting $\lfrak_\theta(s)=M_\theta(W([0,s]))$ for $s>0$\footnote{In fact, Le Gall constructs the measure $\widetilde{M}_\theta  :=\lim_{\varepsilon\to 0} \varepsilon^{-\alpha}\mathrm{Leb}(\cdot \cap \{z\in \mathbb{C}:  \mathbf{\Lambda} W_s\in z+\Ccal_{\theta} \; \forall s\le \inf\{r:|\mathbf{\Lambda} W_r-z|\le \varepsilon\}\})$ and then defines $\lfrak_\theta(s)=\widetilde{M}_\theta(\mathbf{\Lambda} W([0,s]))$. This is the same as setting $\lfrak_\theta(s)=M_\theta(W([0,s]))$ if   $M_\theta$ is the pushforward of $\widetilde{M}_\theta$ by $\mathbf{\Lambda}$, which is how we have reached the precise definition of $M_\theta$ above.}. We will use this choice of local time in what follows.

Let $\tfrak_{\theta}$ be the inverse local time
\[
  \tfrak_{\theta}(t) := \inf\{s\ge 0, \; \lfrak_{\theta}(s)>t\}.
\]
In the same vein as for forward cone excursions, we can use $\tfrak_\theta$ to define the \emph{backward cone excursion} process. Recall that $E$ is the set of functions $e$ defined on a finite interval $[0,\zeta]$, with values in $\C$ and vanishing at $\zeta$ (with a cemetery function denoted by $\lozenge$).

\begin{figure}[ht]
  \bigskip
  \begin{center}
    \includegraphics[scale=0.9]{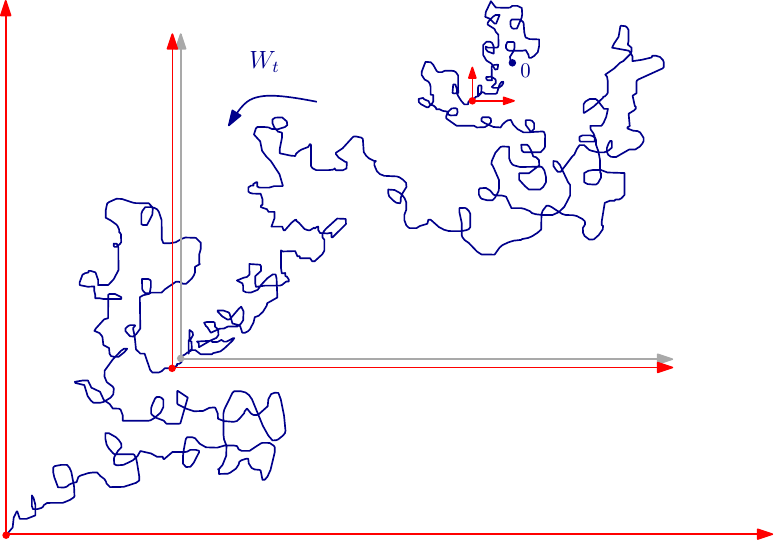}
  \end{center}
  \caption{Backward cone times of planar (correlated) Brownian motion. The reader should imagine accumulation of cone times, as suggested by the grey quadrant in the middle of the picture.}
  \label{fig: backward cone points}
\end{figure}

\begin{Def}\label{def: backward cone excursion process}
  The backward cone excursion process is the process $\efrak_{\theta} = (\efrak_{\theta}(s), s>0)$ on $(\Omega,\Fcal,\Pb)$ with values in $(E,\mathscr{E})$, defined as follows: 
  \begin{itemize}
    \item[(i)] if $\tfrak_{\theta}(s)>\tfrak_{\theta}(s^-)$, then 
      \[
        \efrak_{\theta}(s) : r \mapsto W(\tfrak_{\theta}(s^-)+r)-W(\tfrak_{\theta}(s)), \quad 0\le r\le \tfrak_{\theta}(s)-\tfrak_{\theta}(s^-),
      \]
    \item[(ii)] if $\tfrak_{\theta}(s)=\tfrak_{\theta}(s^-)$ then $\efrak_{\theta}(s):=\lozenge$.
  \end{itemize}
\end{Def}
\noindent This definition is made so that the cone excursions $\efrak_{\theta}(s)$
are paths in $\R_+^2$ which stay in $(\R_+^*)^2$ until ending at the apex (the origin),
and there is a non-degenerate cone excursion $\efrak_{\theta}(s)$ whenever $\lfrak_{\theta}$ has a constant stretch at time $s$. We claim that this defines a Poisson point process.
\begin{Prop} \label{prop: excursion PPP}
  The process $(\efrak_{\theta}(s),s>0)$ is an $(\Fcal_{\tfrak_{\theta}(s)}, s>0)$--Poisson point process {in $(E,\mathscr{E})$} with intensity measure {denoted by} $\nlg_{\theta}$.
\end{Prop}
\begin{proof}
  We check the properties listed in \cref{def: ppp}. 
  \begin{enumerate}
    \item Plainly, $(\efrak_{\theta}(s), s>0)$ is a point process in the sense of \cite[Definition XII.1.1]{RY} (the fact that there are at most countably many non-degenerate excursions can be seen as a consequence of the fact that the jump times of $\tfrak_{\theta}$ are at most countable).
    \item We check that $(\efrak_{\theta}(s), s>0)$ is $\sigma$--discrete \cite[Definition XII.1.2]{RY}. Let $E_n := \{e\in E, \, \zeta(e) > 1/n\}$, $n\ge 1$. Then $E=\bigcup_{n\ge 1}E_n$, and the $E_n$ are measurable. 
    {
      For a measurable subset $X$ of $E$, we introduce
  \begin{equation} \label{eq:N^X_t}
    N^{X}_t := \sum_{s\le t} \mathds{1}_{\efrak_{\theta}(s)\in X}, \quad t>0.
  \end{equation}
    }
    The counting functions $N^{E_n}_t$, $t>0$,
      are \textit{a.s.} finite random variables. Indeed, set $T_0:=0$ and $T_{k+1}:=\inf\{s>T_k, \; \tfrak_{\theta}(s)-\tfrak_{\theta}(s^-)>1/n\}$, $k\ge 0$. By definition, 
      \[
        N^{E_n}_t := \sum_{k\ge 1} \mathds{1}_{T_k\le t}, \quad t>0,
      \]
      and $N^{E_n}_t \le n\tfrak_{\theta}(t)$.
    \item The process $\efrak_{\theta}$ is clearly $(\Fcal_{\tfrak_{\theta}(s)})$--adapted.
    \item Finally, for any {measurable} subset $X$ of $E$, and $t,r> 0$, write 
      \[
        N^{X}_{(t,t+r]} := \sum_{t< s\le t+r} \mathds{1}_{\efrak_{\theta}(s)\in {X}}.
      \]
      Denote by $\Theta=(\Theta_u,u\ge 0)$ the shift operator defined on $E$ by $\Theta_u \circ e := e(u + \cdot)$ (by convention we extend $e$ to be $0$ outside $[0,\zeta]$).
      Since almost surely, for all $s$, $\Theta_u \circ \efrak_{\theta}(s) = \efrak_{\theta}(s+u)$ (see \cite[Section X.1]{RY}, and in particular the remark following Proposition X.1.3 on finite continuous additive functionals), by shifting the excursion process we get {for all $A\subset \N$,}
      \[
        \Pb\left(N^{X}_{(t,t+r]}\in A \, | \, \Fcal_{\tfrak_{\theta}(t)} \right) = \Pb\left(N^{X}_r \circ \Theta_{\tfrak_{\theta}(t)}\in A \, | \, \Fcal_{\tfrak_{\theta}(t)} \right),
      \]
      where $\Theta$ denotes the shift operator {and $N^X_r$ is as in \eqref{eq:N^X_t}}. Now by the strong Markov property of $W$, this is
      \[
        \Pb\left(N^{X}_{(t,t+r]}\in A \, | \, \Fcal_{\tfrak_{\theta}(t)} \right) = \Pb_{W(\tfrak_{\theta}(t))}\left(N^{X}_r \in A \right).
      \]
      But the excursion process $\efrak_{\theta}$ is by definition independent of the start point of $W$, hence finally
      \[
        \Pb\left(N^{X}_{(t,t+r]}\in A \, | \, \Fcal_{\tfrak_{\theta}(t)} \right) = \Pb\left(N^{X}_r \in A \right).
      \]
      This concludes the proof of the Poisson point process property.
  \end{enumerate}
\end{proof}

Finally, \cite[Theorem 5.2]{LG-cones} determines the law of $W$ time-changed by $\tfrak_{\theta}$ as follows.
The case $\theta=\pi$, which we do not consider in this paper, corresponds to
Spitzer's construction of the Cauchy process \cite{Spi}.
Recall again that $\alpha=\frac{\pi}{\theta} \in (1,2)$.
\begin{Thm} \label{thm: LG stable}
  The process $(W_{\tfrak_{\theta}(t)}, t\ge 0)$ is a stable Lévy process in the plane, with index $2-\alpha$. 
\end{Thm}

\noindent Note that $2-\alpha \in(0,1)$, so that backward cones define stable processes with indices in $(0,1)$, whereas forward cones are associated to stable processes with index $\alpha\in(1,2)$.

\begin{Rk}
  We note that \cite[Theorem 5.2]{LG-cones} also gives the law of $\tfrak_{\theta}$ as a stable subordinator with index $1-\frac{\alpha}{2}$.
\end{Rk}


At this point, we should emphasise that the structure of $W
\circ\tfrak_{\theta}$ is much more involved than {that of}
$W\circ\uptau_{\theta}$ appearing in \cref{thm: DMS stable}. The issue is that
there is \emph{no independence} between co-ordinate processes in the backward
cone times framework. Indeed, the backward cone excursions that are cut out in
$W$ go from the \emph{interior} of a quadrant to its apex, and therefore at a
jump time of $W \circ\tfrak_{\theta}$, both components jump simultaneously. In
particular, the Lévy measure of $-(W \circ\tfrak_{\theta})$ can be
written in polar co-ordinates as
\begin{equation} \label{eq: LG Lévy measure}
  \mathfrak{L}_{\theta}(\mathrm{d}r,\mathrm{d}\phi) = \frac{\mathrm{d}r}{r^{3-\alpha}} \cdot\mfrak_\theta(\mathrm{d}\phi),
\end{equation}
but the angular part $\mfrak_\theta$ of the measure does not seem to be known in general. {This is actually left as an open problem in \cite[Remark (ii), p613]{LG-cones}.} Although $\mfrak_\theta$ is not explicit, it is characterised (after applying $\mathbf{\Lambda}$ and a rotation) by formula (5.j) in \cite{LG-cones}. 

\subsection{Basic properties of the backward cone excursion measure \texorpdfstring{$\nlg_{\theta}$}{n\textunderscore theta}}
\label{sec: backward cone excursion}

We study  $\nlg_{\theta}$, in the general case when $\theta\in(\frac{\pi}{2},\pi]$, by establishing a sort of Markov property under $\nlg_{\theta}$, deriving the density of the {starting} point, and proving the convergence of the normalised backward cone excursion measure to the normalised forward one when the point is sent to the boundary. We will denote  a generic cone excursion by $e$, and write $\zeta$ for its duration.

\paragraph{The Markov property of $\nlg_{\theta}$.} One of the core properties of the classical one-dimensional Itô measure is its Markov property, see \cite[Theorem XII.4.1]{RY}. In this case, it roughly states that for $t>0$, on the event $t<\zeta$ and conditioned on $(e(s), s\le t)$, the law of the trajectory of $e$ from time $t$ onwards is that of an independent standard Brownian motion starting at $e(t)$, and killed upon reaching $0$. Of course, under $\nlg_{\theta}$, the statement is less straightforward, as there is some dependence on the past. Indeed, backward cone times are defined so that the \emph{whole} past trajectory is contained in a quadrant, hence ending the excursion should depend on the past even before $t$. The next result states that, loosely speaking, this is the only dependence.
\begin{Prop}[The Markov property under $\nlg_{\theta}$] \label{prop: markov nlg}
  Let $t>0$. On the event that $t<\zeta$, and conditioned on $(e(s)-e(0), 0\le s\le t)$, the law of $(e(t+s)-e(t), 0\le s \le \zeta-t)$ is that of an independent correlated Brownian motion $W$ as in \eqref{eq: cov MoT}, started at $0$, and stopped at the first simultaneous running infimum $I$ of $W$ such that $W^1_I\le \inf\{e^1(s)-e^1(t), 0 \le s \le t\}=\inf\{(e^1(s)-e^1(0)), 0 \le s \le t\}-(e^1(t)-e^1(0))$ and $W^2_I\le \inf\{e^2(s)-e^2(0), 0 \le s \le t\}-(e^2(t)-e^2(0))$.
\end{Prop}
In other words, the behaviour of $e$ after time $t$ is that of an independent correlated Brownian motion stopped at the first simultaneous running infimum ``below the past trajectory''.

\begin{Rk} \label{rk: strong markov nlg}
  By standard arguments, one can extend this description to the case of stopping times, thus proving a strong Markov property under $\nlg_\theta$.
\end{Rk}


\begin{proof}[Proof of \cref{prop: markov nlg}]
  The proof follows the lines of \cite[Theorem XII.4.1]{RY}. Denote by $\Theta^0=(\Theta^0_t,t\ge 0)$ the shift operator {defined by $\Theta^0_t \circ e := e(t + \cdot) - e(t)$ for $e\in E$ (by convention we extend $e$ to be $0$ outside $[0,\zeta]$)}. We want to prove that for all $A(t)\in \sigma(e(s)-e(0), 0\le s\le t)$ and $X\in \mathscr{E}$, 
  \begin{equation} \label{eq: markov prop}
    \nlg_{\theta}\left(A(t) \cap \{ \Theta^0_t \circ e \in  X\}\right)
    =
    \nlg_{\theta}\left(\mathds{1}_{A(t)} \cdot \Pb((W_s,\, s\le I)\in  X)\right),
  \end{equation}
  where $I$ is defined as in \cref{prop: markov nlg}. In particular, we stress that $I$ is averaged under both $\Pb$ and $\nlg_\theta$ in \eqref{eq: markov prop}.

  We will derive identity \eqref{eq: markov prop} from the Poisson point process structure of the excursion process in \cref{prop: excursion PPP}. Denote by $\efrak_{\theta}^{A(t)}$ the Poisson point process obtained by restriction of $\efrak_{\theta}$ to those excursions $e$ for which $A(t)$ occurs. Note that the intensity measure of this point process is finite (since such excursions must satisfy $\zeta>t$), and therefore we can consider its first jump time $S_1$. Now recall (for instance from \cite[Lemma XII.1.13]{RY}) the following classical identity:
  \begin{equation} \label{eq: PPP first time}
    \frac{\nlg_{\theta}\left(A(t)  \cap \{ \Theta_t^0 \circ e \in  X\}\right)}{\nlg_{\theta}(A(t))}
    =
    \Pb\big(\Theta_t^0 \circ \efrak_{\theta}^{A(t)}(S_1) \in X\big).
  \end{equation}
  We can write 
  \[
    \efrak_{\theta}^{A(t)}(S_1) = \left( W_{\tfrak_{\theta}(S_1^-)+r}-W_{\tfrak_{\theta}(S_1)}, \; 0\le r \le \tfrak_{\theta}(S_1)-\tfrak_{\theta}(S_1^-) \right).
  \]
  Since $S_1$ is a $(\Fcal_{\tfrak_{\theta}(s)}, s\ge 0)$--stopping time, $\tfrak_{\theta}(S_1^-)$ and $\tfrak_{\theta}(S_1)$ are $(\Fcal_s,s\ge 0)$--stopping times. For clarity, set $T=\tfrak_{\theta}(S_1^-)+t$. Then
  \[
    \Pb\big(\Theta_t^0 \circ \efrak_{\theta}^{A(t)}(S_1) \in X\big)
    =
    \Pb\big( (W_{T+r}-W_T, \, r\le J-T)\in X \big),
  \]
  where $J$ is the first simultaneous running infimum of $W$ after $T$ such that the corresponding quadrant also contains $(W_s, \, \tfrak_{\theta}(S_1^-) \le s\le T)$. By the strong Markov property of $W$ at time $T$, we can rewrite this as 
  \[
    \Pb\big(\Theta_t^0 \circ \efrak_{\theta}^{A(t)}(S_1) \in X\big)
    =
    \Eb\big[\Pb\big( (\widetilde{W}_{r}, \, r\le \tilde{I})\in X   \mid (W_s-W_{\tfrak_{\theta}(S_1^-)}, \tfrak_{\theta}(S_1^-) \le s \le T)\big)  \big],
  \]
  where $\widetilde{W}$ is an independent correlated Brownian motion started from $0$, and $\tilde{I}$ is the first simultaneous running infimum of $\widetilde{W}$ after $T$ which falls below the path $(W_s-W_T, \tfrak_{\theta}(S_1^-) \le s \le T)$.

  Coming back to \eqref{eq: PPP first time}, we proved that 
  \[
    \nlg_{\theta}\left(A(t)  \cap \{ \Theta_t^0 \circ e \in  X\}\right)
    =
    \nlg_{\theta}(A(t)) \cdot \Eb\big[\Pb\big( (\widetilde{W}_{r}, \, r\le \tilde{I})\in X   \mid (W_s-W_{\tfrak_{\theta}(S_1^-)}, \tfrak_{\theta}(S_1^-) \le s \le T)\big)  \big].
  \]
  The same argument entails that the law of $(e(s)-e(0), 0\le s\le t)$ under $\nlg_{\theta}(\;\cdot \mid A(t))$ is that of $(W_s-W_{\tfrak_{\theta}(S_1^-)}, \tfrak_{\theta}(S_1^-) \le s \le T)$. Therefore, we conclude that
  \[
    \nlg_{\theta}\left(A(t) \cap \{ \Theta^0_t \circ e \in  X\}\right)
    =
    \nlg_{\theta}\left(\mathds{1}_{A(t)} \cdot \Pb((W_s,\, s\le I)\in  X)\right),
  \]
  which is the desired Markov property \eqref{eq: markov prop}. 
\end{proof}

\paragraph{Density of the {start point} under $\nlg_{\theta}$, and the normalised backward cone excursion measure.}
We will want to relate the two types of cone excursions (forward and backward) in the limit when the start point is taken to the boundary. A straightforward consequence of the results of \cite{LG-cones} concerning the $(2-\alpha)$--stable process is that we can disintegrate the backward measure $\nlg_{\theta}$ over the start point. 

\begin{Prop} \label{prop: nlg disintegration}
  We have the following disintegration formula for $\nlg_{\theta}$ in polar co-ordinates:
  \begin{equation} \label{eq: disint nlg}
    \nlg_{\theta} = \int_0^{\infty} \frac{\mathrm{d}r}{r^{3-\alpha}} \int_{0}^{\pi/2}  \mfrak_\theta(\mathrm{d}\phi) P_{\theta}^{r\mathrm{e}^{i\phi}},
  \end{equation}
  where $\mfrak_\theta$ is the finite positive measure on $(0,\pi/2)$ which appears in \eqref{eq: LG Lévy measure}, and the $P^{r\mathrm{e}^{i\phi}}_{\theta}$ are probability measures supported on excursions $e\in E$ such that $e((0,\zeta)) \subset (\R_+^*)^2$ and $e(0) = r\mathrm{e}^{i\phi}$.
\end{Prop}
Note that  in \cref{sec: forward cones} we already defined the law $P_{\theta}^z$ for $z\in\partial \R_+^2\setminus \{0\}$ on the \emph{boundary}. Here we define  laws $P_{\theta}^z$ for $z\in(\R_+^*)^{2}$ in the \emph{interior} of the quadrant. This slight abuse of notation will be justified by \cref{prop: convergence measures} below.

\begin{proof}[Proof of \cref{prop: nlg disintegration}]
  By definition of $\nlg_{\theta}$ as the intensity measure of the Poisson point process of excursions, we know that for any measurable $A\in \C$ {and non-negative measurable function $f$ on $E$ such that $f(\lozenge)=0$},
  \[
    \nlg_{\theta}(\mathds{1}_{\{e(0)\in A\}} \cdot f(e)) 
    =
    \Eb\bigg[ \sum_{0<s\le 1} \mathds{1}_{\{\efrak_{\theta}(s)(0) \in A\}} \cdot f(\efrak_{\theta}(s))\bigg].
  \]
  It remains to notice that the quantities ${\efrak_{\theta}(s)(0)}$ correspond exactly to the jumps of the process $-W\circ \tfrak_{\theta}$. We now use \cref{thm: LG stable} borrowed from \cite{LG-cones} to express the above expectation. Indeed, we know that $-W\circ \tfrak_{\theta}$ is a $(2-\alpha)$--stable Lévy process in the plane, hence its Lévy measure has the form given by \eqref{eq: LG Lévy measure}. An application of the compensation formula then yields
  \begin{equation} \label{eq: proof disintegration}
    \nlg_{\theta}(\mathds{1}_{e(0)\in A} \cdot f(e)) 
    =
    \int_{(0,\infty)\times (0,\pi/2)}\mathfrak{L}_{\theta}(\mathrm{d}r,\mathrm{d}\phi) P_{\theta}^{r\mathrm{e}^{i\phi}}(f).
  \end{equation}
  Plugging \eqref{eq: LG Lévy measure} into \eqref{eq: proof disintegration}, we obtain the desired identity.
\end{proof}

\begin{Rk}
  The disintegration in \cref{prop: nlg disintegration} is not completely explicit since the measure $\mfrak_\theta$ is unknown. {Determining a closed-form expression for it was left as an open problem in Le Gall's work \cite[Remark (ii), p613]{LG-cones}.} In \cref{sec: joint law} we will describe it explicitly in the case when $\theta= \frac{2\pi}{3}$, {thereby solving this open problem in that case.}
\end{Rk}

We now want to relate the backward excursion measure $\nlg_{\theta}$ and the forward one $\ndms_{\theta}$ when the start point is taken to the boundary. 
For {$z\in \R_+^2$ and} $r>0$, let $B_+(z,r)$ be the intersection of the ball with radius $r$ around $z$ and the quadrant $\R_+^2$. Write $T_z(r)$ for the hitting time of $\partial B_+(z,r)$.

\begin{Prop} \label{prop: convergence measures}
  {Let $z\in\partial\R_+^2 \setminus \{0\}$}. Then {the probability measures $Q_\varepsilon := \nlg_{\theta}(\,\cdot\mid e(0)\in B_+(z,\varepsilon))$, $\varepsilon>0$, converge weakly as $\varepsilon \rightarrow 0$ to $P_{\theta}^z$.}
\end{Prop}

\begin{proof}
  Fix $z\in\partial \R_+^2\setminus \{0\}$ and write $Q_\varepsilon := \nlg_{\theta}( \cdot \mid e(0)\in B_+(z,\varepsilon))$. Let $\overline{E}$ be defined as $E$, but without the requirement that $e(\zeta(e))=0$. {We first claim that it} is enough to prove that for all bounded continuous function $f$ on $\overline{E}$, for all $r\in(0,|z|)$ and $t>0$,
  \begin{multline} \label{eq:cvg_Q_eps}
    \Eb^{Q_\varepsilon}\big[f(e(s+T_z(r))-e(T_z(r)), 0\le s\le t) \mathds{1}_{\{\zeta>t+T_z(r)\}}\big] \\
    \longrightarrow
    \Eb^{P_{\theta}^z}[f(e(s+T_z(r))-e(T_z(r)), 0\le s\le t)\mathds{1}_{\{\zeta>t+T_z(r)\}}], \quad \text{as } \varepsilon \to 0.
  \end{multline}
  {Indeed, assume that this convergence holds. By the Portmanteau theorem, in order to prove the week convergence of $Q_\varepsilon$ to $P_{\theta}^z$ as $\varepsilon \rightarrow 0$, we only need to consider bounded Lipschitz functions. But in that case, the left-hand side above converges to $\Eb^{Q_{\varepsilon}}[f(e(\cdot)-z)]$ as $r\to 0$ uniformly in $\varepsilon$, by Lipschitz continuity. On the other hand the right-hand side converges to $\Eb^{P_{\theta}^z}[f(e(\cdot)- z)]$ as $r\to 0$. This would conclude the proof of \cref{prop: convergence measures}.}

{It remains to prove the convergence in \eqref{eq:cvg_Q_eps}.}
  In what follows, $W$ denotes a correlated Brownian motion with covariances \eqref{eq: cov MoT}, started at $0$. Conditional on $(e(s), 0\le s\le T_z(r))$, we let $B_t$ the event that $(W_s, 0\le s\le t)$ does not have any backward cone point below the whole trajectory $(e(s)-e(T_z(r)), 0\le s\le T_z(r))$, and $A_t$ the event that $(W_s, 0\le s\le t)$ stays in the quadrant $-e(T_z(r))+\R_+^2$ rooted at $-e(T_z(r))$. Then by the strong Markov property at time $T_z(r)$ under $\nlg_{\theta}$ (\cref{prop: markov nlg} and \cref{rk: strong markov nlg}), 
  \begin{multline} \label{eq: strong Markov T_z(r)}
    \Eb^{Q_\varepsilon}\big[f(e(s+T_z(r))-e(T_z(r)), 0\le s\le t) \mathds{1}_{\{\zeta>t+T_z(r)\}}\big]  \\
    =
    \Eb^{Q_\varepsilon}\big[ \Eb[f(W_s, 0\le s\le t) \mathds{1}_{B_{t}} \mid A_{t}, (e(s), 0\le s\le T_z(r)) ]\big].
  \end{multline}
  Under $Q_{\varepsilon}$, the random variable $e(T_z(r))$ is bounded, hence the law of $e(T_z(r))$ is tight as $\varepsilon \to 0$. Therefore, we may consider a subsequential limit $X_r$, under $\Pb$. Now fix $\delta>0$.

  We claim that we can also remove the event $B_t$ in \eqref{eq: strong Markov T_z(r)} as $\varepsilon\to 0$. Indeed, $Q_\varepsilon$--almost surely,
  \[
    \Pb(B_t^{\text{c}}\mid A_{t}, (e(s), 0\le s\le T_z(r))) \to 0, \quad \text{as } \varepsilon\to 0.
  \]
  Hence by dominated convergence, we see that for $\varepsilon>0$ small enough,
  \begin{multline} \label{eq: removing B_t}
    \big|\Eb^{Q_{\varepsilon}}\big[f(e(s+T_z(r))-e(T_z(r)), 0\le s\le t) \mathds{1}_{\{\zeta>t+T_z(r)\}}\big]  \\
    -\Eb^{Q_{\varepsilon}}\big[ \Eb[f(W_s, 0\le s\le t) \mid A_{t}, (e(s), 0\le s\le T_z(r)) ]\big] \big|
    \le 
    \delta.
  \end{multline}
  Note that $\Eb[f(W_s, 0\le s\le t) \mid A_{t}, (e(s), 0\le s\le T_z(r)) ]=\Eb[f(W_s, 0\le s\le t) \mid A_{t}, e(T_z(r)) ]$ is now a (measurable and bounded) function of $e(T_z(r))$. Thus we can take a subsequential limit, yielding 
  \[
    \big|\Eb^{Q_{\varepsilon}}\big[ \Eb[f(W_s, 0\le s\le t) \mid A_{t}, (e(s), 0\le s\le T_z(r)) ]\big] \\
    -  \Eb\big[ \Eb[f(W_s, 0\le s\le t) \mid A^{X_r}_{t}, X_{r} ]\big] \big| \le \delta,
  \]
  for $\varepsilon> 0$ small enough (along a subsequence), where $A_t^x$ is the event that $(W_s, 0\le s\le t)$ stays in the quadrant $-x+\R_+^2$. Thus \eqref{eq: removing B_t} ensures that
  \[
    \Eb^{Q_{\varepsilon}}\big[f(e(s+T_z(r))-e(T_z(r)), 0\le s\le t) \mathds{1}_{\{\zeta>t+T_z(r)\}}\big]  \\
    \longrightarrow
    \Eb\big[ \Eb[f(W_s, 0\le s\le t) \mid A^{X_r}_{t}, X_{r} ]\big],
  \]
  along a subsequence, as $\varepsilon\to 0$. Note that the latter subsequence depends \textit{a priori} on $r$. We argue that one can find a subsequence for which the above convergence holds, regardless of $r$: indeed, if $r'<r$, we may run the same argument by stopping the path at time $T_z(r')$ instead of $T_z(r)$. Taking a sequence $r_n\to 0$, we can then use a diagonal extraction procedure to produce a subsequence that is valid for all $r$. Hence there exists a subsequence $\varepsilon_n\to 0$ such that, for all $t>0$ and $r\in(0,|z|)$,
  \begin{multline} \label{eq: cvg Q_epsn}
    \Eb^{Q_{\varepsilon_n}}\big[f(e(s+T_z(r))-e(T_z(r)), 0\le s\le t) \mathds{1}_{\{\zeta>t+T_z(r)\}}\big]  \\
    \longrightarrow
    \Eb\big[  \Eb[f(W_s, 0\le s\le t) \mid A^{X_r}_{t}, X_{r} ]\big],
    \quad \text{as } n\to \infty.
  \end{multline}
  The measures on the right-hand side of \eqref{eq: cvg Q_epsn} define consistent laws on paths, started at $X_r$ and with the transitions of Brownian motion conditioned to stay in the quadrant. By the Kolmogorov extension theorem, this defines a unique probability measure $Q$ on the space $E$. Under this probability measure $Q$, the path $e$ starts at $z$, satisfies $Q(e(t) \in \R_+^{*2} \mid \zeta>t)=1$ for all $t>0$, and has the transition probabilities of Brownian motion conditioned to stay in the quadrant. These properties characterise $P^z_{\theta}$ (see \cite[Theorem 3.1]{MS-mot}), and therefore $Q=P_{\theta}^z$. Hence 
  \begin{multline} \label{eq: Q_eps conclusion}
    \Eb^{Q_{\varepsilon_n}}\big[f(e(s+T_z(r))-e(T_z(r)), 0\le s\le t) \mathds{1}_{\{\zeta>t+T_z(r)\}}\big] \\
    \longrightarrow
    \Eb^{P_{\theta}^z}[f(e(s+T_z(r))-e(T_z(r)), 0\le s\le t)\mathds{1}_{\{\zeta>t+T_z(r)\}}], \quad \text{as } n \to \infty.
  \end{multline}
  Now the above argument shows that this is the only possible limit law of
  \[
    \Eb^{Q_{\varepsilon_n}}\big[f(e(s+T_z(r))-e(T_z(r)), 0\le s\le t) \mathds{1}_{\{\zeta>t+T_z(r)\}}\big].
  \]
  By Prokhorov's theorem, we conclude that the convergence \eqref{eq: Q_eps conclusion} holds not only along a subsequence, but as $\varepsilon\to 0$, {which proves \eqref{eq:cvg_Q_eps}}.
\end{proof}

\subsection{A Bismut description of the backward cone excursion measure \texorpdfstring{$\nlg_{\theta}$}{n\textunderscore theta}}
The classical Bismut description deals with the  Itô measure for one-dimensional Brownian motion, and roughly describes the infinite excursion measure seen from a {Lebesgue-distributed time $t$ in the excursion} (see \cite[Theorem XII.4.7]{RY}). One straightforward consequence is a Bismut description of Brownian excursions in the half-plane \cite[Proposition 2.6]{AD}. This states that under the infinite half-plane Brownian excursion measure, for a Lebesgue-distributed time $t$ from the excursion, the height of the excursion at time $t$ is distributed according to the Lebesgue measure $\mathrm{d}a$. Moreover, it describes the left and right parts of the trajectory from time $t$ onwards as two independent standard Brownian motions stopped when reaching the horizontal line $\{z\in\C, \; \Im(z)=-a\}$. This, intuitively, corresponds to a Bismut description for $\nlg_\theta$ in the case $\theta=\pi$. The nature of the cone excursions for general $\theta$ makes the Bismut description of $\nlg_{\theta}$ more involved, but it remains similar in spirit. Let us now explain the result.

For $e\in E$ and $t\in (0,\zeta)$ we let
\begin{equation}\label{eq: gamma plus minus}
  e^{t,-}:=(e(t-s)-e(t), 0\le s\le t) \quad \text{and} \quad e^{t,+}:=(e(t+s)-e(t), 0\le s\le \zeta-t),
\end{equation}
{and we recall from \cref{sec: intro main result cones} the notation $\varsigma_\theta^t$ for the total cone-free local time of $e^{t,-}$ (with the same normalisation as in \cref{rk: choice ell}).}

\begin{figure}[ht]
  \begin{center}
    \includegraphics[scale=0.87]{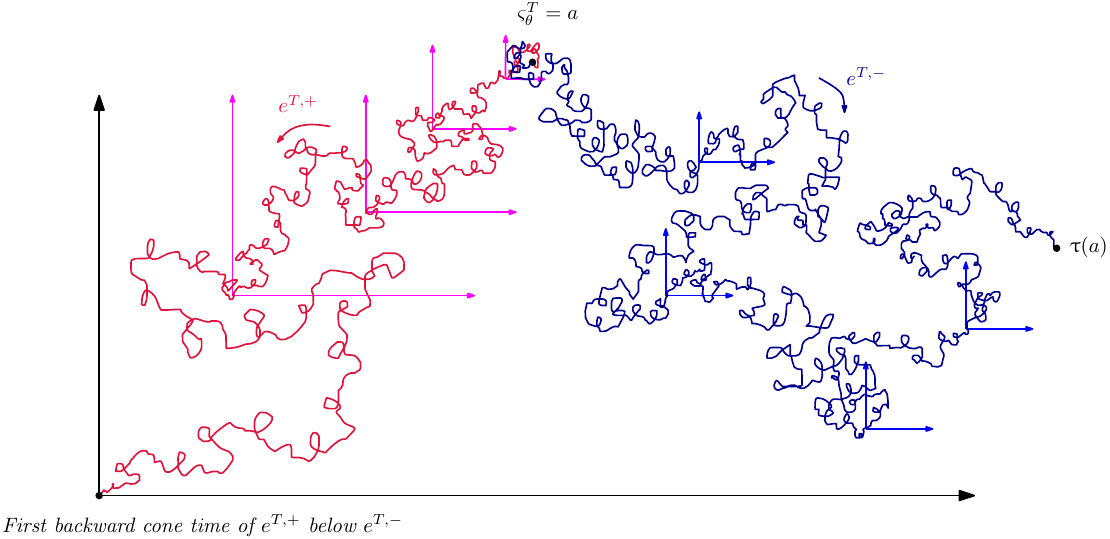}
  \end{center}
  \caption{The Bismut description of $\nlg_{\theta}$.} 
  \label{fig: Bismut}
\end{figure}


\begin{Thm}(Bismut description of $\nlg_{\theta}$) \label{thm: Bismut}
  \noindent Let $\overline{\nlg}_{\theta}$ be the measure on $\R_+\times E$ defined by
  \[
    \overline{\nlg}_{\theta}(\mathrm{d}T,\mathrm{d}e) = \mathds{1}_{0\le T\le \zeta} \, \mathrm{d}T \cdot \nlg_{\theta}(\mathrm{d}e).
  \]
  Then 
  for any non-negative functional $F$ on $E$ and $g$ on $(0,\infty)$:
  \[
    \overline{\nlg}_\theta\left( F(e^{T,-}) {g(\varsigma^T_\theta)}  \right)
    =
    \overline{c}_\theta \cdot \Eb\left[ \int_0^{\infty} \mathrm{d}a \cdot g(a) F(B^{\uptau_{\theta}(a)})  \right]
  \] where 
  $\overline{c}_\theta$ is a constant depending only on $\theta$ (that we will not make explicit) and $B^{\uptau_{\theta}(a)}$ is a correlated Brownian motion as in \eqref{eq: cov MoT}, stopped at the first time 
  that its cone-free local time (as in \cref{sec: forward cones} with the same choice of multiplicative constant, see \cref{rk: choice ell}) equals $a$. Moreover, under $\overline{\nlg}_{\theta}$ and conditionally on $e^{T,-}$, the process $e^{T,+}$ has the law of an independent planar correlated Brownian motion, as in \eqref{eq: cov MoT}, stopped at its first simultaneous running infima that lies below the whole path $e^{T,-}$.
\end{Thm}
\begin{Rk}
  This means that under $\overline{\nlg}_{\theta}$, 
  the marginal ``law'' of the total local time $\varsigma^T_\theta$ is a constant times Lebesgue. Then, conditionally on {$\varsigma^T_\theta=a$}, the law of $e^{T,-}:=(e(T-s)-e(T), 0\le s\le T)$  is a correlated Brownian motion run until the first time $\uptau_{\theta}(a)$ 
  that its cone-free local time is equal to $a$, and conditionally on $e^{T,-}$, $e^{T,+}$ has law as described above. See \cref{fig: Bismut}.

  It is important to point out that, unlike in the one-dimensional or in the half-plane case, the paths $e^{T,-}$ and $e^{T,+}$ are no longer independent conditionally on $\varsigma^T_\theta$. This dependence makes the Bismut description of $\nlg_{\theta}$ much more involved, although we stress that the only dependence concerns the stopping time for $e^{T,+}$.
\end{Rk}

\begin{proof}[Proof of \cref{thm: Bismut}]

  Under $\Pb$, let $W$ be a correlated Brownian motion as in \eqref{eq: cov MoT}, started at the origin. {Let $(\efrak_\theta(s), s>0)$ be its associated backwards cone excursion process, as in 
  \cref{def: backward cone excursion process}. {We use the shorthand} $s_t {:= \lfrak_\theta(t)}$ {and note that for Lebesgue--almost every $t$}, $t\in (\tfrak_\theta(s_t^-),\tfrak_\theta(s_t))$, so that $\efrak_\theta(s_t): r \mapsto W(\tfrak_{\theta}(s_t^-)+r)-W(\tfrak_{\theta}(s_t))$ (for $0\le r\le \tfrak_{\theta}(s_t)-\tfrak_{\theta}(s_t^-)$) is the backward cone excursion straddling $t$.} 
  Let $W^{t,-}:=(W(t-r)-W(t), 0 \le r\le t-\tfrak_\theta(s_t^-))$ 
  be the (time-reversed) part of the trajectory of $W$ between {$\tfrak_\theta(s_t^-)$} and $t$. For any non-negative measurable functional $F$, we will express the following quantity in two different ways:
  \[
    \mathcal{E}_\lambda(F):= \Eb\left[ \int_0^{\infty} \mathrm{e}^{-\lambda {\tfrak_\theta(s_t^-)}} F(W^{t,-}) \mathrm{d}t \right], \quad \lambda>0.
  \]
  We will see that this implies our claims on $\varsigma^T_\theta$ and $e^{T,-}$ in \cref{thm: Bismut}. The last claim on $e^{T,+}$ also follows from the same calculation by taking another function $G$ of the future $W^{t,+}$ after $t$ up to time {$\tfrak_\theta(s_t)$}, but we omit this part for simplicity.

  First, since for all $t\in (\tfrak_\theta(s^-),\tfrak_\theta(s))$ 
  we have  $\tfrak_\theta(s^-)=\tfrak_\theta(s_t^-)$ and $W^{t,-} = (\efrak_\theta(s))^{t-\tfrak_\theta(s^-),{-}}$ (where $e^{t,-}$ for $e\in E$ is as defined in \eqref{eq: gamma plus minus}), we can write
  \[
    \mathcal{E}_\lambda(F) 
    =
    \Eb\left[ \sum_{s>0} \mathrm{e}^{-\lambda \tfrak_\theta(s^-)} \int_{\tfrak_{\theta}(s^-)}^{\tfrak_\theta(s)} F((\efrak_\theta(s))^{t-\tfrak_\theta(s^-),{-}}) \mathrm{d}t \right].
  \]
  We now use the compensation formula for backward cone excursions and a change of variables to deduce that
  \[
    \mathcal{E}_\lambda(F) 
    =
    \Eb\left[ \int_{0}^{\infty} \mathrm{e}^{-\lambda \tfrak_{\theta}(s)} \mathrm{d}s\right] \cdot \nlg_{\theta}\left( \int_0^{\zeta} F(e^{t,-}) \mathrm{d}t \right).
  \]
  Furthermore, the first term on the right above is explicit. Indeed, recall from \cref{sec: backward cone excursion} that $\tfrak_{\theta}$ is a $(1-\frac{\alpha}{2})$--stable subordinator, hence $\Eb\left[\mathrm{e}^{-\lambda \tfrak_{\theta}(s)}\right] = \exp(-\hat c \lambda^{1-\frac{\alpha}{2}}t)$ where $\hat{c}$ is a constant depending only on $\theta$ (and in principle could be calculated from the formulae in \cite{LG-cones}, but we will not do this). Thus
  \begin{equation}  \label{eq: bismut expression 1}
    \mathcal{E}_\lambda(F) 
    =
    ({\hat c}\lambda^{1-\frac{\alpha}{2}})^{-1} \cdot \nlg_{\theta}\left( \int_0^{\zeta} F(e^{t,-}) \mathrm{d}t \right).
  \end{equation}

  \bigskip
  \begin{figure}[ht]
    \begin{center}
      \includegraphics[scale=0.95]{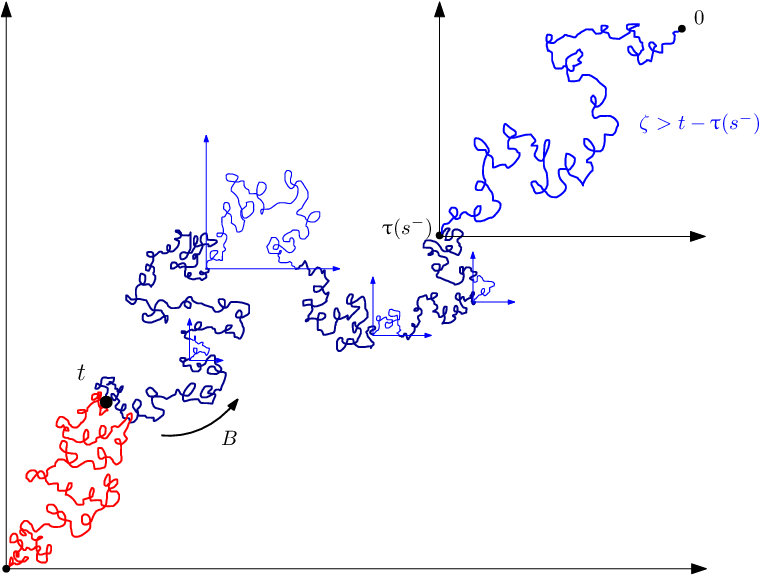}
    \end{center}
    \caption{The backward cone excursion straddling $t$. We start a correlated Brownian motion from $0$, and look at the backward cone excursion straddling time $t$ ({delimited} by the two black cones). Looking back from time $t$ (blue trajectory), we record all the forward cone excursions (bright blue). The excursion process is stopped at time $s$ when reaching an excursion such that $\zeta(\mathbbm{e}_{\theta}(s))>t-{\uptau_\theta}(s^-)$ (the last excursion in bold blue).}
    \label{fig: Bismut proof}
  \end{figure}

  On the other hand, we can consider the correlated Brownian motion $B$, defined to be the time reversal of $W$ from time $t$ to time $0$, then concatenated with an independent correlated Brownian motion. That is, $B(s)=W(t-s)-W(t)$ for $0\le s\le t$, and then $B(s)=W(0)-W(t)+W'(s-t) {= -W(t)+W'(s-t)}$ for $s\ge t$, where $W'$ is a further independent correlated planar Brownian motion. Then if $\ell_\theta, \uptau_\theta, (\mathbbm{e}_\theta(s), s>0)$ is the cone-free local time, inverse local time and \emph{forward cone excursion process} associated to $B$, $W^{t,-}$ is simply $B$ stopped at the first time $\uptau_\theta(s^-)$ such that $\zeta(\mathbbm{e}_{\theta}(s))>t-\uptau_\theta(s^-)$ (the first time that the forward cone excursion process records an excursion of duration large enough so that it straddles the original time $0$). See \cref{fig: Bismut proof}. We denote $B^{\uptau_{\theta}(s^-)} := (B(u),  u\le \uptau_{\theta}(s^-))$ for simplicity. The previous discussion amounts to
  \[
    \mathcal{E}_\lambda(F) 
    =
    \int_0^{\infty} \mathrm{d}t \cdot \Eb\left[ \sum_{s>0} \mathrm{e}^{-\lambda (t-\uptau_{\theta}(s^-))} F(B^{\uptau_{\theta}(s^-)}) \mathds{1}_{\zeta(\mathbbm{e}_{\theta}(s))>t-\uptau_{\theta}(s^-)} \mathds{1}_{t>\uptau_{\theta}(s^-)} \right].
  \]
  We can now use again the compensation formula, this time for \emph{forward} cone excursions to obtain that
  \[
    \mathcal{E}_\lambda(F) 
    =
    \Eb\left[ \int_0^{\infty} \mathrm{d}s \cdot F(B^{\uptau_{\theta}(s)}) \int_{\uptau_{\theta}(s)}^{\infty} \mathrm{d}t  \cdot \mathrm{e}^{-\lambda (t-\uptau_{\theta}(s))}  \ndms_{\theta}(\zeta>t-\uptau_{\theta}(s)) \right].
  \]
  A simple change of variables brings the above display to
  \[
    \mathcal{E}_\lambda(F) 
    =
    \Eb\left[ \int_0^{\infty} \mathrm{d}s \cdot F(B^{\uptau_{\theta}(s)})  \right] \cdot \int_{0}^{\infty} \mathrm{d}t  \mathrm{e}^{-\lambda t}  \ndms_{\theta}(\zeta>t).
  \]
  Recall from \cref{rk: tau subordinator} that for $t>0$,
  $\ndms_{\theta}(\zeta>t) = c't^{-\alpha/2}$ for some given $c'$.
  Therefore we conclude that
  \begin{equation} \label{eq: bismut expression 2}
    \mathcal{E}_\lambda(F) 
    =
    c' \Gamma(1-\alpha/2) \lambda^{\alpha/2-1} \cdot \Eb\left[ \int_0^{\infty} \mathrm{d}s \cdot F(B^{\uptau_{\theta}(s)})  \right].
  \end{equation}

  We finally combine \eqref{eq: bismut expression 1} and \eqref{eq: bismut expression 2} into 
  \[
    \nlg\left(\int_0^{\zeta} F(e^{t,-}) \mathrm{d}t \right)
    =
    \hat{c} c' \Gamma(1-\alpha/2) \cdot \Eb\left[ \int_0^{\infty} \mathrm{d}s \cdot F(B^{\uptau_{\theta}(s)})  \right].
  \]
  Since $B^{\uptau_\theta(s)}$ has the law of a planar correlated Brownian motion run until the first time that its cone-free local time equals $s$, this proves the first claim of \cref{thm: Bismut}.

\end{proof}

\section{Special properties of \texorpdfstring{$\frac{2\pi}{3}$}{2pi/3}--cone excursions}
\label{sec: special cones}

In this section, we take $\theta=\tfrac{2\pi}{3}$ and derive some special features of $\tfrac{2\pi}{3}$--cone excursions. From now on, we drop the subscript $\theta$ for ease of notation. From the LQG perspective, the case $\theta=\tfrac{2\pi}{3}$ corresponds to $\gamma=\sqrt{8/3}$ and $\kappa'=6$, as explained in \cref{sec: mot}. As we shall see, $\nlg$ enjoys many nice properties such as an explicit joint law for the displacement and the duration, and a re-sampling property. We use this last property to give a Brownian motion proof of a target-invariance property for $\sle_6$ in the $\sqrt{8/3}$--quantum disc, \textit{cf.} \cref{cor: intro target inv lqg}.

\subsection{Joint law of the start point and duration under \texorpdfstring{$\nlg$}{n\textunderscore 2pi/3}}
\label{sec: joint law}

Our first result describes the joint law of the start point and duration under $\nlg$. 

In the LQG framework, this describes the joint ``law'' of the left/right boundary lengths and the area of a $\sqrt{8/3}$--quantum disc: see \cref{sec: mot}. For general $\theta$, we stress that even the law of the start point itself is not explicit under $\nlg_{\theta}$ (see \cref{prop: nlg disintegration}). Remarkably, in the $\theta=\frac{2\pi}{3}$ case, one can work out not only the start point, but also the joint law with the duration. 
{Recall from \cref{prop: nlg disintegration} the laws $P^z$, $z\in (\R_+^*)^2$, disintegrating the measure $\nlg$ over the start point.}

\begin{Prop} \label{prop: joint law displacement/duration}
  The joint {``law''} of the start point $e(0)$ and duration $\zeta(e)$ under $\nlg$ is given by
  \begin{equation} \label{eq: joint density nlg}
    \nlg(e(0)\in (\mathrm{d}l,\mathrm{d}r), \zeta \in \mathrm{d}t) 
    =
    \frac{3^{-5/8}}{8\sqrt{2\pi}}(l+r)^{1/2} \mathrm{e}^{-\frac{(l+r)^2}{2\sqrt{3} t}} t^{-5/2}  \mathrm{d}l \mathrm{d}r \mathrm{d}t.
  \end{equation}
  In particular, the {``law''} of the start point under $\nlg$ is 
  \begin{equation} \label{eq: sqrt(8/3) law displacement}
    \nlg(e(0)\in (\mathrm{d}l,\mathrm{d}r)) 
    = \frac{3^{1/8}}{8}
    \frac{\mathrm{d}l \mathrm{d}r}{(l+r)^{5/2}},
  \end{equation}
  and {for all $(l,r) \in (\R_+^*)^2$}, the law of $\zeta$ under {$P^{(l,r)}$}
  is that of $(l+r)^2 \zeta$ under {$P^{(l',r')}$}
  for any $l',r'>0$ such that $l'+r'=1$.
\end{Prop}
\begin{Rks} \label{rk: independence duration/position}
  \begin{enumerate}
    \item We see from \cref{prop: joint law displacement/duration} that conditionally on $||e(0)||_1:=e^1(0)+e^2(0)$, the duration $\zeta$ is independent of $\frac{e(0)}{||e(0)||_1}$.
    \item {\cref{prop: intro law duration} and \cref{cor: intro law area} follow directly from \cref{prop: joint law displacement/duration} and the mating-of-trees correspondence \cref{thm: MoT complete}, thereby reproving \cite[Theorem 1.2]{AG} for $\gamma=\sqrt{8/3}$. More precisely, it comes from taking a limit as the start point is sent to the boundary, applying \cref{prop: convergence measures}.}
    \item {Formula \eqref{eq: sqrt(8/3) law displacement} gives an expression of the Lévy measure of the process $W\circ \tfrak$ in \cref{thm: LG stable}, which answers a question of Le Gall \cite{LG-cones}, in the case when $\theta=2\pi/3$.}
  \end{enumerate}
\end{Rks}

Define 
\[
  H(a,b,\lambda):=\nlg (1-\mathrm{e}^{-ae^1(0)-be^2(0)-\lambda \zeta}).
\]
For the purposes of comparing with Le Gall~\cite{LG-cones}, it is convenient to change co-ordinates.
Let $(a,b)$ be such that
$(a \quad b) M = (-\cos(\theta_0) \quad -\sin(\theta_0))$ where 
\[
  M=\Lambda^{-1}R=
  \mathbbm{a}\begin{pmatrix}
    \sin(\theta/2) & -\cos(\theta/2) \\
    \sin(\theta/2) & \cos(\theta/2)
  \end{pmatrix}
  =
  \begin{pmatrix}
    3^{1/4} & -3^{-1/4} \\
    3^{1/4} & 3^{-1/4}
  \end{pmatrix},
\]
for $R$ an anticlockwise rotation of $\theta/2=\pi/3$; equivalently 
\begin{equation}\label{eq: a b} 
  3^{1/4}(a+b)=-r_0 \cos(\theta_0) \quad 3^{-1/4}(b-a)=-r_0 \sin(\theta_0).
\end{equation}

Then, \cite[(5.f)]{LG-cones} shows
(taking $\alpha=\pi/3$ and $\nu=\pi/(2\alpha)=3/2$ in the notation of that work) 
that for $(a,b)$ satisfying \eqref{eq: a b} with $\theta_0 \in (-\pi,\pi]$
and $\lambda > 0$,
\begin{equation}\label{eq: LG formula G}
  H(a,b,\lambda)^{-1}
  =
  \frac{2^{3/2}(2\lambda)^{3/4}}{\pi\Gamma(3/2)}
  \int_{-\pi/3}^{\pi/3} \mathrm{d}\theta
  \int_0^\infty \mathrm{d}r
  \mathrm{e}^{r r_0 \cos(\theta-\theta_0)}r K_{3/2}(\sqrt{2\lambda} r)\cos(\tfrac32\theta).
\end{equation}

\begin{Lem}\label{lem: expression G}
  Let $c = (2\sqrt{3})^{-1}$.
  For $(a,b)$ satisfying \eqref{eq: a b} with $r_0=1$, $\theta_0 \in (-\pi,\pi]$
  and for $\lambda>1/2$,
  \[
    H(a,b,\lambda)
    =
    \frac{\lambda^{1/4}\sqrt{\pi}}{2^{5/4}\cdot 3}
    \frac
    {(\tfrac{a}{\sqrt{c\lambda}})^2+(\tfrac{b}{\sqrt{c\lambda}})^2+(\tfrac{a}{\sqrt{c\lambda}})(\tfrac{b}{\sqrt{c\lambda}})-3}
    {\sqrt{2+\tfrac{b}{\sqrt{c\lambda}}}(\tfrac{b}{\sqrt{c\lambda}}-1)
      +
    \sqrt{2+\tfrac{a}{\sqrt{c\lambda}}}(\tfrac{a}{\sqrt{c\lambda}}-1)}.
  \]
  \end{Lem}

\begin{proof}
  We begin from \eqref{eq: LG formula G}. Noting that 
  $K_{3/2}(x) = \sqrt{\pi/2} \mathrm{e}{-x} \bigl(x^{-1/2}+x^{-3/2}\bigr)$
  and changing variables $s=\sqrt{2\lambda}r$, we can obtain
  \[
    H(a,b,\lambda)^{-1}
    =
    \frac{2^{7/4}\lambda^{-1/4}}{\pi} \int_{-\pi/3}^{\pi/3} \mathrm{d}\theta \cos(\tfrac32\theta) I(\theta),
  \]
  where
  \begin{align*}
    I(\theta) 
    & = \int_0^\infty \mathrm{d} s \mathrm{e}^{-s(1-\tfrac{1}{\sqrt{2\lambda}}\cos(\theta-\theta_0))}(\frac{1}{\sqrt{s}}+\sqrt{s}) \\ 
    & = \sqrt{\pi}
    \left(
      \frac{1}{(1-\tfrac{1}{\sqrt{2\lambda}}\cos(\theta-\theta_0))^{1/2}}
      + \frac{(1/2)}{(1-\tfrac{1}{\sqrt{2\lambda}}\cos(\theta-\theta_0))^{3/2}}
    \right) \\
    & = \sqrt{\pi}
    \frac{\tfrac{3}{2}-\tfrac{1}{\sqrt{2\lambda}}\cos(\theta-\theta_0)}{(1-\tfrac{1}{\sqrt{2\lambda}}\cos(\theta-\theta_0))^{3/2}},
  \end{align*}
  and the integral converges provided that $\frac{1}{\sqrt{2\lambda}}\cos(\theta-\theta_0)<1$,
  which holds for the range of $\lambda$ in the statement.

  We can then compute
  \begin{align*}
    H(a,b,\lambda)^{-1}
    &= \frac{2^{7/4}\lambda^{-1/4}}{\sqrt{\pi}}
    \int_{-\pi/3}^{\pi/3}\cos(\tfrac32\theta)
    \frac{\tfrac{3}{2}-\tfrac{1}{\sqrt{2\lambda}}\cos(\theta-\theta_0)}{(1-\tfrac{1}{\sqrt{2\lambda}}\cos(\theta-\theta_0))^{3/2}}
    \mathrm{d}\theta
    \\
    &= \frac{2^{7/4}\lambda^{-1/4}}{\sqrt{\pi}}
    \left. 
      \frac{\tfrac{1}{2\lambda}\sin(2\theta_0-\tfrac{\theta}{2})-\tfrac{1}{\sqrt{2\lambda}}\sin(\theta_0+\tfrac{\theta}{2})+(\tfrac{1}{2\lambda}-1)\sin(\tfrac{3\theta}{2})}
      {(\tfrac{1}{2\lambda}-1)\sqrt{1-\tfrac{1}{\sqrt{2\lambda}}\cos(\theta-\theta_0)}}
    \right\rvert_{-\pi/3}^{\pi/3}.
  \end{align*}
  (This integral can be checked by differentiating the preceding expression and using
  product-sum formulae for trigonometric functions.)
  This evaluates to the following expression:
  \begin{align}
    H(a,b,\lambda)^{-1} 
    &=
    \frac{2^{3/4}}{\lambda^{1/4}\sqrt{\pi}}
    \frac{\sqrt{3}\sin(2\theta_0)-\cos(2\theta_0)+\sqrt{2\lambda}(-\sqrt{3}\sin(\theta_0)-\cos(\theta_0))+(2-4\lambda)}{(1-2\lambda)\sqrt{1-\tfrac{1}{2\sqrt{2\lambda}}(\cos(\theta_0)+\sqrt{3}\sin(\theta_0))}}
    \nonumber \\
    & \quad {} +
    \frac{2^{3/4}}{\lambda^{1/4}\sqrt{\pi}}
    \frac{-\sqrt{3}\sin(2\theta_0)-\cos(2\theta_0)+\sqrt{2\lambda}(\sqrt{3}\sin(\theta_0)-\cos(\theta_0))+(2-4\lambda)}{(1-2\lambda)\sqrt{1-\tfrac{1}{2\sqrt{2\lambda}}(\cos(\theta_0)-\sqrt{3}\sin(\theta_0))}}.
    \label{e:H4.3}
  \end{align}

  Set 
  \[
    A:=2\cdot 3^{1/4}a=(-\cos(\theta_0)+\sqrt{3}\sin(\theta_0))
    \text{ and }
    B:=2\cdot 3^{1/4}b=(-\cos(\theta_0)-\sqrt{3}\sin(\theta_0)).
  \]
  It is straightforward to check that
  \[ 
    A^2-2 = -\cos(2\theta_0)-\sqrt{3}\sin(2\theta_0) 
    \quad ; \quad 
    B^2-2= -\cos(2\theta_0)+\sqrt{3}\sin(2\theta_0).
  \]
  In these variables, the expression for $H(a,b,\lambda)^{-1}$ from \eqref{e:H4.3} then becomes (reordering the summands)
  \begin{equation*}
    H(a,b,\lambda)^{-1} 
    =
    \frac{2^{3/4}}{\lambda^{1/4}\sqrt{\pi}}
    \left(
      \frac{A^2+\sqrt{2\lambda}A-4\lambda}{(1-2\lambda)\sqrt{1+\tfrac{A}{2\sqrt{2\lambda}}}}
      +
      \frac{B^2+\sqrt{2\lambda}B-4\lambda}{(1-2\lambda)\sqrt{1+\tfrac{B}{2\sqrt{2\lambda}}}}
    \right),
  \end{equation*}
  which is equal to 
  \begin{equation*}
    H(a,b,\lambda)^{-1} 
    =
    \frac{2^{5/4}}{\lambda^{1/4}\sqrt{\pi}}
    \left(
      \frac{4\sqrt{3}a^2+23^{1/4}\sqrt{2\lambda}a-4\lambda}{(1-2\lambda)\sqrt{2+\tfrac{a}{\sqrt{c\lambda}}}}
      +
      \frac{4\sqrt{3}b^2+23^{1/4}\sqrt{2\lambda}b-4\lambda}{(1-2\lambda)\sqrt{2+\tfrac{b}{\sqrt{c\lambda}}}}
    \right),
  \end{equation*}
  recalling that $c=(2\sqrt{3})^{-1}$.
  This in turn can be rewritten
  \[
    H(a,b,\lambda)^{-1} 
    =
    \frac{4\cdot 2^{1/4}\lambda^{3/4}}{(1-2\lambda)\sqrt{\pi}}
    \left(
      \frac{(\tfrac{a}{\sqrt{c\lambda}})^2+\tfrac{a}{\sqrt{c\lambda}}-2}{\sqrt{2+\tfrac{a}{\sqrt{c\lambda}}}}
      +
      \frac{(\tfrac{b}{\sqrt{c\lambda}})^2+\tfrac{b}{\sqrt{c\lambda}}-2}{\sqrt{2+\tfrac{b}{\sqrt{c\lambda}}}}
    \right),
  \]
  which, writing $x^2+x-2=(x+2)(x-1)$, is simply 
  \[
    H(a,b,\lambda)^{-1} 
    =
    \frac{4\cdot 2^{1/4}\lambda^{3/4}}{(1-2\lambda)\sqrt{\pi}}
    \left(\sqrt{2+\tfrac{b}{\sqrt{c\lambda}}}(\tfrac{b}{\sqrt{c\lambda}}-1)
      +
    \sqrt{2+\tfrac{a}{\sqrt{c\lambda}}}(\tfrac{a}{\sqrt{c\lambda}}-1)\right).
  \]
  Notice that $A^2+B^2+AB=3$ so that 
  \[
    \bigl(\tfrac{a}{\sqrt{c\lambda}}\bigr)^2
    +\bigl(\tfrac{b}{\sqrt{c\lambda}}\bigr)^2
    +\tfrac{a}{\sqrt{c\lambda}}\tfrac{b}{\sqrt{c\lambda}}
    -3 
    = \frac{3(1-2\lambda)}{2\lambda},
  \]
  which, upon taking reciprocals (and noting that the expression we have for $H(a,b,\lambda)^{-1}$
  is never zero for $\lambda>1/2$), gives the expression in the statement.
\end{proof}

Next let $\mu$ be the infinite measure on the right-hand side of \eqref{eq: joint density nlg}; that is,
\[
  \mu(\mathrm{d}l, \mathrm{d}r, \mathrm{d}t) 
  =
  \frac{3^{-5/8}}{8\sqrt{2\pi}}
  (l+r)^{1/2} \mathrm{e}^{-\frac{(l+r)^2}{2\sqrt{3} t}} t^{-5/2}  \mathrm{d}l \mathrm{d}r \mathrm{d}t,
\]
and set
\[
  \hat{H}(a,b,\lambda)
  :=\int_{(0,\infty)^3} (1-\mathrm{e}^{-al-br-\lambda t})\mu(\mathrm{d}l, \mathrm{d}r,\mathrm{d}t).
\]

\begin{Prop}\label{prop: nlg precise}
  For $a,b$ satisfying \eqref{eq: a b} with $r_0 = 1$, $\theta_0 \in (-\pi,\pi)\setminus \{0\}$
  and $\lambda>1/2$,
  we have
  \[
    H(a,b,\lambda)
    =
    \hat{H}(a,b,\lambda).
  \]
\end{Prop}

We will prove the proposition by deriving an expression for $\hat{H}(a,b,\lambda)$ and equating it with the expression from \cref{lem: expression G}. But first, let us see how it implies \cref{prop: joint law displacement/duration}.

\begin{proof}[Proof of \cref{prop: joint law displacement/duration}]
  Making an integral substitution in \eqref{eq: LG formula G} gives the scaling relation
  \[
    H(r_1a,r_1b,\lambda) = r_1^{1/2} H(a,b,r_1^{-2}\lambda),
  \]
  for $a,b$ satisfying \eqref{eq: a b} with $r_0=1$, $r_1>0$, and $\lambda>0$
  (allowing for the possibility that one, and then both, sides may be infinite).
  Using the expression for $\hat{H}$ in terms of $\mu$ shows that the same
  scaling relation holds for $\hat{H}$.

  With this in mind, \cref{prop: nlg precise} implies that 
  $H(a,b,\lambda) = \hat{H}(a,b,\lambda)$ whenever $(a,b)$ satisfies
  \eqref{eq: a b} with $\theta_0\in(-\pi,\pi)\setminus\{0\}$,
  $r_0>0$ and $\lambda > r_0^2/2$. Since this
  range of arguments contains an open ball in $\R^3$, the uniqueness
  of the Lévy--Khintchine formula implies that 
  $\nlg(e(0)\in (\mathrm{d}l,\mathrm{d}r), \zeta \in \mathrm{d}t) 
  = \mu(\mathrm{d}l,\mathrm{d}r,\mathrm{d}t)$.
\end{proof}

\begin{proof}[Proof of \cref{prop: nlg precise}]
  We will compute $\hat{H}(a,b,\lambda)$ and show that it matches the expression
  we have already derived for $H$.
  Set $c=(2\sqrt{3})^{-1}$ once again, and assume that $a,b>0$.

  First, let $u=l+r$. A change of variables then gives 
  \[
    \hat{H}(a,b,\lambda)
    =
    \frac{3^{-7/8}}{8\sqrt{2\pi}} \lambda^{3/2}
    \int_{(0,\infty)^2}\int_0^u (1-\mathrm{e}^{-au-(b-a)r-t}) \frac{\sqrt{u}}{t^{5/2}} \mathrm{e}^{-c\lambda u^2/t} \mathrm{d}r \mathrm{d}t \mathrm{d}u.
  \]
  Integrating over $r$ we get
  \begin{equation*}
    \hat{H}(a,b,\lambda)
    =
    \begin{cases}
      \frac{3^{-5/8}}{8\sqrt{2\pi}} \lambda^{3/2}
      \int_{(0,\infty)^2} (u-u\mathrm{e}^{-au-t}) \frac{\sqrt{u}}{t^{5/2}} \mathrm{e}^{-c\lambda u^2/t} \mathrm{d}t \mathrm{d}u &  \quad a=b \\[2ex]
      \frac{3^{-5/8}}{8\sqrt{2\pi}} \lambda^{3/2}
      \int_{(0,\infty)^2} (u-\mathrm{e}^{-t}\frac{\mathrm{e}^{-bu}-\mathrm{e}^{-au}}{a-b}) \frac{\sqrt{u}}{t^{5/2}} \mathrm{e}^{-c\lambda u^2/t} \mathrm{d}t \mathrm{d}u & \quad  a\ne b.
    \end{cases}
  \end{equation*}

  Now we use that for $k>0$,
  \[ \int_0^\infty t^{-5/2}\mathrm{e}^{-k/t} \mathrm{d}t = k^{-3/2} \int_0^\infty t^{-5/2}\mathrm{e}^{-1/t} \mathrm{d}t = \frac{\sqrt{\pi}}{2}k^{-3/2}, \]
  \[ \text{ and } \quad \int_0^\infty t^{-5/2}\mathrm{e}^{-k/t-t} \mathrm{d}t = k^{-3/2} \int_0^\infty t^{-5/2}\mathrm{e}^{-1/t-kt} \mathrm{d}t = \frac{\sqrt{\pi}}{2}k^{-3/2}(1+2\sqrt{k})\mathrm{e}^{-2\sqrt{k}}, \]
  where the last equality can be obtained \textit{e.g.} from  \cite[Formula 7.12.23]{erdelyi1953higher}, using the exact expression for the modified Bessel function $K_{3/2}$.
  Integrating over $t$ (and setting $k=c\lambda u^2$) we get
  \begin{equation*}
    \hat{H}(a,b,\lambda):= 
    \begin{cases}
      \frac{3^{1/8}}{8}
      \int_{(0,\infty)} u^{-3/2}(1-(1+2\sqrt{c \lambda }u)\mathrm{e}^{-2\sqrt{c\lambda }u}\mathrm{e}^{-au})du &  \quad a=b \\
      \frac{3^{1/8}}{8}
      \int_{(0,\infty)}u^{-5/2} (u-(1+2\sqrt{c\lambda}u)\mathrm{e}^{-2\sqrt{c\lambda}u}\frac{\mathrm{e}^{-bu}-\mathrm{e}^{-au}}{a-b})  du & \quad  a\ne b.
    \end{cases}
  \end{equation*}
  Since we know how to integrate $u^{-3/2}\mathrm{e}^{- k u}$ and $u^{-1/2}\mathrm{e}^{-ku}$ in terms
  of gamma functions, it only remains to evaluate the integrals above, to obtain that
  \begin{equation*}
    \hat{H}(a,b,\lambda)
    =
    \begin{cases}
      \frac{\sqrt{\pi}}{4\cdot 2^{1/4}}
      \lambda^{1/4}
      \frac{\frac{a}{\sqrt{c\lambda}} + 1}{\sqrt{\frac{a}{\sqrt{c\lambda}}+2}}
      &  \quad a=b \\
      \frac{\sqrt{\pi}}{6\cdot 2^{1/4}}
      \lambda^{1/4}\frac{(\frac{a}{\sqrt{c\lambda}}-1)\sqrt{\frac{a}{\sqrt{c\lambda}}+2}-(\frac{b}{\sqrt{c\lambda}}-1)\sqrt{\frac{b}{\sqrt{c\lambda}}+2}}{\frac{a}{\sqrt{c\lambda}}-\frac{b}{\sqrt{c\lambda}}} 
      & \quad  a\ne b.
    \end{cases}
  \end{equation*}
  Moreover, if \eqref{eq: a b} is satisfied with $r_0=1$ and $\lambda>1/2$ (which ensures that 
  $(\frac{a}{\sqrt{c\lambda}}-1)\sqrt{\frac{a}{\sqrt{c\lambda}}+2}
  +(\frac{b}{\sqrt{c\lambda}}-1)\sqrt{\frac{b}{\sqrt{c\lambda}}+2}\ne 0$), we can
  simplify the case $a\ne b$ to the following expression:
  \[ 
    \hat{H}(a,b,\lambda)
    =
    \frac{\sqrt{\pi}}{6\cdot 2^{1/4}}
    \lambda^{1/4}
    \frac{(\tfrac{a}{\sqrt{c\lambda}})^2+(\tfrac{b}{\sqrt{c\lambda}})^2+(\tfrac{a}{\sqrt{c\lambda}})(\tfrac{b}{\sqrt{c\lambda}})-3}{((\tfrac{a}{\sqrt{c\lambda}})-1)\sqrt{(\tfrac{a}{\sqrt{c\lambda}})+2}+((\tfrac{b}{\sqrt{c\lambda}})-1)\sqrt{(\tfrac{b}{\sqrt{c\lambda}})+2}}.
  \]
  This matches our computation for $H(a,b,\lambda)$ in \cref{lem: expression G}.
\end{proof}

\subsection{A target-invariance property} \label{sec: target invariance}

We now prove a version of the target-invariance property of $\sle_6$ on the $\sqrt{8/3}$--quantum disc using only Brownian motion arguments. For $z\in \R_+^2$, let $\Pb_z$ denote the law of two independent Brownian motions $W$ and $W'$, starting at $0$ and $z$ respectively. Our Brownian motion results will hold under the law 
\begin{equation} \label{eq: Q_L def}
  Q_L := \int_{0}^1  \mathrm{d}x \cdot 	\Pb_{xL,(1-x)L}(\, \cdot \,) , \quad L\ge 0,
\end{equation}
where we average over a uniform \emph{angle for the $1$--norm}.
The law $Q_L$ should be understood as follows: we first sample a uniform variable $U\in(0,1)$, and then sample two independent Brownian motions $W$ and $W'$ where $W$ starts at $0$ and $W'$ at $L \cdot (U, 1-U)$. Importantly, under $Q_L$ the initial displacement $W'(0)-W(0)$ is random, but $||W'(0)-W(0)||_1 = L$ is deterministic. 

We now fix $L\ge 0$. The main observable in this section is the process $S$ defined as follows. 
For $a\ge 0$, consider the process $(W'(\uptau(t)), \;t\le a)$ corresponding to
re-parametrising $W'$ by inverse local time on the set of forward $\frac{2\pi}{3}$ cone-free times. Introduce the first passage time 
\begin{equation} \label{eq: sfrak def}
  \sfrak(a) := \inf\{s\ge 0, \; W'(\uptau(t)) \in W(\tfrak(s))+\R_+^2 \; \text{for all} \; t\le a \},
\end{equation}
of the (backward cone) process $(W(\tfrak(t)), \; t\ge 0)$ below the path $(W'(\uptau(t)), \;t\le a)$. Finally, define 
\begin{equation} \label{eq: def Y(a) and S(a)}
  Y(a):= W'\circ\uptau(a)-W\circ \tfrak(\sfrak(a)), \quad \text{and} \quad S(a):= ||Y(a)||_1.
\end{equation}
In particular, note that $S(0)=L$, and that $Y$ is a (two-dimensional) Markov process. On the other hand, it is not clear (and will in fact take quite some work to show) that $S$ is Markov. 
{Determining whether a function of a Markov process is still a Markov process is a problem known as \emph{Markov functions} \cite{rogers1981markov}. We will not use the general theory since in our case it would not simplify the proof.}

{The process $Y$ has a clear LQG interpretation as it describes the left/right boundary length process when exploring an $\sle_6$--decorated $\sqrt{8/3}$--\emph{quantum cone} outwards from a typical point (see the paragraph following \cref{prop: intro pathwise 3/2 stable cond}). Likewise the process $S$ describes the total quantum boundary length in the same exploration. We will not use this fact, but we stress that the processes $Y$ and $S$ will arise in our context from the Bismut description (\cref{thm: Bismut}), which connects cone excursions to whole plane Brownian motion in a similar way to how quantum discs relate to quantum cones.}
We start with a basic scaling property for $Y$.
\begin{Prop} \label{prop: Z self-similar}
  $Y$ is self-similar with index $\frac32$.
\end{Prop}
\begin{proof}
  This follows from direct calculations. First, recall that $W'\circ \uptau$ is self-similar with index $\frac32$ (\cref{thm: DMS stable}), whereas $W\circ \tfrak$ is self-similar with index $\frac12$ (\cref{thm: LG stable}). We now need to deal with the time-change $\sfrak$. Let $x>0$.
  Scaling $W'\circ\uptau$ we obtain:
  \begin{align*}
    \text{Under $\Pb_z$: } \sfrak(x^{-3/2}a) 
    &= \inf\{s\ge 0, \; W'(\uptau(t)) \in W(\tfrak(s))+\R_+^2 \; \text{for all} \; t\le x^{-3/2}a \} \\
    &= \inf\{s\ge 0, \; W'(\uptau(x^{-3/2}t)) \in W(\tfrak(s))+\R_+^2 \; \text{for all} \; t\le a \} \\
    &\overset{\mathrm{d}}{=} \inf\{s\ge 0, \; W'(\uptau(t)) \in xW(\tfrak(s))+\R_+^2 \; \text{for all} \; t\le a \} 
    =: \mathfrak{s}'(a) \text{ under } \Pb_{xz}
  \end{align*}
  Now, using the scaling of $W\circ\mathfrak{t}$, we observe that (under measure $\Pb_{xz}$ on both sides)
  $xW(\mathfrak{t}(\mathfrak{s}'(a))) \overset{\mathrm{d}}{=} W(x^{1/2}\mathfrak{t}(\mathfrak{s}''(a)))$,
  where
  \begin{align*}
    \mathfrak{s}''(a) 
    &= \inf\{s\ge 0, \; W'(\uptau(t)) \in W(\tfrak(x^{1/2}s))+\R_+^2 \; \text{for all} \; t\le a \} \\
    &= x^{-1/2} \inf\{s\ge 0, \; W'(\uptau(t)) \in W(\tfrak(s))+\R_+^2 \; \text{for all} \; t\le a \} \\
    &= x^{-1/2} \sfrak(a).
  \end{align*}
  Combining this last observation with the two equalities in distribution, we obtain that
  the law of $(xW(\mathfrak{t}(\mathfrak{s}(x^{-3/2}a))),xW'(\uptau(x^{-3/2}a)))$ under $\Pb_z$
  is equal to that of $(W(\mathfrak{t}(\mathfrak{s}(a))),W'(\uptau(a)))$ under $\Pb_{xz}$,
  which is the self-similarity property of the two dimensional process. The claim
  follows.
\end{proof}

The next result is a simple consequence of Bismut's description of $\nlg$ (\cref{thm: Bismut}) and the scaling property of $Y$. 
\begin{Cor} \label{cor: idp local time}
  Under the measure $\overline{\nlg}$ introduced in \cref{thm: Bismut}, the total cone-free local time $\varsigma^T$ of $T$ is independent of $\frac{e(0)}{||e(0)||_1}$. 
\end{Cor}
\begin{proof}
  By \cref{thm: Bismut}, for all non-negative measurable functions $h$ and $f$,
  \[
    \overline{\nlg} \left(h(\varsigma^T)f\left(\frac{e(0)}{||e(0)||_1} \right)\right)
    =
    \overline{c} \int_0^\infty \mathrm{d}a h(a) \Eb\left[f\left(\frac{Y(a)}{S(a)} \right)\right].
  \]
  By the scaling in \cref{prop: Z self-similar}, we obtain
  \[
    \overline{\nlg} \left(h(\varsigma^T)f\left(\frac{e(0)}{||e(0)||_1} \right)\right)
    =
    \overline{c} \int_0^\infty \mathrm{d}a h(a) \Eb\left[f\left(\frac{Y(1)}{S(1)} \right)\right]
    =
    \overline{\nlg} (h(\varsigma^T)) \Eb\left[f\left(\frac{Y(1)}{S(1)} \right)\right],
  \]
  which is the claim in \cref{cor: idp local time}.
\end{proof}

Our main goal in this subsection is to prove that the process $S$ defined in \eqref{eq: def Y(a) and S(a)} is Markov. We will then describe explicitly the law of $S$ later in \cref{sec: BM construction conditioning}. We emphasise that the laws of $(W'(\uptau(t), \;t\le a))$ and $(W(\tfrak(t)), \; t\ge 0)$ are known from \cref{thm: DMS stable} and \cref{thm: LG stable}. However, the definition of $Y$ and $S$ involves an intricate time-change $\sfrak$ which breaks the independence of $W$ and $W'$. 
We start with a technical lemma giving a rough bound on the distribution function of $S$.
\begin{Lem} \label{lem: bound S<c}
  There exists a constant $M>0$ such that, for all $c>0$ and $a>0$, 
  \[
    Q_0(S(a)\le c) \le M \frac{\sqrt{c}}{a^{1/3}}.
  \]
\end{Lem}
\begin{proof}
  For $b\ge 0$, let $\Sigma'(b)$ (\textit{resp.} $\Sigma(b)$) be the sum of the co-ordinates of $W'\circ \uptau(b)$ (\textit{resp.} $W\circ\tfrak(b)$). The definition of $\sfrak$ implies that
  \[
    S(a) 
    =
    \Sigma'(a) - \Sigma(\sfrak(a)).
  \]
  Since $\Sigma'(a)=0$ happens with probability $0$, we may split the probability as follows:
  \begin{equation} \label{eq: S<c split}
    Q_0(S(a)\le c) = Q_0\big(\Sigma'(a) - \Sigma(\sfrak(a))\le c, \Sigma'(a)>0\big) + Q_0\big(\Sigma'(a) - \Sigma(\sfrak(a))\le c, \Sigma'(a)<0\big).
  \end{equation}

  Let us start with the first term. Noting that $\Sigma(\sfrak(a)) \le 0$ by definition of backward cone points, we bound this term as
  \[
    Q_0\big(\Sigma'(a) - \Sigma(\sfrak(a))\le c, \Sigma'(a)>0\big)
    \le
    Q_0\big(\Sigma'(a) \le c, \Sigma'(a)>0\big)
    \le
    \Eb^{Q_0}\left[ \sqrt{\frac{c}{\Sigma'(a)}} \cdot \mathds{1}_{\Sigma'(a)>0} \right].
  \]
  By scaling, we therefore get
  \[
    Q_0\big(\Sigma'(a) - \Sigma(\sfrak(a))\le c, \Sigma'(a)>0\big)
    \le
    \frac{\sqrt{c}}{a^{1/3}} \Eb^{Q_0}\left[ \Sigma'(1)^{-1/2} \cdot \mathds{1}_{\Sigma'(1)>0} \right].
  \]
  The latter moment is finite by an application of \cite[Theorem 1.13]{KP}, and this
  yields the bound in the statement.

  We now deal with the second term of \eqref{eq: S<c split}. First, we bring the question to a one-dimensional problem by defining
  \[
    \sigma(a) 
    :=
    \inf\{b>0, \; \Sigma(b) < \Sigma'(a)\}.	
  \]
  Since $\sfrak(a)\ge \sigma(a)$ and $b\mapsto \Sigma(b)$ is decreasing, we observe that 
  \[
    \Sigma'(a) - \Sigma(\sfrak(a)) \ge \Sigma'(a) - \Sigma(\sigma(a)).
  \]
  Therefore,
  \begin{equation} \label{eq: second term overshoot}
    Q_0\big(\Sigma'(a) - \Sigma(\sfrak(a))\le c, \Sigma'(a)<0\big)
    \le
    Q_0\big(\Sigma'(a) - \Sigma(\sigma(a))\le c, \Sigma'(a)<0\big).
  \end{equation}
  Conditioned on $\Sigma'(a)$, the random variable $\Sigma'(a) - \Sigma(\sigma(a))$ is nothing but the (downward) overshoot of $\Sigma$ at level $\Sigma'(a)$. The process $\widehat{\Sigma} := -\Sigma$ is a $\frac12$--stable subordinator; write $\eta_x := \inf\{b>0, \; \widehat{\Sigma}(b) > x \}$ for its first passage time above $x>0$. By \cite[Corollary 3.5]{KP}, the law of the overshoot at $x>0$ is explicitly given by
  \[
    Q_0\big(\widehat{\Sigma}(\eta_x)-x \le c\big) 
    =
    \frac{1}{2\pi} \int_0^c \mathrm{d}u \int_0^x \mathrm{d}y (x-y)^{-1/2}(y+u)^{-3/2}.	
  \]
  An elementary calculation yields the expression
  \[
    Q_0\big(\widehat{\Sigma}(\eta_x)-x \le c\big) 
    =
    \frac{2}{\pi} \arctan{\sqrt{\frac{c}{x}}}
    \le
    \frac{2}{\pi} \sqrt{\frac{c}{x}}.	
  \] 
  Coming back to our expression \eqref{eq: second term overshoot}, we deduce that 
  \[
    Q_0\big(\Sigma'(a) - \Sigma(\sfrak(a))\le c, \Sigma'(a)<0\big)
    \le
    \frac{2}{\pi} \sqrt{c} \cdot \Eb^{Q_0}\big[ (-\Sigma'(a))^{-1/2} \mathds{1}_{\Sigma'(a)<0} \big].
  \]
  By scaling, we end up with the bound 
  \[
    Q_0\big(\Sigma'(a) - \Sigma(\sfrak(a))\le c, \Sigma'(a)<0\big)
    \le
    \frac{2}{\pi} \frac{\sqrt{c}}{a^{1/3}} \cdot \Eb^{Q_0}\big[ (-\Sigma'(1))^{-1/2} \mathds{1}_{\Sigma'(1)<0} \big],
  \]
  where we claim that the moment is finite by another application of \cite[Theorem 1.13]{KP}. This concludes the proof of \cref{lem: bound S<c}.
\end{proof}

Our main result in this section is the following.

\begin{Prop} \label{prop: target invariance Q_L}
  Let $a\ge 0$. Under $Q_L$, the law of $\frac{Y(a)}{S(a)}$ is that of $(U, 1-U)$  where $U$ is uniform in $(0,1)$ and independent of $S(a)$.
\end{Prop}
\begin{proof}
  \newcounter{targetinv}
  \stepcounter{targetinv}
  We divide the proof into four steps, the crucial step being the third one.

  \medskip
  $\triangleright$ \textit{Step \thetargetinv: Absolute continuity of $S(a)$.}
  For technical reasons, we must first prove that $S(a)$ has a density (with respect to Lebesgue measure) under $Q_0$.
  Prefiguring notation appearing later, let $\Xi(a) = W\circ \mathfrak{t}(a)$,
  $\Xi'(a) = W'\circ\uptau(a)$, and
  $V_i(a) = \inf_{b\le a} \Xi'_i(b)$ for $i=1,2$ and $a\ge 0$.
  As has already been noted, $\Xi'$ and $\Xi$ are stable processes.

  Let $g_1$ and $g_2$ be non-negative, bounded measurable functions, and define
  \[h(a) = \Eb^{Q_0}[ g_1(V_1(a)) g_2(\Xi'_1(a) - V_1(a))].\] 
  Denote by
  $\mathcal{L}h(q) = \int_0^\infty \mathrm{e}^{-qa} h(a) \, \mathrm{d}a$ the Laplace transform.
  Then, using the Wiener-Hopf factorisation \cite[Theorem~VI.5]{Ber}
  of $\Xi'_1$       at the independent exponential time $\mathrm{e}_q$ with rate $q$, and the duality principle for Lévy processes \cite[Lemmas~3.4 and~3.5]{kyprianou2014fluctuations}, we have
  \[
    q\mathcal{L}h(q)
    = \Eb^{Q_0}[ g_1(V_1(\mathrm{e}_q)) g_2(\Xi'_1(\mathrm{e}_q) - V_1(\mathrm{e}_q)) ]
    = \Eb^{Q_0}[ g_1(V_1(\mathrm{e}_q))] \Eb^{\hat{Q}_0} [ g_2(-V_1(\mathrm{e}_q))],
  \]
  where $\Xi'$ under $\hat{Q}_0$ has the law of $-\Xi'$ under $Q_0$.
  We denote by $f_a$ and $\hat{f}_a$ the densities of
  $V_1(a)$ under $Q_0$ and $\hat{Q}_0$, respectively, which
  exist and are smooth, as noted in
  the remark at the start of the proof of Theorem~9 in \cite{Kuz-extrema}.
  Proceeding from the above equality gives
  \[
    q \mathcal{L}h(q)
    = \int_0^\infty q\mathrm{e}^{-qa} \int_{-\infty}^0 f_a(x) g_1(x) \, \mathrm{d}x \, \mathrm{d}a
    \cdot \int_0^\infty q\mathrm{e}^{-qa} \int_{-\infty}^0 \hat{f}_a(x) g_2(-x) \, \mathrm{d}x \, \mathrm{d}a
    = q^2 \mathcal{L} G_1(q) \mathcal{L} G_2(q),
  \]
  where
  $G_1(a) = \int_{-\infty}^0 f_a(x) g_1(x) \, \mathrm{d}x$
  and $G_2(a) = \int_{-\infty}^0 \hat{f}_a(x) g_2(-x) \, \mathrm{d}x$.
  To summarise, writing $G_1 * G_2$ for convolution, we have shown
  \[
    \mathcal{L} h(q) = q \mathcal{L}(G_1*G_2)(q),
    \qquad q > 0.
  \]
  Provided $G_1$ is differentiable (which we will show shortly),
  standard properties of Laplace transforms and convolutions give us
  \begin{align*}
    h(a)
    = (G_1*G_2)'(a)
    &= (G_1'*G_2)(a)  \\
    &= \int_0^a \int_{-\infty}^0 \frac{\partial f_b(x)}{\partial b} g_1(x) \, \mathrm{d} x
    \int_{-\infty}^0 \hat{f}_{a-b}(y) g_2(-y) \, \mathrm{d}y \, \mathrm{d}b
    \\
    &= \int_{-\infty}^0 \int_0^\infty g_1(x) g_2(y) f(x,y) \, \mathrm{d} x\, \mathrm{d}y,
  \end{align*}
  with the result that we have shown
  \[
    f(x,y)
    =
    \int_0^a \frac{\partial f_b(x)}{\partial b} \hat{f}_{a-b}(-y) \, \mathrm{d}b,
  \]
  to be the density of $(V_1(a), \Xi'_1(a) - V_1(a))$.
  The scaling property of $\Xi'$ implies that $f_a(x) = a^{-2/3} f_1(a^{-2/3}x)$,
  and likewise for $\hat{f}_a$, which ensures that the integrand above is indeed
  measurable, as well as proving that $G_1$ is differentiable, which justifies the argument.

  Transforming with a linear map, it follows immediately that $(\Xi'_1(a), V_1(a))$ is absolutely
  continuous. Moreover, since the components of $\Xi'$ are independent, the same is true of
  the $\R^2\times\R^2$--valued random variable
  $(\Xi'(a), V(a))$. Let us write $F\colon \R^2\times\R^2 \to\R$ for its density.

  Let $T_x = \inf\{ a \ge 0 : \Xi_i(a) < x_i, i=1,2 \}$. If we take any
  non-negative, bounded measurable $G$ and $x\in \R^2$, standard arguments
  based around the Poisson random measure of jumps of $\Xi$ produce
  this calculation:
  \begin{align*}
    \Eb^{Q_0}[ G(x - \Xi(T_x)) ]
    &= \Eb^{Q_0} \biggl[
      \sum_{a>0} G(x-\Xi(a)) \mathds{1}_{\{\forall i: \Xi_i(a) < x_i\}}
      \mathds{1}_{\{\exists i: \Xi_i(a-) \ge x_i \}}
    \biggr] \\
    &= \int_{\R^2} \int_{\R^2}
    u(w) \pi(z)
    G(x-w-z) \mathds{1}_{\{\forall i: w_i+z_i<x_i\}}
    \mathds{1}_{\{\exists i: w_i \ge x_i \}} \, \mathrm{d} z \, \mathrm{d}w \\
    &=
    \int_{\R^2} \int_{\R^2}
    u(w) \pi(x-w-y)
    G(y) \mathds{1}_{\{\forall i: y_i > 0 \}}
    \mathds{1}_{\{\exists i: w_i \ge x_i\}}
    \, \mathrm{d}y \, \mathrm{d}w,
  \end{align*}
  where $u$ and $\pi$ are, respectively, the densities of the potential measure and Lévy
  measure of $\Xi$; the absolute continuity of the potential measure is given as part
  of \cite[Theorem~3.11]{KP}.
  It follows that $x - \Xi(T_x)$ is absolutely continuous, and its density,
  say $y\mapsto h^x(y)$, is jointly measurable in $x$ and $y$.

  Combining the results we have so far,
  \begin{align*}
    \Eb^{Q_0}\bigl[ G(Y(a))\bigr]
    &= \Eb^{Q_0}\bigl[ G(\Xi'(a) - \Xi(T_{V(a)})) \bigr] \\
    &= \int_{\R^2\times \R^2} F(z,x) \Eb^{Q_0} \bigl[G(z - \Xi(T_x))\bigr]
    \, \mathrm{d}z\, \mathrm{d} x \\
    &= \int_{\R^2\times\R^2\times \R^2}
    F(z,x) h^x(y) G(z-x+y)
    \, \mathrm{d}z\, \mathrm{d} x \, \mathrm{d}y .
  \end{align*}
  It follows directly that $Y(a)$ is absolutely continuous.

  \medskip \stepcounter{targetinv}
  $\triangleright$ \textit{Step \thetargetinv: Reduction to $L=0$.}
  We claim that it is enough to prove \cref{prop: target invariance Q_L} for $L=0$. To this end, assume
  that \cref{prop: target invariance Q_L} is true under $Q_0$.

  Then for $b\ge 0$, let $\mu_b(L)$ be the density
  of $S(b)$ under $Q_0$.
  Since $S$ is self-similar with index $\frac32$ (by \cref{prop: Z self-similar}), we first note that 
  \begin{equation} \label{eq: mu scaling}
    \mu_b(L)= b^{-2/3} \mu_1(b^{-2/3} L).	
  \end{equation}
  Let $f$ be any non-negative bounded measurable function defined on $\R_+^2$ and $g:x\in\R_+\mapsto \mathds{1}_{x\le c}$ where $c>0$ is arbitrary.
  Notice that by our bound in \cref{lem: bound S<c}, we have
  \begin{equation} \label{eq: assumption finite moment}
    \int_0^\infty \mathrm{d}b \, b^\alpha \Eb^{Q_0} \left[ f\left( \frac{Y(a+b)}{S(a+b)}\right) g(S(a+b))\right] < \infty,
  \end{equation}
  at least for $\alpha\in(-1,-2/3)$, regardless of $c$.

  On the one hand, by the claim under $Q_0$, for all $a\ge 0$, we have 
  \begin{equation} \label{eq: extension Q_L from Q_0}
    \int_0^\infty \mathrm{d}b \, b^\alpha \Eb^{Q_0} \left[ f\left( \frac{Y(a+b)}{S(a+b)}\right) g(S(a+b))\right] 
    =
    \Eb[f((U,1-U))] \cdot \int_0^\infty \mathrm{d}b \, b^\alpha \Eb^{Q_0} [ g(S(a+b))]. 		
  \end{equation}
  On the other hand, by the Markov property of $Y$ at time $b$ and \cref{prop: target invariance Q_L} under $Q_0$ again, 
  \begin{align*}
    \int_0^\infty \mathrm{d}b \, b^\alpha \Eb^{Q_0} \left[ f\left( \frac{Y(a+b)}{S(a+b)}\right) g(S(a+b))\right] 
    &=
    \int_0^\infty \mathrm{d}b \, b^\alpha \Eb^{Q_0}\left[\Eb^{Q_{Z(b)}} \left[ f\left( \frac{Y(a)}{S(a)}\right) g(S(a))\right] \right] \\
    &=
    \int_0^\infty \mathrm{d}b \, b^\alpha \int_0^{\infty} \mathrm{d}L \, \mu_b(L) \Eb^{Q_L}\left[ f\left( \frac{Y(a)}{S(a)}\right) g(S(a)) \right].
  \end{align*}
  Then by \eqref{eq: mu scaling}, 
  \begin{multline*}
    \int_0^\infty \mathrm{d}b \, b^\alpha \Eb^{Q_0} \left[ f\left( \frac{Y(a+b)}{S(a+b)}\right) g(S(a+b))\right] \\
    =
    \int_0^\infty \mathrm{d}b \, b^{\alpha-\frac23} \int_0^{\infty} \mathrm{d}L \, \mu_1(b^{-2/3}L) \Eb^{Q_L}\left[ f\left( \frac{Y(a)}{S(a)}\right) g(S(a))\right].
  \end{multline*}
  Finally, using the change of variables $b\mapsto B = b^{-2/3}L$, the above display becomes
  \begin{multline} \label{eq: extension Q_L markov}
    \int_0^\infty \mathrm{d}b \, b^\alpha \Eb^{Q_0} \left[ f\left( \frac{Y(a+b)}{S(a+b)}\right) g(S(a+b))\right] \\
    =
    \int_0^\infty   \mathrm{d}B \, \mu_1(B) \, \frac{3}{2} B^{-\frac32(1+\alpha)} \cdot  \int_0^{\infty} \mathrm{d}L \, L^{\frac12(1+3\alpha)} \Eb^{Q_L}\left[ f\left( \frac{Y(a)}{S(a)}\right) g(S(a)) \right].
  \end{multline}
  Writing $I(\alpha) := \int_0^\infty   \mathrm{d}B \, \mu_1(B) \, B^{-\frac32(1+\alpha)}$, we notice from our previous calculations that $I(\alpha)<\infty$ as soon as \eqref{eq: assumption finite moment} holds. Equating \eqref{eq: extension Q_L from Q_0} and \eqref{eq: extension Q_L markov} thus yields
  \begin{equation} \label{eq: equality mellin}
    \int_0^{\infty} \mathrm{d}L \, L^{\frac12(1+3\alpha)} \Eb^{Q_L}\left[ f\left( \frac{Y(a)}{S(a)}\right) g(S(a))\right]  
    =
    \Eb[f((U,1-U))] \cdot \int_0^{\infty} \mathrm{d}L \, L^{\frac12(1+3\alpha)} \Eb^{Q_L}[ g(S(a))].
  \end{equation}
  Equation \eqref{eq: equality mellin} proves that the two functions $L\mapsto \Eb^{Q_L} \left[ f\left( \frac{Y(a)}{S(a)}\right) g(S(a))\right]$ and $L\mapsto \Eb[f((U,1-U))] \Eb^{Q_L} [ g(S(a))]$ have the same Mellin transform at $s=\frac32(1+\alpha)$. Since this is satisfied on an open interval of $\alpha$, namely $\alpha\in(-1,-2/3)$, we conclude that, for almost every $L\ge 0$,
  \[
    \Eb^{Q_L}\left[ f\left( \frac{Y(a)}{S(a)}\right) g(S(a))\right]
    =
    \Eb[f((U,1-U))] \Eb^{Q_L}[ g(S(a))].
  \] 
  By a standard continuity argument, the latter equality extends to all $L\ge 0$. This is exactly the claim of \cref{prop: target invariance Q_L} under $Q_L$.

  \medskip\refstepcounter{targetinv}\label{step:random-time}
  $\triangleright$ \textit{Step \thetargetinv: Proof for a special random time.}
  We now restrict to $L=0$. It turns out to be more convenient to prove a
  variant of \cref{prop: target invariance Q_L} for some specific random times.
  More precisely, let $t>0$ be any (deterministic) time (we view $t$ as a time
  on the Brownian motion $W'$). Almost surely, $t$ is a \emph{pinched} time of
  $W'$ in the sense of \cref{sec: forward cones}. Consider the first forward
  cone-free time in $W'$ such that the corresponding excursion straddles $t$, that
  is $\uptau(\ell(t))$ where we recall from \cref{sec: forward cones} that
  $\ell$ denotes the (forward) cone-free local time. We then prove \cref{prop:
  target invariance Q_L} for the random time $a=\ell(t)$.
  We rewrite $Y(\ell(t)^-)$ using the Brownian motion seen from $t$: s
  ince we assumed $L=0$, we can glue $W$ onto the whole past of
  $W'$ seen from $t$ (shifted by $W'(t)$). Let
  $B(s) := W'(t-s)-W'(t)$ for $0\le s\le
  t$ and $B(s) := W(s)-W'(t)$ for $s>t$; by time-reversal, $B$ is a Brownian motion.
  Notice that the above $Y(\ell(t)^-)$
  now simply translates into the displacement of the backward cone excursion
  straddling time $t$ in $B$ (see \cref{fig: law S proof}). We can therefore
  make use of the structure of backward excursions of $B$. Write $\Eb^B$ for
  the expectation with respect to the new Brownian motion $B$, and recall from
  \cref{sec: backward cones} the notation for the backward cone excursion
  process $(\efrak(s), s> 0)$. For any non-negative measurable functions $f,g$, 
  \[
    \Eb\bigg[f\bigg(\frac{Y(\ell(t)^-)}{S(\ell(t)^-)}\bigg) g(S(\ell(t)^-)) \bigg]
    =
    \Eb^B \bigg[ \sum_{s>0} f\bigg(\frac{\efrak(s)(0)}{||\efrak(s)(0)||_1}\bigg) g(||\efrak(s)(0)||_1) \mathds{1}_{\zeta(\efrak(s))>t-\tfrak(s^-)>0}\bigg].
  \]
  By the compensation formula for the excursion process $\efrak$, this is
  \begin{multline*}
    \Eb\bigg[f\bigg(\frac{Y(\ell(t)^-)}{S(\ell(t)^-)}\bigg) g(S(\ell(t)^-)) \bigg] \\
    =
    \Eb^B \bigg[ \int_0^{\infty} \mathrm{d}\lfrak(s) \mathds{1}_{\tfrak(s^-)<t} \cdot \nlg \bigg(f\bigg(\frac{e(0)}{||e(0)||_1}\bigg) g(||e(0)||_1) \mathds{1}_{\zeta(e)>t-\tfrak(s^-)} \, \Big| \, \tfrak(s^-)\bigg)\bigg].
  \end{multline*}
  We now use the joint law of the duration and displacement under $\nlg$ (see \cref{prop: joint law displacement/duration} and \cref{rk: independence duration/position} (i)). We end up with 
  \begin{multline*}
    \Eb\bigg[f\bigg(\frac{Y(\ell(t)^-)}{S(\ell(t)^-)}\bigg) g(S(\ell(t)^-)) \bigg] \\
    =
    \Eb[f((U,1-U))] \cdot \Eb^B \bigg[ \int_0^{\infty} \mathrm{d}\lfrak(s) \mathds{1}_{\tfrak(s^-)<t} \cdot \nlg \bigg(g(||e(0)||_1) \mathds{1}_{\zeta(e)>t-\tfrak(s^-)} \, \Big| \, \tfrak(s^-)\bigg)\bigg].
  \end{multline*}
  Taking $f=1$ and comparing expressions, this forces
  \[
    \Eb\bigg[f\bigg(\frac{Y(\ell(t)^-)}{S(\ell(t)^-)}\bigg) g(S(\ell(t)^-)) \bigg]
    =
    \Eb[f((U,1-U))] \cdot \Eb[g(S(\ell(t)^-))],	
  \] 
  which was our claim.

  \begin{figure}
    \bigskip
    \begin{center}
      \includegraphics[scale=0.85]{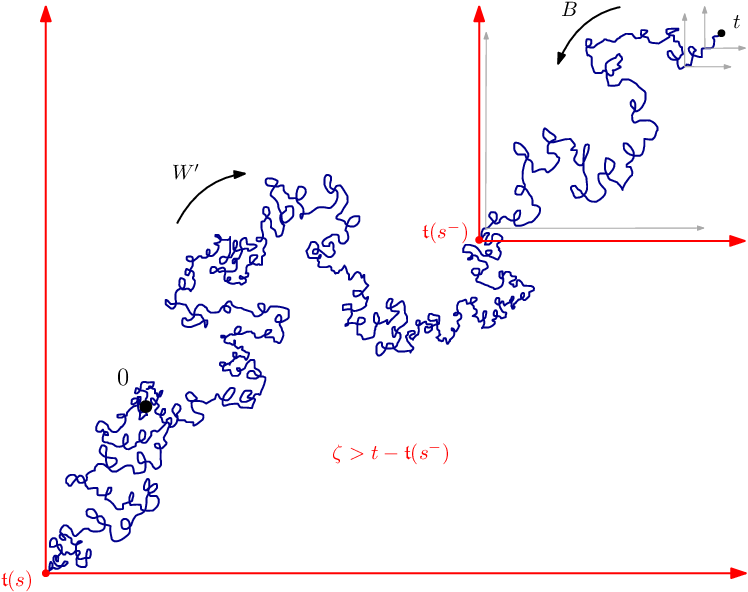}
    \end{center}
    \caption{The backward cone time process, seen backwards from time $t$. The backward cone excursions of $B$ are represented in grey. The excursion process is stopped when an excursion straddles $0$ (red), \textit{i.e.} when $\zeta >t-\tfrak(s^-)$.}
    \label{fig: law S proof}
  \end{figure}

  \medskip \stepcounter{targetinv}
  $\triangleright$ \textit{Step \thetargetinv: Concluding the proof for a deterministic time $a$.} 
  Let $q>0$ and $\mathrm{e}_q$ an independent exponential random variable with
  parameter $q$. 
  The result in Step~\ref{step:random-time} can be extended to include a function
  of the local time $\ell(t)$ using \cref{cor: idp local time}. This yields, for
  non-negative measurable
  functions $f,g,h$,
  \begin{equation} \label{eq: idp general time exponential}
    \Eb\bigg[ f\bigg(\frac{Y(\ell(\mathrm{e}_q)^-)}{S(\ell(\mathrm{e}_q)^-)}\bigg) g(S(\ell(\mathrm{e}_q)^-)) h(\ell(\mathrm{e}_q))\bigg]	
    =
    \Eb[f((U,1-U))] \cdot \Eb\big[ g(S(\ell(\mathrm{e}_q)^-)) h(\ell(\mathrm{e}_q))\big].
  \end{equation}
  On the other hand, the left-hand side may be rewritten by summing over forward {cone-free} times $a$:
  \begin{align*}
    \Eb\bigg[ f\bigg(\frac{Y(\ell(\mathrm{e}_q)^-)}{S(\ell(\mathrm{e}_q)^-)}\bigg) g(S(\ell(\mathrm{e}_q)^-)) & h(\ell(\mathrm{e}_q))\bigg] \\
    &=
    \Eb\bigg[ \int_0^{\infty} q\mathrm{e}^{-qt} f\bigg(\frac{Y(\ell(t)^-)}{S(\ell(t)^-)}\bigg) g(S(\ell(t)^-)) h(\ell(t)) \mathrm{d}t\bigg]	\\
    &=
    \Eb\bigg[ \sum_{a>0} h(a) f\bigg(\frac{Y(a^-)}{S(a^-)}\bigg) g(S(a^-)) \int_{\uptau(a^-)}^{\uptau(a)} q\mathrm{e}^{-qt} \mathrm{d}t\bigg] \\
    &=
    \Eb\bigg[ \sum_{a>0} h(a) f\bigg(\frac{Y(a^-)}{S(a^-)}\bigg) g(S(a^-)) \mathrm{e}^{-q\uptau(a^-)} \int_{0}^{\uptau(a)-\uptau(a^-)} q\mathrm{e}^{-qt} \mathrm{d}t\bigg],
  \end{align*}
  since for all $t\in (\uptau(a^-),\uptau(a))$, $\ell(t)=a$. By the compensation formula, this is
  \begin{multline} \label{eq: idp general time excursion}
    \Eb\bigg[ f\bigg(\frac{Y(\ell(\mathrm{e}_q)^-)}{S(\ell(\mathrm{e}_q)^-)}\bigg) g(S(\ell(\mathrm{e}_q)^-)) h(\ell(\mathrm{e}_q))\bigg] \\
    =
    \Eb\bigg[ \int_0^\infty \mathrm{d}a \, h(a) f\bigg(\frac{Y(a)}{S(a)}\bigg) g(S(a)) \mathrm{e}^{-q\uptau(a)}\bigg] \ndms\bigg(\int_{0}^{\zeta} q\mathrm{e}^{-qt} \mathrm{d}t\bigg).
  \end{multline}
  Taking $f=1$ and equating expressions, \eqref{eq: idp general time exponential} and \eqref{eq: idp general time excursion} imply 
  \[
    \int_0^\infty \mathrm{d}a \, h(a)\Eb\bigg[ f\bigg(\frac{Y(a)}{S(a)}\bigg) g(S(a)) \mathrm{e}^{-q\uptau(a)}\bigg]
    =
    \Eb[f((U,1-U))] \int_0^\infty \mathrm{d}a \, h(a) \Eb\big[  g(S(a)) \mathrm{e}^{-q\uptau(a)}\big]. 
  \]
  Since this holds for arbitrary $h$, it must hold for almost every $a\ge 0$ that  
  \[
    \Eb\bigg[ f\bigg(\frac{Y(a)}{S(a)}\bigg) g(S(a)) \mathrm{e}^{-q\uptau(a)}\bigg]
    =
    \Eb[f((U,1-U))] \cdot \Eb\big[ g(S(a)) \mathrm{e}^{-q\uptau(a)}\big]. 
  \]
  A continuity argument extends the previous identity to all $a\ge 0$. This proves our claim for any deterministic time $a$.
\end{proof}

The following result is a straightforward corollary of \cref{prop: target invariance Q_L}.
\begin{Cor} \label{cor: Markov S}
  The process $S$ defined in \eqref{eq: def Y(a) and S(a)} is a Markov process under $Q_L$.
\end{Cor}
\begin{proof}
  We prove the Markov property for two times $a$ and $b$ with $b>a$ {(the proof extends easily from there)}. Let $F, G$ two arbitrary test functions, which are measurable and non-negative. Then by the Markov property of $Y$,
  \[
    \Eb^{Q_L}[F(S(a)) G(S(b))] 
    = 
    \Eb^{Q_L}\big[F(S(a)) \Eb_{Y(a)} [G(S(b))]\big]. 
  \]
  We now use \cref{prop: target invariance Q_L} to get 
  \begin{align*}
    \Eb^{Q_L}[F(S(a)) G(S(b))] 
    &= 
    \Eb^{Q_L}\big[F(S(a)) \Eb_{US(a), (1-U)S(a)} [G(S(b))]\big] \\
    &=
    \Eb^{Q_L}\big[F(S(a)) \Eb^{Q_{S(a)}}[G(S(b))]\big]. 
  \end{align*}
  This proves the Markov property at times $a$ and $b$.
\end{proof}

We also have the following stronger version of \cref{prop: target invariance Q_L}.
\begin{Prop} \label{prop: target invariance Q_L strong}
  Let $a\ge 0$ and $L\ge 0$. Under $Q_L$, the law of $\frac{Y(a)}{S(a)}$ is that of $(U, 1-U)$ where $U$ is uniform in $(0,1)$ and independent of the whole process $(S(b),b\ge 0)$.
\end{Prop}
\begin{proof}
  The claim is actually a consequence of \cref{prop: target invariance Q_L} and the Markov property of $Y$. We only prove the independence of $\frac{Y(a)}{S(a)}$ and $(S(b),b\le a)$ to avoid notational clutter. Take a sequence of times $0\le a_1 < \ldots < a_n := a$, and non-negative measurable functions $F_1, \ldots, F_n$ and $G$. Then by repeated applications of the Markov property of $Y$ and \cref{prop: target invariance Q_L}, we obtain
  \begin{align*}
    &\Eb^{Q_L}\Big[F_1(S(a_1))\cdots F_n(S(a_n)) G\Big(\frac{Y(a_n)}{S(a_n)}\Big)\Big] \\
    &=
    \Eb^{Q_L}\Big[F_1(S(a_1))\Eb^{Q_{S(a_1)}}\Big[F_2(S(a_2-a_1))\cdots F_n(S(a_n-a_1)) G\Big(\frac{Y(a_n-a_1)}{S(a_n-a_1)}\Big)\Big]\Big] \\
    &=
    \Eb^{Q_L}\Big[F_1(S(a_1))\Eb^{Q_{S(a_1)}}\Big[F_2(S(a_2-a_1))\cdots \Eb^{Q_{S(a_n-a_{n-1})}}\Big[ F_n(S(a_n-a_{n-1})) G\Big(\frac{Y(a_n-a_{n-1})}{S(a_n-a_{n-1})}\Big)\Big]\cdots\Big]\Big].
  \end{align*}
  Applying again \cref{prop: target invariance Q_L} and unfolding the expectations, we get our claim:
  \[
    \Eb^{Q_L}\Big[F_1(S(a_1))\cdots F_n(S(a_n)) G\Big(\frac{Y(a_n)}{S(a_n)}\Big)\Big]
    =
    \Eb^{Q_L}\Big[F_1(S(a_1))\cdots F_n(S(a_n))\Big] \Eb^{Q_L}\Big[G\Big(\frac{Y(a_n)}{S(a_n)}\Big)\Big].	
  \]
  The same arguments prove the independence with the whole process $S$. 
\end{proof}

We apply \cref{prop: target invariance Q_L} to proving the target-invariance property of $\sle_6$ in the $\sqrt{8/3}$--quantum disc, \textit{i.e.} \cref{prop: intro target inv cone} and \cref{cor: intro target inv lqg} mentioned in the introduction. We first state the result in terms of Brownian excursions. {It will actually hold more generally under $P^z$ for $z \in \R_+^2\setminus \{0\}$. Define, as in \eqref{eq: intro def P bar}, for all $z\in\R_+^2 \setminus \{0\}$,
\begin{equation} \label{eq: def P bar}
  \overline{P}^z (\mathrm{d}T,\mathrm{d}e)
  :=
  \sqrt{3} \lVert z\rVert_1^{-2} \mathds{1}_{\{ 0\le T\le \zeta(e)\}} \mathrm{d}T P^z(\mathrm{d}e).
\end{equation}
It can be checked from \cref{prop: joint law displacement/duration} that $\Eb^{P^z}[\zeta] = {\lVert z \rVert_1^2}/{\sqrt{3}}$, which ensures that $\overline{P}^z$ is a probability measure on $\R_+ \times E$.
} 
Recall the notation in Sections \ref{sec: intro main result cones} and \ref{sec: further}: in particular, we recall that $\varsigma^t$ is the local time for forward cones towards time $t$, and $Z^t$ is the branch towards time $t$, as in \eqref{eq: cal Z and Z^t} {(note that all these definitions extend naturally to the case when $z\in \R_+^2 \setminus \{0\}$)}.
\begin{Prop} \label{prop: target invariance}
  {Let $z\in \R_+^2 \setminus \{0\}$.} Under the law $\overline{P}^z$ defined in \eqref{eq: def P bar}, the following property holds: for all $a\ge 0$, on the event that $\varsigma^{T}>a$, $\frac{\mathcal{Z}^T(a)}{Z^T(a)}$ is independent of $(Z^T(b), b\ge 0)$ and distributed as $(U,1-U)$ with $U$ uniform in $(0,1)$. 
\end{Prop}

\noindent {\cref{prop: intro target inv cone} implies the result of \cref{cor: intro target inv lqg} on} $\sle_6$ explorations of a $\sqrt{8/3}$--quantum disc $(\D,h,-i)$. The proof follows directly from the mating-of-trees theorem (\cref{thm: MoT complete}).
In the setting of \cref{sec: intro LQG}, consider a unit boundary length $\gamma$--quantum disc $(\D,h,-i)$ with law weighted by its total quantum area $\mu_h^\gamma(\D)$, and given $h$, we sample a point $z^{\bullet}$ in $\D$ according to the quantum area measure $\mu_h^\gamma$.
We then look at the branch $\eta^{z^\bullet}$ of the space-filling SLE$_6$ exploration targeted at $z^\bullet$ and define $(L^\bullet, R^\bullet)$ as the left and right quantum boundary length process of the component containing this point, when $\eta^{z^\bullet}$ is parametrised by quantum natural time. We write $Z^\bullet := L^\bullet + R^\bullet$ for the total boundary length process, and $\varsigma^\bullet$ for the duration of the branch $\eta^{z^\bullet}$.

\begin{Cor}\label{cor: target invariance} 
  Let $z^{\bullet}$ a point sampled in $\D$ according to the Liouville measure, biased by the quantum area of $\D$, and denote by $\varsigma^{\bullet}$ the quantum natural time towards $z^{\bullet}$.
  For all $a\ge 0$, on the event that $\varsigma^{\bullet}>a$,  $\left( \frac{L^\bullet(a)}{Z^\bullet(a)},  \frac{R^\bullet(a)}{Z^\bullet(a)}\right)$ is independent of $(Z^\bullet(b), b\ge 0)$ and distributed as $(U,1-U)$ with $U$ uniform in $(0,1)$.
\end{Cor}
\noindent In other words, given the total boundary length, we can resample the target point uniformly on the boundary of the domain at any time $a$.

\begin{proof}[Proof of \cref{prop: target invariance}]
  {We prove the result when $z\in (\R_+^*)^2$. The case when $z$ is at the boundary is then obtained through the convergence statement of \cref{prop: convergence measures}.}
  The proof is a typical application of \cref{prop: target invariance Q_L strong} 
  and Bismut's description of $\nlg$. 
  Fix $a\ge 0$. Let $F$ and $G$ two non-negative
  measurable functions. By \cref{thm: Bismut},
  \[
    \overline{\nlg}\Big(F(Z^T(b),b\ge 0)G\Big(\frac{\mathcal{Z}^T(a)}{Z^T(a)}\Big) \mathds{1}_{\varsigma^{T} >a} \Big)	
    =
    \overline{c} \int_a^\infty \mathrm{d}A \cdot \Eb\Big[ F(S(A-b), b\le A) G\Big(\frac{Y(A-a)}{S(A-a)}\Big)\Big],
  \]
  with $S$ and $Y$ defined in \eqref{eq: def Y(a) and S(a)}. The independence given in \cref{prop: target invariance Q_L strong} implies that
  \[
    \overline{\nlg}\Big(F(Z^T(b),b\ge 0)G\Big(\frac{\mathcal{Z}^T(a)}{Z^T(a)}\Big) \mathds{1}_{\varsigma^{T} >a} \Big)	
    =
    \overline{c} \Eb[G((U,1-U))] \cdot \int_a^\infty \mathrm{d}A \cdot \Eb\big[F(S(A-b), b\le A)\big],
  \]
  where $U$ is uniform in $(0,1)$.
  On the other hand, applying the above to $G=1$ and substituting, we obtain
  \[
    \overline{\nlg}\Big(F(Z^T(b),b\ge 0)G\Big(\frac{\mathcal{Z}^T(a)}{Z^T(a)}\Big) \mathds{1}_{\varsigma^{T} >a} \Big)	
    =
    \Eb[G((U,1-U))] \cdot \overline{\nlg}\Big(F(Z^T(b),b\ge 0)\mathds{1}_{\varsigma^{T} >a} \Big). 
  \]
  We obtain the claim in the statement by disintegrating over the start point
  under $\nlg$: we {multiply} $F$ by an extra function of $Z^T(0)$, apply the equation
  above, and obtain that,
  for almost every $z\in (\R_+^*)^2$,
  \[
    \Eb^{P^z} \Big[F(Z^T(b),b\ge 0)G\Big(\frac{\mathcal{Z}^T(a)}{Z^T(a)}\Big) \mathds{1}_{\varsigma^{T} >a} \Big]
    =
    \Eb[G((U,1-U))] \cdot \Eb^{P^z} \Big[F(Z^T(b),b\ge 0) \mathds{1}_{\varsigma^{T} >a} \Big].
  \]
  A continuity argument shows that the above equality is valid for all $z\in (\R_+^*)^2$, thus concluding the proof.
\end{proof}

\begin{Rk}
  We can find the law of $Z^T$ explicitly. This will be done in
  \cref{prop: law uniform exploration} as we first need to describe the
  distribution of $S$.
\end{Rk}


\subsection{A Brownian motion construction of the spectrally positive \texorpdfstring{$3/2$}{3/2}--stable process conditioned to remain positive} \label{sec: BM construction conditioning}
In the previous section, we defined a process $S$ which we proved to be Markov.
We now take one step further and determine the law of the process $S$, leading to \cref{prop: intro pathwise 3/2 stable cond}. This
will be done by determining the generator of $S$, and we first need some
technical results.

The following pair of lemmas is valid when $-\Xi$ is a two-dimensional
$\beta$--stable Lévy process with Lévy measure 
$c_3 (x+y)^{-(\beta+2)}\mathds{1}_{\{x,y>0\}}\mathrm{d}x\mathrm{d}y$
and $\Xi'$ is a two-dimensional $\alpha$--stable Lévy process with independent
components, such that for $i=1,2$, the Lévy measure of the $i$-th component is
$c_i x^{-(\alpha+1)}\mathds{1}_{\{x>0\}}\mathrm{d}x$.
We also need to assume that $\beta\in(0,1)$ and $\alpha = \beta+1$.
We will apply these results to $\Xi := W\circ \tfrak$ and $\Xi':= W'\circ \uptau$, so that $\alpha = 3/2$, $\beta = 1/2$,
$c_1=c_2=1$ and $c_3=\frac{3^{1/8}}{8}$,
but it is interesting to observe that they are valid more generally.

For $z\in\R^2$, we write $z=(z_1,z_2)$ the co-ordinates of $z$, and
introduce,
\[
  V_i(t)
  :=
  \underset{s\le t}{\inf} \; \Xi'_i(s), \quad i=1,2, \ t\ge 0,
\]
the running infimum of the $i$-th component of $\Xi'$, as well as
\[
  \sigma_i(z)
  \coloneqq \inf\{s\ge 0: \Xi_i(s) \le -z \},
  \quad z\ge 0,
\]
the first passage time of the $i$-th component of $\Xi$ below a level.
We take $\Xi$ and $\Xi'$ (and thence also $V$) to start from $0$ under $\Pb$.
With these conventions and definitions in place, we can state our two preparatory lemmas.

\begin{Lem}
  \label{l:stable-moments}
  For all $p>0$,
  \[
    \Eb[ \sigma_1(1)^p ]
    =
    \frac{\Gamma(p+1)}
    {\Gamma(\beta p+1)} \biggl( - \frac{c_3}{\beta+1} \Gamma(-\beta) \biggr)^{-p}.
  \]
  and
  \[
    \Eb[(-V_i(1))^p]
    =
    \frac{\Gamma(p+1)}{\Gamma\bigl(\frac{p}{\alpha}+1\bigr)}
    \bigl( c_i \Gamma(-\alpha) \bigr)^{\frac{p}{\alpha}}.
  \]
  In particular,
  \[
    \Eb[\sigma_1(1)]\Eb[(-V_i(1))^{1+\beta}]
    = \frac{c_i}{c_3}(\beta+1).
  \]
\end{Lem}
\begin{proof}
  It is shown in \cite[Proposition 1(iii)]{Bin-occ} that the distribution of $\sigma_1(1)$
  is Mittag-Leffler, and its moments are computed by \cite[\S 0.3]{Pit-csp}.
  Since we need to trace the normalisation constants, and the proof is quite short,
  we sketch it here.
  For computations of the constants appearing in the Laplace exponents, we refer to
  \cite[Remark~14.20 and Example~46.7]{Sato}.

  We begin by observing that $-\Xi_1$ is a subordinator with
  L\'evy measure $\frac{c_3}{\beta+1} x^{-(\beta+1)} \mathrm{d}x\mathds{1}_{\{x>0\}}$
  and Laplace transform given by
  $\Eb[\mathrm{e}^{z \Xi_1(1)}] = \mathrm{e}^{\frac{c_3\Gamma(-\beta)}{\beta+1} z^{\beta}}$
  for $z\ge 0$ (note that $\Gamma(-\beta)<0$).
  Observing that $\Pb(\sigma_1(1) > a) = \Pb(-\Xi_1(a) \le 1)$, and using integration
  by parts and the scaling property of $-\Xi_1$, we deduce that
  \[
    \Eb[\sigma_1(1)^p] = \int_0^\infty \mathrm{d}a \cdot p a^{p-1} \Pb(-\Xi_1(a) \le 1)= \int_0^\infty \mathrm{d}a \cdot p a^{p-1} \Pb(-\Xi_1(1) \le a^{-1/\beta})
    = \Eb[(-\Xi_1(1))^{-\beta p}],
  \]
  for any $p$ making either side finite.
  The moment on the right-hand side can then be computed by observing that
  $(-\Xi_1(1))^{-\theta} 
  = \frac{1}{\Gamma(\theta)} \int_0^\infty u^{\theta-1} \mathrm{e}^{u\Xi_1(1)} \, \mathrm{d}u$
  for $\theta>0$,
  and applying Fubini's theorem. Standard calculations give the result in the statement
  of the lemma.

  The study of $V_1(1)$ can in fact be reduced to a very similar calculation, as was
  shown in \cite{Bin-maxima}.
  The process $\Xi'_i$ is a one-dimensional spectrally positive Lévy process
  with Lévy measure $c_i x^{-(\alpha+1)} \, \mathrm{d}x \mathds{1}_{\{x>0\}}$.
  Its Laplace exponent is given by $\Eb[\mathrm{e}{-z\Xi'_i(1)}] = \mathrm{e}{\psi_i(-z)}$,
  with $\psi_i(-z) = c_i\Gamma(-\alpha) z^\alpha$,
  where $z\ge 0$.
  By \cite[Proposition 1]{Bin-maxima}, the process $-V_i$ is the inverse of a $\frac{1}{\alpha}$--stable 
  subordinator, say $(-V_i)^{-1}$,
  whose Laplace exponent is the left-inverse of $\psi_i$, namely $\Eb[\mathrm{e}^{-z(-V_i)^{-1}(1)}]=\exp(-\bigl(c_i\Gamma(-\alpha) \bigr)^{-\frac{1}{\alpha}}z^{\frac{1}{\alpha}})$, $z\ge 0$.
  The same considerations as above lead us back to the relationship
  \[
    \Eb[(-V_i(1))^p] = \Eb[(-V_i)^{-1}(1)^{-\frac{p}{\alpha}}]
  \]
  and thence, taking care with the normalisation constants, to the expression in the
  statement.
  The final part of the statement follows by substituting and simplifying.
\end{proof}

  As in \eqref{eq: sfrak def}, we may define more generally
  \[
    \sfrak(t):=
    \inf\{s\ge 0, \; \Xi_i(s) \le V_i(t), \; i=1,2\}. 
  \]
  We also extend our definition of $Q_L$ in \eqref{eq: Q_L def} to this setting, \textit{i.e.} under $Q_L$, $\Xi$ and $\Xi'$ are independent, $\Xi$ starts at $0$ and $\Xi'$ at $L(U,1-U)$, where $U$ is independent uniform in $(0,1)$. We stress that the assumption $\alpha=\beta+1$ is still in force. 

\begin{Lem} \label{Lem: asympt LG time-change}
  Let $L> 0$. Then,
  \[
    \lim_{\delta\to 0} \Eb^{Q_L}[\sfrak(\delta)] = \frac{c_1+c_2}{c_3 L}. 	
  \]
\end{Lem}
\begin{proof}
  Define, for $z\in\R^2$, 
  \[
    T(z)
    :=
    \inf\{ s\ge 0:  \; \Xi_i(s) \le z_i, \;  i=1,2\}
    = \sigma_1(-z_1) \vee \sigma_2(-z_2).
  \]
  Note that the self-similarity of $\Xi$ or $\Xi'$ is inherited by $T$ and by
  $V_i$ ($i=1,2$); namely, for all $c>0$,
  \begin{equation} \label{eq: scaling T and V}
    T(z) 
    \overset{\mathrm{(d)}}{=}	
    c^{\beta} T(z/c), 
    \quad \text{and} \quad
    (V_i(t), i=1,2) 
    \overset{\mathrm{(d)}}{=}	
    (c^{1/\alpha}V_i(t/c), i=1,2).
  \end{equation}
  Let $\delta>0$. We are after 
  \begin{align*}
    \Eb^{Q_L}[\sfrak(\delta)]	
    &=
    \frac{1}{L} \int_0^L \Eb[T(x+V_1(\delta), L-x+V_2(\delta))] \mathrm{d}x \\
    &\overset{\eqref{eq: scaling T and V}}{=}
    \frac{1}{L} \int_0^L \Eb[T(x+\delta^{1/\alpha}V_1(1), L-x+\delta^{1/\alpha}V_2(1))] \mathrm{d}x.
  \end{align*}
  We partition the integral into the following events:
  \begin{align}
    E_1(x) &:= \{x+\delta^{1/\alpha}V_1(1)<0, L-x+\delta^{1/\alpha}V_2(1)\ge 0\}, \notag \\
    E_2(x) &:= \{x+\delta^{1/\alpha}V_1(1)\ge 0, L-x+\delta^{1/\alpha}V_2(1)< 0\}, \notag \\
    E_3(x) &:= \{x+\delta^{1/\alpha}V_1(1)<0, L-x+\delta^{1/\alpha}V_2(1)< 0\}. \label{eq: partition events}
  \end{align}
  The events $E_1(x)$ and $E_2(L-x)$ are symmetric in $(V_1,V_2)$, so we consider
  $E_1(x)$. We will show later that the contribution of $E_3$ can be
  ignored as $\delta$ goes to $0$. 

  On $E_1$, we have $L-x+\delta^{1/\alpha}V_2(1)\ge 0$ and hence 
  \[
    T(x+\delta^{1/\alpha}V_1(1), L-x+\delta^{1/\alpha}V_2(1)) = T(x+\delta^{1/\alpha}V_1(1), 0).
  \]
  Therefore, this term is
  \begin{align*}
    I_1(L,\delta)
    &\coloneqq
    \frac{1}{L} \Eb\bigg[\int_0^L T(x+\delta^{1/\alpha}V_1(1), L-x+\delta^{1/\alpha}V_2(1)) \mathds{1}_{E_1(x)} \mathrm{d}x\bigg] \\
    &=
    \frac{1}{L} \Eb\bigg[\int_0^{L\wedge (-\delta^{1/\alpha}V_1(1)) \wedge (L+\delta^{1/\alpha}V_2(1))} T(x+\delta^{1/\alpha}V_1(1), 0)  \mathrm{d}x\bigg].
  \end{align*}
  By the change of variables $x=\delta^{1/\alpha}y$
  and then the scaling relation \eqref{eq: scaling T and V}, 
  \begin{align*}
    I_1(L,\delta)
    &=
    \frac{1}{L} \delta^{1/\alpha} \Eb\bigg[\int_0^{\delta^{-1/\alpha} L\wedge (-V_1(1)) \wedge (\delta^{-1/\alpha}L+V_2(1))} T(\delta^{1/\alpha}(y+V_1(1)), 0) \mathrm{d}y\bigg].
    \\
    &=
    \frac{1}{L} \delta^{\frac{1+\beta}{\alpha}} \Eb\bigg[\int_0^{\delta^{-1/\alpha} L\wedge (-V_1(1)) \wedge (\delta^{-1/\alpha}L+V_2(1))} T(y+V_1(1), 0)  \mathrm{d}y\bigg].
  \end{align*}
  Notice that $\frac{1+\beta}{\alpha} = 1$. By monotone convergence, we thus get
  \[
    \frac{1}{\delta}I_1(L,\delta)
    \underset{\delta\to 0}{\longrightarrow}
    \frac{1}{L} \Eb\bigg[\int_0^{-V_1(1)} T(y+V_1(1), 0)  \mathrm{d}y\bigg].
  \] 
  Recall that the process $z\mapsto \sigma_1(z) = T(-z, 0)$ is the first passage time process of $\Xi_1$ below $-z$. By the change of variables $y=-V_1(1)z$ and the scaling relation \eqref{eq: scaling T and V} again, the above limit is
  \begin{align}
    \lim_{\delta \to 0} \frac{1}{\delta}I_1(L,\delta)
    =
    \frac{1}{L} \Eb\bigg[\int_0^{-V_1(1)} T(y+V_1(1), 0)  \mathrm{d}y\bigg]	
    &=
    \frac{1}{L} \Eb\bigg[ \sigma_1(1) \int_0^{1}  (-V_1(1)(1-z))^{\beta} (-V_1(1)) \mathrm{d}z\bigg] \notag \\
    &=
    \frac{1}{L}\frac{1}{1+\beta} \Eb\big[ \sigma_1(1) (-V_1(1))^{1+\beta}\big]. \label{eq: T integrable}
  \end{align}
  Hence, by independence of $\Xi$ and $\Xi'$ and Lemma~\ref{l:stable-moments},
  \[
    \lim_{\delta \to 0} \frac{1}{\delta}I_1(L,\delta)
    =
    \frac{1}{L}\frac{1}{1+\beta} \Eb[ \sigma_1(1)]\Eb\big[ (-V_1(1))^{1+\beta}\big]
    = \frac{c_1}{c_3 L}.
  \]
  The analysis of case $E_2$ is identical, and we get for this term
  \[
    \lim_{\delta \to 0} \frac{1}{\delta}I_2(L,\delta)
    = \frac{c_2}{c_3 L}.
  \]

  Finally, we need to show that we can neglect $E_3$ in the limit. 
  For this, we first use the same change of variables $x=\delta^{1/\alpha}y$ as above, turning $I_3(L,\delta)$ into
  \[
    I_3(L,\delta)
    =
    \frac{\delta^{1/\alpha}}{L} 
    \Eb\biggl[
      \int_{0\vee (\delta^{-1/\alpha}L+V_2(1))}^{\delta^{-1/\alpha}L\wedge (-V_1(1))}
      T\big(\delta^{1/\alpha}((y+V_1(1), \delta^{-1/\alpha}L-y+V_2(1)))\big)
      \mathrm{d}y
    \biggr].
  \]
  By scaling, we get
  \[
    I_3(L,\delta)
    =
    \frac{\delta^{\frac{1+\beta}{\alpha}}}{L} 
    \Eb\biggl[
      \int_{0\vee (\delta^{-1/\alpha}L+V_2(1))}^{\delta^{-1/\alpha}L\wedge (-V_1(1))}
      T(y+V_1(1), \delta^{-1/\alpha}L-y+V_2(1))
      \mathrm{d}y
    \biggr].
  \]
  Then we note that the integral vanishes on the event $\{\delta^{-1/\alpha}L + V_2(1) \ge (-V_1(1))\}$. Moreover, we remark that for $x,y,a,b>0$ we have $T(x-a,y-b)\le T(-a,-b)$. Therefore we can bound the above display by
  \begin{multline*}
    I_3(L,\delta)
    \le 
    \delta^{\beta/\alpha} \Eb\bigl[T(V_1(1),V_2(1)) 
    \mathds{1}_{\{-(V_1(1)+V_2(1))>\delta^{-1/\alpha}L\}} \bigr] \\
    \le
    \delta^{\beta/\alpha} \Eb\bigl[T(-1,-1)  \lVert V(1)\rVert_1^\beta
    \mathds{1}_{\{-(V_1(1)+V_2(1))>\delta^{-1/\alpha}L\}} \bigr].
  \end{multline*}
  By independence of $\Xi$ and $\Xi'$ and the Cauchy-Schwarz inequality, we deduce that 
  \[
    I_3(L,\delta)
    \le 
    \delta^{\beta/\alpha}
    \Eb[T(-1,-1)]
    \Pb(-(V_1(1)+V_2(1))>\delta^{-1/\alpha}L)^{1/2}
    \Eb\bigl[ \lVert V(1)\rVert_1^{2\beta}\bigr]^{1/2}.  
  \]


  We consider the three terms appearing on the right-hand side above.
  The term 
  \[\Eb[T(-1,-1)] = \Eb[\sigma_1(1) \vee \sigma_2(1)] \le \Eb[\sigma_1(1)]+\Eb[\sigma_2(2)],\] 
  is
  finite by Lemma~\ref{l:stable-moments}. For the second term, raising both sides to the power $p\ge 0$ and applying Markov's inequality, we obtain
    \[
      \Pb(-(V_1(1)+V_2(1))>\delta^{-1/\alpha}L) \le \frac{\delta^{p\alpha}}{L^p} \Eb[(-(V_1(1)+V_2(1)))^p].
    \] 
  The above moment is bounded for any $p \ge 1$ 
  by Lemma~\ref{l:stable-moments}, so by taking $p$ large enough we see that this term decays faster than any power in $\delta$ as $\delta\to 0$.
  Finally, Lemma~\ref{l:stable-moments} shows that the third term is finite. This concludes the proof of \cref{Lem: asympt LG time-change}.
\end{proof}

We now come back to the Brownian motion picture. Recall the definition of the process $S$ introduced in \eqref{eq: def Y(a) and S(a)}. The main result of this section is the following explicit description of the law of $S$.

\begin{Thm} \label{thm: main cone end} 
  Under $Q_L$, $S$ is a spectrally positive $\frac{3}{2}$--stable Lévy process conditioned to remain positive. More precisely, $S$ has the law of the process $\xi^{\uparrow}$ described in \cref{sec: pssMp and Levy}, with $\alpha=\frac32$ and $c_{\Lambda}=2$.
\end{Thm}

\noindent Note that this gives \cref{prop: intro pathwise 3/2 stable cond} from the introduction.

\begin{Rks}
  \begin{enumerate}
    \item
      \cref{thm: main cone end} can be seen as a pathwise construction of the
      $\frac{3}{2}$--stable Lévy process conditioned to remain positive from a
      pair of planar Brownian motions. It forms a two-dimensional analogue to
      the one-dimensional construction of \cite{bertoin1993splitting} (in the
      special case of the $\frac32$--stable process). It would be interesting
      to see whether our construction extends to stable process with other
      indices.
    \item 
      It is known that the growth-fragmentation processes $\Xbf^{\alpha}$ of
      Bertoin, Budd, Curien and Kortchemski (see \cref{sec: GF process}) are
      closely related to $\alpha$--stable processes conditioned to remain
      positive or to be absorbed continuously at $0$, which appear in the
      spinal structure of the growth-fragmentation processes (see \cite{BBCK}).
      These processes also arise, in a scaling limit, from the
      peeling exploration of variants of Boltzmann planar maps
      \cite[Proposition 6.6]{BBCK}. For example, the stable process conditioned
      to be absorbed continuously at $0$ shows up in the case of \emph{pointed}
      planar maps, which are a size-biased version of the planar maps
      considered in \cite{BBCK}. \cref{thm: main cone end} states a result of a
      similar flavour for cone excursions, for $\alpha=3/2$. See also \cref{prop: law uniform exploration} for the process conditioned to be absorbed continuously at $0$.
  \end{enumerate}
\end{Rks}

\begin{proof}
  By \cref{cor: Markov S}, we know that under $Q_L$, the process $S$ is Markov.
  Our claim will follow once we identify the generator of $S$ with
  $\mathscr{G}_{3/2}$ in \eqref{eq: generator conditioned stable}. More
  precisely, {it is enough to} show that for $L>0$ and all function 
  {$f\in \{f:[0,\infty] \to \R, \; f, xf' \text{ and }
  x^2f'' \text{ are continuous on } [0,\infty]\} \subset \text{Dom}(\mathscr{G}_{3/2})$,} 
  \begin{multline} \label{eq: S generator}
    \frac{1}{\delta}\big(\Eb^{Q_L}[f(S(\delta))] - f(L)\big) \\
    \underset{\delta\to 0}{\longrightarrow} 
    (c_1+c_2)\left(\int_{0}^\infty \frac{f(L+z)-f(L)}{L} \frac{\mathrm{d}z}{z^{3/2}} + \int_{0}^\infty (f(L+z)-f(L)-zf'(L)) \frac{\mathrm{d}z}{z^{5/2}} \right).
  \end{multline}
  Note that $c_1+c_2 =2$ in our case, giving the constant
  $c_{\Lambda} = 2$ in the statement.

  For ease of notation it will be convenient to set $\Sigma := (W\circ \tfrak)_1 + (W\circ \tfrak)_2$ and $\Sigma' := (W'\circ \uptau)_1 + (W'\circ \uptau)_2$.
  We also extend the definition of $f$ by declaring that $f(x)=0$ for $x<0$.
  We begin by splitting the expectation as follows:
  \begin{equation}\label{eq: generator split}
    \Eb^{Q_L}[f(S(\delta))] - f(L)
    =
    \Eb^{Q_L}\big[f(S(\delta)) - f(\Sigma'(\delta))\big] 
    + \big(\Eb^{Q_L}[f(\Sigma'(\delta))]- f(L)\big).
  \end{equation}
  We recall that under $Q_L$, $\Sigma'$ starts from $L$.
  The second term is easier to deal with. Indeed, we note that $\Sigma'(\delta)$ is a spectrally positive $\frac32$--stable Lévy process with Lévy measure $(c_1+c_2) \mathds{1}_{x>0} \frac{\mathrm{d}x}{x^{5/2}}$. Therefore its generator is given by \cite[Section 3.1]{CC} as
  \[
    \frac{1}{\delta} \big(\Eb^{Q_L}[f(\Sigma'(\delta))]- {f(L)}\big)
    \underset{\delta\to 0}{\longrightarrow} 
    (c_1+c_2) \int_0^\infty (f(L+z)-f(L)-zf'(L)) \frac{\mathrm{d}z}{z^{5/2}}, 
  \]
  thus explaining the second term in the expression \eqref{eq: S generator}.

  It remains to deal with the first term of \eqref{eq: generator split}. For technical reasons that will appear later on, we further split this term as:
  \begin{multline} \label{eq: Sigma' split cutoff}
    \Eb^{Q_L}\big[f(S(\delta)) - f(\Sigma'(\delta))\big] \\
    =
    \Eb^{Q_L}\big[\mathds{1}_{\{\Sigma'(\delta)\ge L/2\}}(f(S(\delta)) - f(\Sigma'(\delta)))\big]
    +
    \Eb^{Q_L}\big[\mathds{1}_{\{\Sigma'(\delta)< L/2\}}(f(S(\delta)) - f(\Sigma'(\delta)))\big].
  \end{multline}
  The second term of \eqref{eq: Sigma' split cutoff} can be bounded as
  \[
    \Eb^{Q_L}\big[\mathds{1}_{\{\Sigma'(\delta)< L/2\}}(f(S(\delta)) - f(\Sigma'(\delta)))\big]
    \le
    2 \lVert f \rVert_{\infty} \cdot \Pb^{Q_L}(\Sigma'(\delta) < L/2).
  \]
  The latter tail probability is sublinear in $\delta$, as can be seen by a Chernoff bound. Indeed for $q>0$, using the formula for the Laplace exponent $\Psi_\alpha$ of $\Sigma'$ in \eqref{eq: Psi stable >0}, we obtain
  \begin{equation} \label{eq: generator cutoff sublinear}
    \Pb^{Q_L}(\Sigma'(\delta) < L/2)
    \le
    \Pb(\mathrm{e}^{-qL}\mathrm{e}^{-q\Sigma'(\delta)} > \mathrm{e}^{-q L/2})
    \le
    \mathrm{e}^{-q L/2}\mathrm{e}^{\Psi_\alpha(-q)\delta}
    =
    \mathrm{e}^{-q L/2}\mathrm{e}^{(c_1+c_2)\Gamma(-3/2) q^{3/2}\delta}.
  \end{equation}
  The result follows by taking, say, $q=\delta^{-2/3}$.

  We now explain how to deal with the first term of \eqref{eq: Sigma' split cutoff}.
  To do so, we first use that the process $-\Sigma$ is a $\frac12$--stable subordinator with Lévy measure $c_3\frac{\mathrm{d}z}{z^{3/2}}$. As a consequence, its generator is given by the formula
  \[
    \mathscr{H} f(L) :=  c_3\int_0^\infty (f(L+z)-f(L)) \frac{\mathrm{d}z}{z^{3/2}}.
  \]
  By standard arguments (see \textit{e.g.} \cite[Proposition VII.1.6]{RY}), we deduce that for any $x\ge 0$, the process 
  \[
    M_s^x(f):= 
    f(x-\Sigma(s)) - f(x) - c_3\int_0^s  \mathrm{d}u \int_0^\infty \big(f(x-\Sigma(u)+z)-f(x-\Sigma(u))\big) \frac{\mathrm{d}z}{z^{3/2}}, \quad s\ge 0,
  \]
  is a martingale. Under the conditional law given $W'\circ \uptau$, the
  process $M^{\Sigma'(\delta)}(f)$ is therefore a martingale. Furthermore,
  and still conditional on $W'\circ \uptau$, the variable $\sfrak(\delta)$ is a stopping
  time for $W\circ\tfrak$, which is almost surely finite. Moreover,
  observe that our assumptions on $f$ imply that $f$ is bounded and has
  bounded first order derivative on any interval of the form $[a_0,\infty)$,
  $a_0>0$. Thus for any $a_0>0$, there exists $C>0$ such that for $a>a_0$
  and $z\ge 0$,
  \begin{equation} \label{eq: bound derivative f}
    |f(a+z)-f(a)| \le C (z\wedge 1),  
  \end{equation}
  Now on the event $\{\Sigma'(\delta)\ge L/2\}$, since $-\Sigma$ is always positive, we see that $\Sigma'(\delta)-\Sigma(u)>L/2$, and therefore we have the bound
  \[
    M^{\Sigma'(\delta)}_{\sfrak(\delta)}(f)
    \le
    C'(1+\sfrak(\delta)).
  \]
  Since $\sfrak(\delta)$ is integrable 
  (seen in the proof of \cref{Lem: asympt LG time-change})
  we may apply the optional stopping theorem to deduce that 
  \begin{multline*}
    \Eb^{Q_L}\big[\mathds{1}_{\{\Sigma'(\delta)\ge L/2\}} (f(S(\delta)) - f(\Sigma'(\delta)))\big] \\
    = 
    c_3 \Eb^{Q_L}\Big[ \mathds{1}_{\{\Sigma'(\delta)\ge L/2\}} \int_0^{\sfrak(\delta)} \mathrm{d}u \int_0^{\infty} \big( f(\Sigma'(\delta)-\Sigma(u)+z)-f(\Sigma'(\delta)-\Sigma(u))\big)\frac{\mathrm{d}z}{z^{3/2}} \Big].  
  \end{multline*}
  Through the change of variables $u=\sfrak(\delta)v$, the above expression becomes
  \begin{multline*}
    \Eb^{Q_L}\big[\mathds{1}_{\{\Sigma'(\delta)\ge L/2\}} (f(S(\delta)) - f(\Sigma'(\delta)))\big] 
    \\ 
    = c_3 \Eb^{Q_L}\Big[ \sfrak(\delta) \mathds{1}_{\{\Sigma'(\delta)\ge L/2\}}  \int_0^{1} \mathrm{d}v \int_0^{\infty} \big( f(\Sigma'(\delta)-\Sigma(\sfrak(\delta)v)+z)-f(\Sigma'(\delta)-\Sigma(\sfrak(\delta)v))\big)\frac{\mathrm{d}z}{z^{3/2}} \Big].  
  \end{multline*}
  We now claim that, as $\delta\to 0$, this is of order 
  \begin{equation} \label{eq: asympt first term generator}
    \Eb^{Q_L}\big[\mathds{1}_{\{\Sigma'(\delta)\ge L/2\}} (f(S(\delta)) - f(\Sigma'(\delta)))\big]
    \sim
    c_3 \Eb^{Q_L}[\sfrak(\delta)] \int_0^\infty \big(f(L+z)-f(L)\big) \frac{\mathrm{d}z}{z^{3/2}}.  
  \end{equation}
  Assuming this and using \cref{Lem: asympt LG time-change}, we end up with 
  \[
    \frac{1}{\delta} \Eb^{Q_L}\big[\mathds{1}_{\{\Sigma'(\delta)\ge L/2\}}(f(S(\delta)) - f(\Sigma'(\delta)))\big]
    \rightarrow 
    \frac{c_1+c_2}{L} \int_0^\infty \big(f(L+z)-f(L)\big) \frac{\mathrm{d}z}{z^{3/2}},
    \quad \text{as } \delta\to 0, 
  \]
  therefore providing the first term of \eqref{eq: S generator}.

  To conclude, we then need to prove \eqref{eq: asympt first term generator}. The difference between the left and right hand sides of \eqref{eq: asympt first term generator} is $c_3$ times
  \[
    A(\delta) := \Eb^{Q_L}\Big[\sfrak(\delta)\int_0^1 \int_0^\infty L_f(\delta, v,z) \frac{\mathrm{d}z}{z^{3/2}} \mathrm{d}v\Big], 
  \]
  where 
  \[
    L_f(\delta,v,z) 
    :=
    \mathds{1}_{\{\Sigma'(\delta)\ge L/2\}} \big(f(\Sigma'(\delta)-\Sigma(\sfrak(\delta)v)+z)-f(\Sigma'(\delta)-\Sigma(\sfrak(\delta)v))\big)+ f(L)-f(L+z).
  \]
  We want to prove that $A(\delta)= o(\delta)$ as $\delta\to 0$. The first step is to use the same scaling arguments as in the proof of \cref{Lem: asympt LG time-change} for $\Xi =W\circ\tfrak$ and $\Xi' = W'\circ\uptau$. Recalling the notation in \eqref{eq: scaling T and V}, we have
  \[
    A(\delta)
    =
    \frac{1}{L} \int_0^L \mathrm{d}x \cdot \Eb\Big[ T(x+V_1(\delta), L-x+V_2(\delta)) \cdot \int_0^1 \int_0^\infty L_f(\delta,v,z,x)\frac{\mathrm{d}z}{z^{3/2}} \mathrm{d}v \Big],  
  \]
  where $L_f(\delta,v,z,x)$ is
  \begin{multline*}
    L_f(\delta,v,z,x)
    :=
    \mathds{1}_{\{L+\Sigma'(\delta)\ge L/2\}} \Big( f\big(L+\Sigma'(\delta)-\Sigma(T(x+V_1(\delta), L-x+V_2(\delta))v)+z\big) \\
    -f\big(L+\Sigma'(\delta)-\Sigma(T(x+V_1(\delta), L-x+V_2(\delta))v)\big) \Big) + f(L)-f(L+z).
  \end{multline*}
  Note that $\Sigma$ and $\Sigma'$ are now both starting from $0$, under $\Pb$.
  Partition again according to the events $E_1, E_2, E_3$ as in \eqref{eq: partition events} -- we denote by $A_1(\delta), A_2(\delta), A_3(\delta)$ the corresponding restrictions of $A(\delta)$. Let us first consider case $E_1$. In this case, recall that $T(x+V_1(\delta), L-x+V_2(\delta)) = T(x+V_1(\delta), 0)$ Unfolding the scaling relations and performing the change of variables $x=\delta^{1/\alpha} y$ all at once, we arrive after some tedious calculations at
  \[
    A_1(\delta)
    =
    \frac{\delta}{L} \Eb\bigg[ \int_0^{\delta^{-1/\alpha} L\wedge (-V_1(1)) \wedge (\delta^{-1/\alpha}L+V_2(1))} \mathrm{d}y T(y+V_1(1), 0) \int_0^1 \int_0^\infty \tilde{L}_f(\delta,v,z,y)\frac{\mathrm{d}z}{z^{3/2}} \mathrm{d}v \bigg],  
  \]
  with
  \begin{multline*}
    \tilde{L}_f(\delta,v,z,y)
    =
    \mathds{1}_{\{L+\delta^{1/\alpha}\Sigma'(1)\ge L/2\}} \Big(f\big(L+\delta^{1/\alpha}\Sigma'(1)-\delta^{1/\alpha}\Sigma(T(y+V_1(1), 0))+z \big) \\
    -f\big(L+\delta^{1/\alpha}\Sigma'(1)-\delta^{1/\alpha}\Sigma(T(y+V_1(1), 0))\big) \Big)+ f(L)-f(L+z).
  \end{multline*}
  We now show that the above expectation goes to $0$ as $\delta\to 0$. Plainly, $\tilde{L}_f(\delta,v,z,y) \to 0$ as $\delta\to 0$, for fixed $v,z,y$. 
  Moreover, noticing again that $-\Sigma$ remains positive, we see that on the above indicator, $L+\delta^{1/\alpha}\Sigma'(1)-\delta^{1/\alpha}\Sigma(T(y+V_1(1), 0)) > L/2$.  Thus we may leverage the uniform bound \eqref{eq: bound derivative f} to show that
  for all $(\delta,v,z,y)$,
  \[
    |\tilde{L}_f(\delta,v,z,y)| \le 2C (z\wedge 1).
  \]
  This gives the domination assumption since 
  \[
    \Eb\bigg[\int_0^{-V_1(1)} T(y+V_1(1), 0)  \mathrm{d}y \int_0^1 \int_0^\infty \frac{(z\wedge 1)}{z^{3/2}} \mathrm{d}z\mathrm{d}v\bigg]
    <\infty,
  \]
  as a result of the calculation \eqref{eq: T integrable}. We conclude by dominated convergence that $A_1(\delta)$ is sublinear as $\delta \to 0$. For symmetry reasons, so is $A_2(\delta)$.

  It remains to deal with $E_3$. This case is actually easier since the corresponding term $I_3$ is already sublinear in the proof of \cref{Lem: asympt LG time-change}. Therefore, we can afford to use the bound \eqref{eq: bound derivative f} directly, yielding
  \[
    A_3(\delta)
    \le 
    \frac{C}{L} \Eb\bigg[\int_0^L T(x+\delta^{1/\alpha}V_1(1), L-x+\delta^{1/\alpha}V_2(1)) \mathds{1}_{E_3(x)} \mathrm{d}x\bigg],
  \]
  for some (other) constant $C>0$. With the notation in the proof of \cref{Lem: asympt LG time-change}, this is $A_3(\delta) \le C I_3(L,\delta)$ which is sublinear as a consequence of that proof. This establishes \eqref{eq: asympt first term generator} and our claim in \cref{thm: main cone end}.
\end{proof}

%
%

\section{The growth-fragmentation process} \label{sec: proof gf}

We now establish our main theorem (\cref{thm: main}). The general strategy is similar to that of \cite[Theorem 3.3]{AD}, which roughly corresponds to the case $\theta=\pi$ (see also \cite{LGR,DS}). {We actually start by proving \cref{prop: law uniform exploration}, and then deduce \cref{thm: main} from the law of the uniform exploration.} Again we restrict to $\theta=\frac{2\pi}{3}$, and we drop the subscript $\theta$ for ease of notation.


\subsection{Law of the uniform exploration}\label{sec: law unif}
Our first result describes the branch of the growth-fragmentation $\bf Z$ targeting a uniform time in the excursion biased by its duration, {as stated in \cref{prop: law uniform exploration}}. Equivalently, it describes the law of the branch towards a point sampled from the Liouville measure in the unit-boundary quantum disc biased by its area. This is reminiscent of the last item of \cite[Proposition 6.6]{BBCK}, which states a scaling limit result for the exploration towards a uniformly chosen vertex in the size-biased random planar map. {For $z\in\R_+^2 \setminus\{0\}$, we recall that we introduced in \eqref{eq: def P bar} certain probability measures $\overline{P}^z$ sampling a time $T$ together with the excursion $e$.} 
Recall also from \eqref{eq: cal Z and Z^t} the definition of the process $(Z^t(a))_{a<\varsigma^t}$, $t\in(0,\zeta)$ associated with $e$. 
For $t\in(0,\zeta)$ such that $\varsigma^t >a$, we also let $e_a^{(t)}$ be the subpath of $e$ between $g_t(a)$ and $d_t(a)$ (recall \eqref{eq: def d_t and g_t}). \medskip

{
  We start with a general lemma providing a key formula for any functional of $(Z^T, \zeta(e_a^{(T)}))$ under the infinite measure $\overline{\nlg}$.
  \begin{Lem} \label{lem:key_joint_nlg}
   Let $H$ be a bounded continuous functional on the space of finite càdlàg paths, and $F$ a non-negative measurable function defined on $\R_+$.
   Then for all $a>0$,
\[
  \overline{\nlg}\big(F(\zeta(e_a^{(T)}))H(Z^T(b),\, b \in [0,a])\mathds{1}_{\{a<\varsigma^{T}\}}\big)
  =
  \overline{\nlg}\big(F(\zeta) \tilde{h}(-a,e(0))\big),
\]
where
\[
    \tilde{h}(-a,(x,y)) := \Eb_{x,y} [  H\left(S(a-b),\, b\in [0,a]\right)  ].
\]
  \end{Lem}
  \begin{proof}
    With the notation of the lemma, we have by the Bismut description of $\overline{\nlg}$ (\cref{thm: Bismut}) and the Markov property of the process $Y$ defined in \eqref{eq: def Y(a) and S(a)},
    \begin{align*}
      \overline{\nlg}(F(\zeta(e_a^{(T)})) & H(Z^T(b),\, b \in [0,a])\mathds{1}_{\{a<\varsigma^{T}\}}) \\
      &=
      \overline{c}\int_a^{\infty} \mathrm{d}A \cdot \Eb[F(\uptau(A-a)+\tfrak(\sfrak(A-a))) H(S(A-b), b\in[0,a])] \\ 
      &=
      \overline{c}\int_a^{\infty} \mathrm{d}A \cdot \Eb\big[F(\uptau(A-a)+\tfrak(\sfrak(A-a)))\tilde{h}(-a,Y(A-a)) \big] \\
      &=
      \overline{c}\int_0^{\infty} \mathrm{d}A \cdot \Eb\big[F(\uptau(A)+\tfrak(\sfrak(A)))\tilde{h}(-a,Y(A)) \big],
    \end{align*}
    where 
    \[
      \tilde{h}(-a,(x,y)) := \Eb_{x,y} [  H\left(S(a-b),\, b\in [0,a]\right)  ].
    \]
    Using Bismut's description again, we see that
  \[
    \overline{\nlg}(F(\zeta(e_a^{(T)})) H(Z^T(b),\, b \in [0,a])\mathds{1}_{\{a<\varsigma^{T}\}})
    =
    \overline{\nlg}\big(F(\zeta)\tilde{h}(-a,e(0))\big),
  \]
  which is our claim.
  \end{proof}
}

The main result of this section is the following description of the law of the uniform exploration.
\begin{Prop} \label{prop: law uniform exploration}
  Let $z\in\R_+^2 \setminus \{0\}$. Under $\overline{P}^z$, the process $Z^T$ is a spectrally negative $\frac32$--stable process conditioned to be absorbed continuously at $0$ started at $\lVert z\rVert_1$. More precisely, it has the law of the process $\xi^{\searrow}$ described in \cref{sec: pssMp and Levy}, with $\alpha = \frac32$ and $c_{\Lambda}=2$. 
\end{Prop}

\noindent {Note that this proves \cref{prop: law unif intro} (and \cref{cor: law unif intro}).}

\begin{proof}
{We prove the statement for $z\in (\R_+^*)^2$, noting that the claim then easily follows for $z\in \partial \R_+^2 \setminus \{0\}$ by taking limits, using the convergence in \cref{prop: convergence measures}.}
   Let $a\ge0$ and $H$ be a bounded continuous functional on the space of finite càdlàg paths. By \cref{lem:key_joint_nlg},
  \[
    \overline{\nlg}(H(Z^T(b),\, b \in [0,a])\mathds{1}_{\{a<\varsigma^{T}\}})
    =
    \overline{\nlg}\big(\tilde{h}(-a,e(0))\big)
    =
    \nlg\big(\tilde{h}(-a,e(0)) \zeta(e)\big).
  \]
  Disintegrating the right-hand side over $e(0)$, we get
  \begin{multline*}
    \overline{\nlg}(H(Z^T(b),\, b \in [0,a])\mathds{1}_{\{a<\varsigma^{T}\}})
    =
    \int_{\R_+^2} \frac{\mathrm{d}l \mathrm{d}r}{(l+r)^{5/2}}  \tilde{h}(-a,(l,r)) \Eb^{P^{(l,r)}}[\zeta] \\
    =
    \frac{1}{\sqrt{3}} \int_{\R_+^2} \frac{\mathrm{d}l \mathrm{d}r}{(l+r)^{1/2}}  \tilde{h}(-a,(l,r))
    =
    \frac{1}{\sqrt{3}} \int_{0}^{\infty} L^{1/2} \mathrm{d}L \int_0^L \frac{\mathrm{d}r}{L} \Eb_{L-r,r} [  H\left(S(a-b),\, b\in [0,a]\right) ],
  \end{multline*}
  where we used in the second equality that $\Eb^{P^z}[\zeta] = \frac{\lVert z \rVert_1^2}{\sqrt{3}}$ for all $z\in\R_+^2 \setminus \{0\}$. In other words,
  \begin{equation} \label{eq:key_Z^T}
    \overline{\nlg}(H(Z^T(b),\, b \in [0,a])\mathds{1}_{\{a<\varsigma^{T}\}})
    =
    \frac{1}{\sqrt{3}} \int_{0}^{\infty} L^{1/2} \mathrm{d}L \cdot \Eb^{Q_L} [  H\left(S(a-b),\, b\in [0,a]\right) ].
  \end{equation}

{
  The functional on the right-hand side involves the time-reversal of $S$, \textit{i.e.} the time-reversal of a stable process conditioned to stay positive (according to \cref{thm: main cone end}). Now recall from \cref{sec: pssMp and Levy} that the spectrally positive $\frac32$--stable process conditioned to remain positive can be written as a Doob $h^{\uparrow}$--transform of the spectrally positive $\frac32$--stable process killed when entering the negative half-line, with harmonic function $h^\uparrow(x)=x$. As a consequence, 
  \begin{multline}
    \label{eq:unif_expl_h_transform}
    \overline{\nlg}(H(Z^T(b),\, b \in [0,a])\mathds{1}_{\{a<\varsigma^{T}\}}) \\
    = \frac{1}{\sqrt{3}}\int_{0}^{\infty} \mathrm{d}L \cdot \Eb^{Q_L} \left[ \frac{S_+(a)}{L^{1/2}} H\left(S_+(a-b),\, b\in [0,a]\right) \mathds{1}_{\{\forall b\in[0,a], \; S_+(b)>0\}} \right],
  \end{multline}
where under $Q_{L}$, $S_+$ is the spectrally positive $\frac32$--stable process starting at $L$. 
Now, we use duality with respect to the Lebesgue measure of the Lévy process $S_+$ (see \cite[Section II.1]{Ber}). 
Let $S_-$ be the {spectrally \emph{negative} $\frac32$--stable process (with law $-S_+$) which under $Q_{\ell}$ starts from $\ell\in\R$.} Duality entails 
  \begin{align*}
    \int_{0}^{\infty} \mathrm{d}L & \cdot \Eb^{Q_L} \left[ \frac{S_+(a)}{L^{1/2}} H\left(S_+(a-b),\, b\in [0,a]\right) \mathds{1}_{\{\forall b\in[0,a], \; S_+(b)>0\}} \right] \notag \\
    &= \int_{-\infty}^{\infty} \mathrm{d}\ell \cdot \Eb^{Q_{\ell}} \left[ \frac{\ell}{S_-(a)^{1/2}} H\left(S_-(b),\, b\in [0,a]\right) \mathds{1}_{\{\forall b\in[0,a], \; S_-(b)> 0\}}  \right]  \notag \\
    &= \int_{0}^{\infty} \mathrm{d}\ell \cdot \Eb^{Q_{\ell}} \left[ \frac{\ell}{S_-(a)^{1/2}} H\left(S_-(b),\, b\in [0,a]\right) \mathds{1}_{\{\forall b\in[0,a], \; S_-(b)>0 \}}  \right]. 
  \end{align*}
 Going back to \eqref{eq:unif_expl_h_transform}, we obtained
  \begin{multline*}
    \overline{\nlg}(H(Z^T(b),\, b \in [0,a])\mathds{1}_{\{a<\varsigma^{T}\}})  \\
    = \frac{1}{\sqrt{3}} \int_{0}^{\infty} \mathrm{d}\ell \cdot \Eb^{Q_{\ell}} \left[ \frac{\ell}{S_-(a)^{1/2}} H\left(S_-(b),\, b\in [0,a]\right) \mathds{1}_{\{\forall b\in[0,a], \; S_-(b)>0 \} }  \right].
  \end{multline*}
Disintegrating the measure $\overline{\nlg}$ over $e(0)$ as in \eqref{eq: sqrt(8/3) law displacement}, we get that 
\begin{multline*}
  \frac{1}{\sqrt{3}}\int_{0}^\infty \frac{\mathrm{d}\ell}{\ell^{3/2}} \int_{z\in\R_+^2 \setminus \{0\}: \lVert z\rVert_1 =1} \mathrm{d}z \Eb^{\overline{P}^{\ell z}}\big[H(Z^T(b),\, b \in [0,a])\mathds{1}_{\{a<\varsigma^T\}}\big] \\
  = \frac{1}{\sqrt{3}} \int_{0}^{\infty} \mathrm{d}\ell \cdot \Eb^{Q_{\ell}} \left[ \frac{\ell}{S_-(a)^{1/2}} H\left(S_-(b),\, b\in [0,a]\right) \mathds{1}_{\{\forall b\in[0,a], \; S_-(b)>0 \} }  \right].
\end{multline*}
The previous arguments extend if we multiply $H$ by an arbitrary function $f$ of $Z^T(0) = \lVert e(0)\rVert_1$ and $g$ of the angular part $\frac{e(0)}{\lVert e(0)\rVert_1}$. By independence between $S$ and $\frac{Y}{S}$ (\cref{prop: target invariance Q_L strong}), the previous identity yields
\begin{multline*}
  \int_{0}^{\infty} \frac{\mathrm{d}\ell}{\ell^{3/2}} \ell^2 f(\ell)
  \int_{z\in\R^2_+: \lVert z\rVert_1 =1} \mathrm{d}z g(z) \Eb^{\overline{P}^{\ell z}}\big[H(Z^T(b),\, b \in [0,a])\mathds{1}_{\{a<\varsigma^T\}}\big] \\
  = \int_{z\in\R^2_+: \lVert z\rVert_1 =1} g(z) \mathrm{d}z \int_{0}^{\infty} \mathrm{d}\ell \cdot  f(\ell) \Eb^{Q_\ell} \bigg[ \frac{\ell}{S_-(a)^{1/2}} H\left(S_-(b),\, b\in [0,a]\right) \mathds{1}_{\{\forall b\in[0,a], \; S_-(b)>0 \} }  \bigg].
\end{multline*}
Here we emphasise that the term $\frac{\ell^2}{\sqrt{3}}$ on the left-hand side comes from the definition of $\overline{P}^{\ell z}$, see \eqref{eq: def P bar}.
This equality holds for all non-negative measurable function $f$, and a continuity argument brings that for all $z\in\R_+^2$ with $\lVert z \rVert_1 = \ell$,
\[
  \Eb^{\overline{P}^z}\big[H(Z^T(b),\, b \in [0,a])\mathds{1}_{\{a<\varsigma^T\}}\big]
  =
  \Eb^{Q_\ell} \bigg[ \frac{\ell^{1/2}}{S_-(a)^{1/2}} H(S_-(b),\, b\in [0,a]) \mathds{1}_{\{\forall b\in[0,a], \; S_-(b)>0 \} }  \bigg].
\]
By the last paragraph of \cref{sec: pssMp and Levy}, we check that the above $h$--transform coincides with that of a spectrally negative $\frac32$--stable Lévy process conditioned to be absorbed continuously at $0$, as we claimed.
}
\end{proof}

\subsection{The law of the locally largest fragment \texorpdfstring{$Z^*$}{Z*}} \label{sec: law loc largest}

This section is devoted to deriving the law of a specific branch of the process $\Zbf$ of \eqref{eq: GF cones}. We show that this branch has the same law as the so-called \emph{locally largest fragment} $X^{3/2}$ of the growth-fragmentation $\Xbf^{3/2}$ described explicitly at the end of \cref{sec: GF process}. 
Specifically, recall from \eqref{eq: cal Z and Z^t} the notation $Z^t$ for the branch targeted at $t$. Under $\nlg$ (or $P^z$ for $z\in\R_+^2 \setminus \{0\}$, {including the boundary}) we denote by $t^*\in(0,\zeta)$ the unique time $t$ such that, for all $a\in(0,\varsigma^{t})$,
\[
  Z^t(a) > \frac12 Z^t(a^-).
\]
In other words, at each (local) time $a$ when the excursion straddling $t^*$ at time $a$ splits into two excursions, the branch towards time $t^*$ is the one following the largest excursion, in terms of the $1$--norm of their displacements. It can be shown that (under either measure) $t^*$ is well-defined and unique outside of a negligible set, following the topological arguments presented in \cite[Section 2.5]{AD}. Without loss of generality, we henceforth implicitly restrict to this event in all the arguments below.
We define $Z^*:=Z^{t^*}$ and $\varsigma^*:= \varsigma^{t^*}$. We start with a technical lemma, {which is \cite[Lemma 18]{LGR}}.

  \begin{Lem} \label{eq: lem LGR}
   {Let $\xi^\searrow$ be the spectrally negative $\frac32$--stable process conditioned to be absorbed at $0$, as in \cref{sec: pssMp and Levy} with $c_\Lambda = 2$.} Denote by $\widetilde{\xi}$ its underlying Lamperti transform. Then the process 
    \[
      M(a) :=
      \mathrm{e}^{-2 \widetilde{\xi}(a)} \mathds{1}_{\{\forall b\in[0,a], \; \Delta \widetilde{\xi}(b)> -\log(2) \}},
      \quad a\ge 0,
    \]
    is a martingale with respect to the natural filtration associated with $\xi$. In addition, under the corresponding change of measure, the process $\xi$ is a Lévy process with Laplace exponent 
    \begin{equation} \label{eq: Laplace lem LGR}
      \Psi^*(q) 
      :=
      -\frac{16}{3}q + 2\int_{-\log(2)}^0 (\mathrm{e}^{qy}-1-q(\mathrm{e}^y-1)) \mathrm{e}^{-3y/2}(1-\mathrm{e}^y)^{-5/2}\mathrm{d}y,
      \quad
      q\in \R.
    \end{equation}
  \end{Lem}

  \begin{proof}[Proof of \cref{eq: lem LGR}]
    {
    Recall that the Laplace exponent of $\widetilde{\xi}$ is given according to \eqref{eq: Laplace xi searrow} by
    \[
      \widetilde{\Psi}(q) 
      :=
      2 \int_{-\infty}^0 (\mathrm{e}^{qy}-1-q(\mathrm{e}^y-1)) \frac{\mathrm{e}^{y/2}\mathrm{d}y}{(1-\mathrm{e}^{y})^{5/2}},
      \quad q\ge 0.
    \]
    The claim is therefore \cite[Lemma 18]{LGR}, tracing the normalising constants.
    }
  \end{proof}

We may now derive explicitly the law of the process $Z^*$.

  \begin{Thm} \label{thm: law loc largest}
    Let $z\in \R_+^2 \setminus \{0\}$ and denote $\ell := \lVert z \rVert_1$. Under $P^z$, the process $(Z^*(a), 0\le a<\varsigma^*)$ is a positive self-similar Markov process with index $\frac32$ starting from $\ell$. Its Lamperti representation is 
    \[
      Z^*(a)
      :=
      \ell \exp(\xi^*(\tau^*(\ell^{-3/2}a))),
    \]
    where $\xi^*$ is a Lévy process with Laplace exponent \eqref{eq: Laplace lem LGR}
    and $\tau^*$ is the Lamperti time change
    \[
      \tau^*(t) 
      :=\inf\{s\ge 0, \int_0^s \mathrm{e}^{\tfrac32\xi^*(u)}\mathrm{d}u>t\}, \quad t\ge 0.
    \]
  \end{Thm}

\noindent The rest of this section is devoted to the proof of \cref{thm: law loc largest}. 
{
\begin{proof}
We will first relate the processes $Z^*$ and $Z^T$. Let $H$ be a bounded continuous non-negative function defined on the space of finite càdlàg paths, and $a\ge 0$. 
We first observe that for $\nlg$--almost every excursion $e$, 
\begin{equation} \label{eq: loc largest integration}
  H(Z^*(b),\, b \in [0,a])\mathds{1}_{\{a<\varsigma^*\}} = \int_{0}^{\zeta(e)} H(Z^{t}(b), b\in[0,a]) \mathds{1}_{\{\varsigma^t > a\}\cap \{
    \forall b\in[0,a], \; Z^t(b)>\frac12  Z^t(b^-)
  \}} \frac{\mathrm{d}t}{\zeta(e_a^{(t)})},
\end{equation}
where $e_a^{(t)}$ is the subpath of $e$ between $g_t(a)$ and $d_t(a)$ (recall \eqref{eq: def d_t and g_t}). 
Taking expectations in \eqref{eq: loc largest integration} under $\nlg$, we have
\[
  \nlg(H(Z^*(b),\, b \in [0,a])\mathds{1}_{\{a<\varsigma^*\}})
  =
  \overline{\nlg}\bigg( \frac{1}{\zeta(e_a^{(T)})}H(Z^T(b),\, b \in [0,a]) \mathds{1}_{\{a<\varsigma^T\} \cap \{
    \forall b\in[0,a], \; Z^T(b)>\frac12  Z^T(b^-)
  \}}\bigg).
\]
We may now apply \cref{lem:key_joint_nlg} to obtain
\begin{equation} \label{eq:Z*_key}
  \nlg(H(Z^*(b),\, b \in [0,a])\mathds{1}_{\{a<\varsigma^*\}})
  =
  \overline{\nlg}(\tilde{h}(-a,e(0))/\zeta)
  =
  \nlg(\tilde{h}(-a,e(0))),
\end{equation}
with 
\[
  \tilde{h}(-a,(x,y)) = \Eb_{x,y}\big[H(S(a-b),\, b \in [0,a]) \mathds{1}_{\{\forall b\in[0,a], \; S(b)>\frac12 S(b^-)\}}\big].
\]
On the other hand, note that the scaling and conditional independence property of $\zeta$ under $\nlg$ stated in \cref{prop: joint law displacement/duration}, together with the fact that $\Eb^{P^z}[\zeta] = \frac{\lVert z \rVert_1^2}{\sqrt{3}}$ entail that
\[
  \overline{\nlg}\Big(\tilde{h}(-a,e(0)) \lVert e(0) \rVert_1^{-2}\Big)
  =
  \nlg\Big(\tilde{h}(-a,e(0)) \lVert e(0) \rVert_1^{-2} \zeta\Big)
  =
  \frac{1}{\sqrt{3}} \nlg(\tilde{h}(-a,e(0))),
\]
which is nothing but the right-hand side of \cref{eq:Z*_key}. Therefore 
\[
  \nlg(H(Z^*(b),\, b \in [0,a])\mathds{1}_{\{a<\varsigma^*\}})
  =
  \sqrt{3} \cdot \overline{\nlg}\Big(\tilde{h}(-a,e(0)) \lVert e(0) \rVert_1^{-2}\Big).
\]
We apply again \cref{lem:key_joint_nlg}. The above display boils down to
\begin{multline*}
  \nlg(H(Z^*(b),\, b \in [0,a])\mathds{1}_{\{a<\varsigma^*\}}) \\
  =
  \sqrt{3} \cdot \overline{\nlg}\Big( (Z^T(a))^{-2} H(Z^T(b),\, b \in [0,a])\mathds{1}_{\{a<\varsigma^{T}\} \cap \{\forall b\in[0,a], \; Z^T(b)>\frac12  Z^T(b^-) \}}\Big).
\end{multline*}
It remains to disintegrate over $\lVert e(0) \rVert_1$. Using \cref{prop: joint law displacement/duration} and again the fact that $\Eb^{P^z}[\zeta] = \frac{\lVert z \rVert_1^2}{\sqrt{3}}$, we end up with
\begin{align*}
  &\int_0^\infty \frac{\mathrm{d}L}{L^{3/2}} \Eb^{Q_L}[H(Z^*(b),\, b \in [0,a])\mathds{1}_{\{a<\varsigma^*\}}] \\
  &=
  \int_0^\infty \frac{\mathrm{d}L}{L^{3/2}} L^2 \int_{z\in\R^2_+: \lVert z\rVert_1 =1} \mathrm{d}z \Eb^{\overline{P}^{Lz}}\Big[(Z^T(a))^{-2} H(Z^T(b),\, b \in [0,a])\mathds{1}_{\{a<\varsigma^{T}\} \cap \{\forall b\in[0,a], \; Z^T(b)>\frac12  Z^T(b^-) \}}\Big] \\
  &=
  \int_0^\infty \frac{\mathrm{d}L}{L^{3/2}} L^2 \Eb^{\overline{P}^{L}}\Big[(Z^T(a))^{-2} H(Z^T(b),\, b \in [0,a])\mathds{1}_{\{a<\varsigma^{T}\} \cap \{\forall b\in[0,a], \; Z^T(b)>\frac12  Z^T(b^-) \}}\Big],
\end{align*}
where we wrote $\overline{P}^{L} := \overline{P}^{L\cdot (0,1)}$ say, since according to \cref{prop: law uniform exploration} the law of $Z^T$ under $\overline{P}^{Lz}$ {is the same for all $z\in\R_+^2$ such that $\lVert z\rVert_1 = 1$.}
Note the extra factor $L^2$ coming from the definition of $\overline{P}^{Lz}$ in \eqref{eq: def P bar}. Multiplying $H$ by a function $f$ of $L= Z^*(0) = Z^T(0)$ and using a continuity argument, we obtain that for all $L>0$,
\begin{multline*}
  \Eb^{Q_L}[H(Z^*(b),\, b \in [0,a])\mathds{1}_{\{a<\varsigma^*\}}] \\
  =
  L^2 \Eb^{\overline{P}^{L}}\Big[(Z^T(a))^{-2} H(Z^T(b),\, b \in [0,a])\mathds{1}_{\{a<\varsigma^{T}\} \cap \{\forall b\in[0,a], \; Z^T(b)>\frac12  Z^T(b^-) \}}\Big].
\end{multline*}
The above chain of arguments extends if we add in a functional of the angular part $\frac{e(0)}{\lVert e(0) \rVert_1}$, yielding for all $z\in (\R_+^*)^2$ with $\lVert z\rVert_1 = L>0$,
\begin{multline} \label{eq: law Z* h-transform}
  \Eb^{P^ z}[H(Z^*(b),\, b \in [0,a])\mathds{1}_{\{a<\varsigma^*\}}] \\
  =
  L^2 \Eb^{\overline{P}^{L}}\Big[(Z^T(a))^{-2} H(Z^T(b),\, b \in [0,a])\mathds{1}_{\{a<\varsigma^{T}\} \cap \{\forall b\in[0,a], \; Z^T(b)>\frac12  Z^T(b^-) \}}\Big].
\end{multline}

This essentially describes the law of $Z^*$ as a Doob $h$--transform of the process $Z^T$. In particular, it gives that $Z^*$ is a positive self-similar Markov process with index $\frac32$. To conclude, it remains to work out the Lamperti exponent of $Z^*$. 
To do so, we use \cref{prop: law uniform exploration}, which states that under $\overline{P}^L$, $Z^T$ is a spectrally negative $\frac32$--stable process conditioned to be absorbed at $0$. More precisely, we can write it as the Lamperti transform of the Lévy process $\widetilde{\xi}$ in \cref{eq: lem LGR}. Now equation \eqref{eq: law Z* h-transform} rewrites
  \begin{multline} \label{eq: loc largest xi^0}
    \Eb^{P^z}\big[H(Z^*(b),\, b \in [0,a])\mathds{1}_{\{a<\varsigma^*\}}\big] \\
    =
    \Eb \Big[ \mathrm{e}^{-2 \widetilde{\xi}(\tau(a))} H(\ell \exp(\widetilde{\xi}(\tau(b))),\, b\in [0,a]) \mathds{1}_{\{\forall b\in[0,\tau(a)], \; \Delta \widetilde{\xi}(b)> -\log(2) \} }  \Big],
  \end{multline}
  where $\tau$ is the Lamperti time-change \eqref{eq: Lamperti time change}.
We then short-circuit the derivation of the Lamperti exponent of $Z^*$ using \cref{eq: lem LGR} as an input. 
We first write \eqref{eq: loc largest xi^0} as
  \[
    \Eb^{P^z}\big[H(Z^*(b),\, b \in [0,a])\mathds{1}_{\{a<\varsigma^*\}}\big] 
    =
    \Eb\big[ M(\tau(a)) H(\ell \exp(\widetilde{\xi}(\tau(b))),\, b\in [0,a])  \big].
  \]
  Now for any $c>0$, the optional stopping theorem entails that 
  \[
    \Eb\big[ M(\tau(a)) H(\ell \exp(\widetilde{\xi}(\tau(b))),\, b\in [0,a]) \mathds{1}_{\{\tau(a)<c\}} \big]
    =
    \Eb \big[ M(c) H(\ell \exp(\widetilde{\xi}(\tau(b))),\, b\in [0,a]) \mathds{1}_{\{\tau(a)<c\}} \big].
  \]
  By \cref{eq: lem LGR}, the right-hand side above boils down to
  \[
    \Eb \big[ H(\ell \exp(\xi^*(\tau^*(b))),\, b\in [0,a]) \mathds{1}_{\{\tau^*(a)<c\}} \big],
  \]
  where $\xi^*$ is the Lévy process with Laplace exponent \eqref{eq: Laplace lem LGR} and $\tau^*$ the associated Lamperti time change with exponent $3/2$. Finally, we take $c\to\infty$ to obtain
  \[
    \Eb^{P^z}\big[H(Z^*(b),\, b \in [0,a])\mathds{1}_{\{a<\varsigma^*\}}\big] 
    =
    \Eb^{Q_\ell} \big[ H(\ell \exp(\xi^*(\tau^*(b))),\, b\in [0,a]) \big],
  \]
  which is precisely our claim when $z\in \R_+^{*2}$. For $z\in\partial \R_+^2 \setminus \{0\}$, the statement readily follows from the convergence of measures in \cref{prop: convergence measures}.  
\end{proof}
}


\subsection{Proof of \texorpdfstring{\cref{thm: main}}{Theorem \ref*{thm: main}}}
\label{sec:proof of main thm}
In this section, we prove our main result on the growth-fragmentation process. Recall that $\mathbf{X}^{3/2}$ is the growth-fragmentation process introduced in \cref{sec: GF process}.
\begin{Thm} \label{thm: GF process}
  Under $P^z$, the process $\mathbf{Z}$ has the law of the growth-fragmentation process $\mathbf{X}^{3/2}$.
\end{Thm}

\begin{proof}
  The heart of \cref{thm: GF process} has already been proved in \cref{sec: law loc largest}, where we constructed a process $Z^*$ as the driving process of $\mathbf{X}^{3/2}$.
  There are two claims that remain to be proved. Both are adapted from \cite{AD}, so we feel free to only sketch the arguments.

  First, we need to show that almost surely, every fragment in $\mathbf{Z}$ can be found in the lineage of $Z^*$. 
  Note that we can define the cell system driven by $Z^*$ as in \cref{sec: GF process}, where now each cell in the genealogy of $Z^*$ corresponds to a unique collection of decreasing intervals $\{(g_u(\varsigma^u-c),d_u(\varsigma^u-c)); c\in (b^u,\varsigma^u)\}$ for some $u$ and some $0\le b^u\le \varsigma^u$.  For each $b\ge 0$ we can consider the set of cells of this form with $b^u\le b\le \varsigma^u$, and denote by $(\overline{\mathbf{Z}}^*(b), b\ge 0)$ the \emph{enhanced} process that records 
  the sub-excursions $e^{(u)}_{b}$, as defined after equation \eqref{eq: loc largest integration}, corresponding to such $u$. 

From now on, we argue on almost every realisation $e$ under $P^z$, and fix $t\in (0,\zeta)$ and $0\le a < \varsigma^t$. 
  Let 
  \[
    \mathcal{A}
    :=
    \big\{b\in  [0,a], \; e^{(t)}_{b} \in \overline{\mathbf{Z}}^*(b) \big\},
  \]
  where $e^{(t)}_{b}$ was defined after equation \eqref{eq: loc largest integration}. On the one hand, we claim that $\mathcal{A}$ is \emph{open}. In fact, if $b\in\mathcal{A}$ with $b<a$, we can carry on the exploration for a bit by following the locally largest evolution inside the sub-excursion $e^{(t)}_{b}$. Since the locally largest excursions are always in $\overline{\mathbf{Z}}^*$, this proves that $\mathcal{A}$ is open. On the other hand, we claim that $\mathcal{A}$ is also \emph{closed}. Indeed, take an increasing sequence $(b_n, n\in \N)$ in $\mathcal{A}$ that converges to some $b_\infty$. Then by taking $n$ large enough, one can see that the excursions $e_b^{(t)}$, $b_n\le b<b_\infty$, are following the locally largest evolution inside $e_{b_n}^{(t)}$. This proves that $b_\infty\in\mathcal{A}$ and $\mathcal{A}$ is closed. Since $\mathcal{A}$ contains $0$, we get that $\mathcal{A} = [0,a]$ by connectedness. The previous argument holds almost surely for all $a$ and {all rational $t$, and so by \eqref{eq: GF cones}} this proves that every fragment in $\mathbf{Z}$ can be found in the lineage of $Z^*$. 

  Secondly, we need to prove that the children of $Z^*$ are conditionally independent and have the same distribution as $Z^*$ started from their respective sizes. Here we argue under $\nlg$ (the claim then follows from the usual disintegration argument). Fix $a>0$ and denote by $(e_i^a)_{i\ge 1}$ the sub-excursions created by the jumps of $Z^*$ before time $a$, ranked by descending order of the $1$--norm of their displacements $z_i^a := e_i^a(0) $. For $f_i$ and $g_i$, $i\ge 1$, non-negative measurable functions, and any $n\ge 1$, we prove the equality:
  \begin{equation} \label{eq: offspring conditional idp}
    \nlg\bigg( \mathds{1}_{\{a<\varsigma^*\}} \prod_{i=1}^n f_i(e_i^a) g_i(z_i^a)\bigg)  
    =
    \nlg\bigg( \mathds{1}_{\{a<\varsigma^*\}} \prod_{i=1}^n \Eb^{P^{z_i^a}}[f_i] g_i(z_i^a)\bigg). 
  \end{equation}
  This would prove the claim on the law of the children of $Z^*$, since the law of the process $\mathbf{Z}$ under $P^z$ depends only on $\lVert z\rVert_1$. 
  To prove the above equality, we first write that almost surely, 
  \[
    \mathds{1}_{\{a<\varsigma^*\}} \prod_{i=1}^n f_i(e_i^a) g_i(z_i^a)
    =
    \int_0^{\zeta(e)} \mathds{1}_{{\{e_a^{(t)}=e_a^{(t^*)}\}}} \mathds{1}_{\{a<\varsigma^{t}\}} \prod_{i=1}^n f_i(e_i^{a,t}) g_i(z_i^{a,t}) \frac{\mathrm{d}t}{\zeta(e^{(t)}_a)}, 
  \]
  where the $e_i^{a,t}$ and $z_i^{a,t}$ denote the excursions and displacements cut out in the exploration towards $t$ (ranked accordingly). 
  Then the idea is to use Bismut's description of $\nlg$ (\cref{thm: Bismut}). From Bismut's description, we see that the excursions $e_i^{a,t}$ come from the backward or forward cone excursions of $W$ or $W'$ respectively. These are two Poisson point processes with respective intensity measures $\nlg$ and $\ndms$. Call these excursions $\varepsilon_i$, with displacements $z_i$ (again ranked accordingly). By basic properties of Poisson point processes, we obtain that conditioned on the $z_i$'s, these excursions are independent with respective laws $P^{z_i}$. 
  We feel free to skip the details to avoid cumbersome technical work. To summarise, we arrive at
  \begin{multline} \label{eq: Bismut offspring idp}
    \nlg\bigg( \mathds{1}_{\{a<\varsigma^*\}} \prod_{i=1}^n f_i(e_i^a) g_i(z_i^a)\bigg)  \\
    =
    \int_a^\infty \mathrm{d}A \Eb \bigg[ \frac{1}{\uptau(A-a)+\sfrak(A-a)}  \Eb_{Y(A-a)}\bigg[ \mathds{1}_{\{\forall b\in[0,a], \; S((a-b)^-)>\frac12 S(a-b) \}} \prod_{i=1}^n \Eb^{P^{z_i}}[f_i] g_i(z_i) \bigg]\bigg].
  \end{multline}
  Here it is important to note that the event $\{\forall b\in[0,a], \; S((a-b)^-)>\frac12 S(a-b) \}$ is measurable with respect to $S$, so that it factors out in the conditioning. Now one can start from the right-hand side of \eqref{eq: Bismut offspring idp} and apply again Bismut's description backwards. This yields our claim \eqref{eq: offspring conditional idp}.
\end{proof}


\subsection{Convergence of the martingale towards the duration}
\label{sec: few prop of Zbf}
{We conclude by providing a proof of \cref{thm: martingale}, putting forward a distinguished martingale that appears for the process $\bf Z$, and establishing its convergence towards the duration of a cone excursion under $P^z$.} These results were already obtained for $\mathbf{X}^{3/2}$ in \cite{BBCK}, but we reprove them using the coupling with the Brownian cone excursion given by $\bf Z$. By analogy with \cref{sec: GF process}, we define the cell system $(\mathcal{Z}_u, u\in \cal U)$ driven by the locally largest fragment $Z^*$ of $\bf Z$. Introduce $\mathcal{G}_n := \sigma(\mathcal{Z}_u, |u|\le n-1)$, $n\ge 1$. {We stress that the definition of the cell system $(\mathcal{Z}_u, u\in \cal U)$ (and hence $\mathcal{G}_n$) depends on the choice of driving cell process. Here we only chose the locally largest evolution to fix ideas, but the same result would still hold for any other choice.}
\begin{Thm} \label{thm: Mcal convergence}
  Let $z\in \partial \R_+^2 \setminus \{0\}$. Under $P^z$, the process 
  \[
    \Mcal(n) :=
    \frac{1}{\sqrt{3}}\sum_{|u|=n} \mathcal{Z}_u(0)^2, 
    \quad n\ge 1,  
  \]
  is a $(\mathcal{G}_n)$--martingale. Furthermore, it is uniformly integrable and converges $P^z$--almost surely and in $L^1$ to the duration of the excursion.
\end{Thm}
We stress that this limiting law is explicit, as determined in \cref{prop: joint law displacement/duration}. In particular, the constant $\sqrt{3}$ above comes from the fact that for $z\in\R_+^2 \setminus \{0\}$, $\Eb^{P^z}[\zeta] = \lVert z \rVert_1^2/\sqrt{3}$, as can be seen from \cref{prop: joint law displacement/duration} by simple calculations. Since the duration of the excursion under $P^1$ describes the area of a unit-boundary quantum disc, we can also rephrase the above statement as the convergence of $\Mcal$ towards the area of a unit-boundary quantum disc. The proof is inspired by \cite[Proposition 6.19]{DS}.
\begin{proof}
  By scaling (and symmetry between the axes), we may restrict to the case when $z=1$.
  The key observation is to check that, for all $n\ge 1$, 
  \begin{equation} \label{eq: Mcal Lévy}
    \Mcal(n) = \Eb^{P^1}[\zeta \, | \, \mathcal{G}_n], 
    \quad \text{$P^1$--almost surely}.
  \end{equation}
  Indeed, assuming \eqref{eq: Mcal Lévy} holds, Lévy's theorem implies that $\Mcal$ converges a.s.\ and in $L^1$ to $\Eb^{P^1}[\zeta \, | \, \mathcal{G}_\infty]$, where $\mathcal{G}_\infty := \bigcup_{n\ge 0} \mathcal{G}_n$. Since $\zeta$ is $\mathcal{G}_{\infty}$--measurable, \cref{thm: Mcal convergence} follows.

  It remains to prove \eqref{eq: Mcal Lévy}. We only prove it for $n=1$ since the general case then follows by the branching property of $(\mathcal{Z}_u, u\in \cal U)$. To do so, we split the whole $\zeta$ as a sum of durations of all the excursions at generation $1$. More precisely, we let $(e_i, i\ge 1)$ denote the excursions created by the jumps of $Z^*$, ranked by descending order of $\lVert e_i(0) \rVert_1 = \mathcal{Z}_i(0)$. Since the set of times $s\in(0,\zeta)$ not straddled by any of these $e_i$'s is Lebesgue--negligible, we can write $\zeta = \sum_{i\ge 1} \zeta(e_i)$. Taking the conditional expectation with respect to $\mathcal{G}_1$, we get by the conditional independence of the excursions $e_i$ (see equation \eqref{eq: offspring conditional idp}),
  \[
    \Eb^{P^1}[\zeta \, | \, \mathcal{G}_1] = \sum_{i\ge 1} \Eb^{P^{e_i(0)}}[\zeta].
  \]
  From \cref{prop: joint law displacement/duration}, a back-of-the-envelope calculation shows that $\Eb^{P^z}[\zeta]=\frac{\lVert z \rVert_1^2}{\sqrt{3}}$ for all $z\in\R_+^2 \setminus \{0\}$, and so we end up with 
  \[
    \Eb^{P^1}[\zeta \, | \, \mathcal{G}_1] = \frac{1}{\sqrt{3}}\sum_{i\ge 1} \mathcal{Z}_i(0)^2,
  \]
  which is \eqref{eq: Mcal Lévy} for $n=1$.
\end{proof}


\bibliography{biblio}
\bibliographystyle{alpha}

\end{document}